\documentclass[11pt]{amsart}

\usepackage{amsmath,amsfonts, enumerate, color, amscd, amsthm} 

\newtheorem{theorem}{Theorem}[section]

\newtheorem{corollary}[theorem]{Corollary}
\newtheorem{lemma}[theorem]{Lemma}
\newtheorem{proposition}[theorem]{Proposition}

\newtheorem{definition}[theorem]{Definition}
\newtheorem{remark}[theorem]{Remark}

\usepackage[centerlast,small,sc]{caption}
\newcommand{\cref}[1]{Chapter~\ref{#1}}

\usepackage{graphicx}
\usepackage{epstopdf}
\usepackage[scriptsize]{subfigure}
\usepackage{booktabs}
\usepackage{rotating}
\usepackage{listings}

\usepackage{tikz}
\usetikzlibrary{positioning}
\lstdefinestyle{matlab} {
        language=Matlab,
        keywordstyle=\color{blue},
        commentstyle=\color[rgb]{0.13,0.55,0.13}\em,
        stringstyle=\color[rgb]{0.7,0,0} }
\usepackage[pdfpagemode={UseOutlines},bookmarks=true,bookmarksopen=true,
   bookmarksopenlevel=0,bookmarksnumbered=true,hypertexnames=false,
   colorlinks,linkcolor={blue},citecolor={blue},urlcolor={red},
   pdfstartview={FitV},unicode,breaklinks=true]{hyperref}
\usepackage{latexsym}
\usepackage{amssymb}
\usepackage{amsmath}
\usepackage{amscd}
\usepackage{tikz-cd}
\usepackage{verbatim}
\usepackage{indentfirst}

\newcommand\supp{\mathop{\rm supp}}

\newcommand\id{\mathop{\rm id}}

\newcommand{\calg}{$C^*$-algebra}

\newcommand{\norm}[1]{\lVert#1\rVert}
\newcommand{\bignorm}[1]{\big\lVert#1\big\rVert}
\newcommand{\Bignorm}[1]{\Big\lVert#1\Big\rVert}
\newcommand{\angles}[2]{\langle{#1}\,,{#2}\rangle}

\newcommand{\BH}{\mathcal{B}(\mathcal{H})}
\newcommand{\BHn}{\mathcal{B}(\mathcal{H}^{(n)})}

\newcommand{\BK}{\mathcal{B}(\mathcal{K})}

\newcommand{\hilb}{\mathcal{H}}

\newcommand{\name}{dual operator system homomorphism}

\newcommand\omin{\mathop{\rm OMIN}}
\newcommand\omax{\mathop{\rm OMAX}}
\newcommand\mitp{\otimes_{\min}}
\newcommand\mtp{\otimes_{\max}}
\newcommand\ctp{\otimes_{\rm c}}
\newcommand\SPAN{\textup{span}}

\newcommand\cattop{\bf{Top}}
\newcommand\limtop{\underleftarrow{\lim}_{\bf{Top}}}

\newcommand\catcalg{\bf{C}^*}
\newcommand\limcalg{\underrightarrow{\lim}_{\bf{C}^*}}
\newcommand\catou{\bf{OU}}
\newcommand\limou{\underrightarrow{\lim}_{\bf{OU}}}
\newcommand\cataou{\bf{AOU}}
\newcommand\limaou{\underrightarrow{\lim}_{\bf{AOU}}}
\newcommand\catmou{\bf{MOU}}
\newcommand\limmou{\underrightarrow{\lim}_{\bf{MOU}}}
\newcommand\catos{\bf{OS}}
\newcommand\limos{\underrightarrow{\lim}_{\bf{OS}}}
\newcommand\catosp{\bf{OSp}}

\newcommand\catdos{\bf{DOS}}

\newcommand\catoas{\bf{OC}^*\bf{S}}
\newcommand\limoas{\underrightarrow{\lim}_{\bf{OC}^*\bf{S}}}

\newcommand\limcat{\underrightarrow{\lim}_{\bf{C}}}

\newcommand\lou[1]{\ddot{#1}_{\infty}}
\newcommand\laou[1]{{#1_{\infty}}}

\newcommand\los[1]{\cl{#1}_{\infty}}
\newcommand\losc[1]{\widehat{\cl{#1}}_{\infty}}
\newcommand\lcalg[1]{\widehat{\cl{#1}}_{\infty}}

\newcommand{\cl}[1]{\mathcal{#1}}
\newcommand{\bb}[1]{\mathbb{#1}}
\newcommand{\vertiii}[1]{{\left\vert\kern-0.25ex\left\vert\kern-0.25ex\left\vert #1
    \right\vert\kern-0.25ex\right\vert\kern-0.25ex\right\vert}}

\begin{document}

\title{Inductive limits in the operator system and related categories}

\author[L. Mawhinney]{Linda Mawhinney}
\address{Mathematical Sciences Research Centre,
  Queen's University Belfast, Belfast BT7 1NN, United Kingdom}
\email{lmawhinney03@qub.ac.uk}

\author[I. Todorov]{Ivan G. Todorov}
\address{Mathematical Sciences Research Centre,
  Queen's University Belfast, Belfast BT7 1NN, United Kingdom}
\email{i.todorov@qub.ac.uk}

\date{15 April 2017}

\begin{abstract}
We present a systematic development of inductive limits in the categories of ordered *-vector spaces, 
Archimedean order unit spaces, matrix ordered spaces, operator systems and operator C*-systems. 
We show that the inductive limit intertwines the operation of passing to the maximal operator system structure 
of an Archimedean order unit space, and that the same holds true for the minimal operator system 
structure if the connecting maps are complete order embeddings. 
We prove that the inductive limit commutes with 
the operation of taking the maximal tensor product with another operator system, and
establish analogous results for injective functorial tensor products provided the connecting maps are complete order embeddings. 
We identify the inductive limit of quotient operator systems as a quotient of the inductive limit,
in case the involved kernels are completely biproximinal. 
We describe the inductive limit of graph operator systems as operator systems of topological graphs, 
show that two such operator systems are completely order isomorphic if and only if 
their underlying graphs are isomorphic, identify the C*-envelope of such an operator system, 
and prove a version of Glimm's Theorem on the isomorphism of UHF algebras in the category of operator systems. 
\end{abstract}

\maketitle

\tableofcontents

\section{Introduction}\label{s_intro}

Operator systems were first studied in the late 1960s by Arveson~\cite{Arveson}. 
Over the past five decades they have played a significant role in the development of 
non-commutative functional analysis and 
nowadays there is an extensive body of literature on their structure and properties
\cite{BlecherLeMerdy, ChoiEffros, PaulsenBook}.

Compared to the longer-studied category of C*-algebras, operator systems have the advantage to 
capture in a more subtle way properties of non-commutative order. 
It has become clear, for instance, 
that they can behave very differently than C*-algebras regarding the formation of 
categorical constructs such as tensor products 
\cite{PaulsenSeveriniStahlkeTodorovWinter, PaulsenTodorovTomforde}.
At the same time, their simpler structure
allows one to express complexities of infinite-dimensional phenomena 
through finite-dimensional objects. For example, finite dimensional operator systems
can be used to both reformulate the Connes Embedding Problem \cite{KavrukPaulsenTodorovTomforde}
and to characterise the weak expectation property of C*-algebras \cite{fkpt}.

Inductive limits of C*-algebras first appeared over fifty years ago in~\cite{Glimm}, and have ever since 
occupied a prominent place in C*-algebra theory. 
In addition to their cornerstone role in Elliott's classification programme \cite{Elliott76},
they have been instrumental in applications to quantum physics,
where questions of fundamental theoretical nature can be expressed in those terms 
\cite{fnw}. 
In contrast, there is no similar development in the operator system category. 
While inductive limits of complete operator systems were introduced by Kirchberg in \cite{Kirchberg}, 
and some very recent additions have been made through \cite{LuthraKumar} and \cite{Li},
no systematic study of operator system inductive limits and their applications has been conducted.

The purpose of this paper is to begin a systematic investigation of 
inductive limits of ordered *-vector spaces, operator systems and related categories,
and to highlight some first applications. 
Our approach differs substantially from the one of \cite{Kirchberg}. 
Indeed, due to the emphasis on the development of approximation techniques, 
Kirchberg's interest lies in complete operator systems.
Subsequently, his definition of the inductive limit relies
on the norm structure of the operator systems in question. 
Here, we are interested in the interactions between operator system structures and inductive limits, 
as well as in the tensor product theory, which was developed in \cite{PaulsenTodorovTomforde}  
in the more general case of non-complete operator systems. 
This setting allows to infer all desired properties based on the (matrix) order through the 
Archimedeanisation techniques introduced in \cite{PaulsenTomforde} and \cite{PaulsenTodorovTomforde}
and avoids to a substantial extent the use of norms. 

The paper is organised as follows. 
After recalling some preliminary background in Section \ref{chapt: preliminaries}, 
we construct, in Section \ref{chapt:inductive limit AOU}, the inductive limits in the categories of 
ordered *-vector spaces and Archimedean order unit (AOU) spaces.
We identify the state space of the inductive limit as the inverse limit of the corresponding state spaces. 
In Section \ref{chpt:Inductive limit of operator systems}, which is the core part of the paper, 
we develop in detail the inductive limit in the operator system category. 
We show that the
{\rm OMAX} operation, introduced in \cite{PaulsenTodorovTomforde},  
is intertwined by the inductive limit construction. A similar result holds true 
for the {\rm OMIN} operation, in the case the connecting morphisms are complete order embeddings. 
We identify the inductive limit of quotient operator systems as a quotient of an inductive limit, 
in case the involved kernels are completely biproximinal. 
We then establish a general intertwining theorem for the inductive limit and any 
injective functorial tensor products, provided the connecting maps are complete order embeddings. 
These results apply to the minimal tensor 
product to give a result, recently established in the case of complete operator systems, in 
\cite{LuthraKumar}, and have as a consequence a corresponding commutation result 
for the commuting tensor product that was also highlighted, in the case of complete operator systems, in \cite{LuthraKumar}.
Although this general theorem does not apply to the maximal tensor product, we show that 
the inductive limit intertwines this tensor product as well.
We also develop the inductive limit for the category of operator C*-systems \cite{PaulsenBook}, that is, 
operator systems that are bimodules over a given C*-algebra and whose matrix order structure is 
compatible with the module actions. 

In Section \ref{chpt:graph operator systems}, we 
consider inductive limits of graph operator systems. 
This class of operator systems was 
introduced in \cite{dsw} and subsequently studied in \cite{PaulsenOrtiz}, 
where the authors showed that the graph operator system is a complete isomorphism invariant for the 
corresponding graph, and identified its C*-envelope. 
In view of the importance of graph operator systems in Quantum Information Theory, where they
correspond to confusability graphs of quantum channels \cite{dsw}, 
we establish inductive limit versions of the aforementioned results.
Namely, we show that the inductive limit of graph operator systems
can be thought of as a topological graph operator system, 
and prove that two such operator systems are completely order isomorphic
precisely when their underlying topological graphs are isomorphic. 
We also identify the C*-envelope of such an operator system as 
the C*-subalgebra of the surrounding UHF algebra, generated by the operator system.
Finally, we prove an operator system version of the classical 
theorem of Glimm's that characterises *-isomorphism of two UHF algebras
in terms of embeddings of the intermediate matrix algebras. 
Crucial for this section are Power's monograph \cite{Power} and the development of 
topological equivalence relations therein.

\section{Preliminaries} \label{chapt: preliminaries}

In this section, we gather necessary preliminary material that will be needed in the remainder of the paper.

\subsection{AOU spaces} \label{section:Prelimiaries:AOU}
This subsection contains the basics on Archimedean order unit vector spaces,
as developed by Pauslen and Tomforde in \cite{PaulsenTomforde}. 
A {\it *-vector space} \index{*-vector space} is a complex vector space $V$ equipped with an involution 
$^*$, that is, a mapping $^*: V \rightarrow V$ 
 such that $x^{**}=x$, $(x+y)^* = x^*+y^*$ and $(\lambda x)^* = \bar{\lambda} x^*$ for all $x, y \in V$ and all $\lambda \in \mathbb{C}$.
     Let $V_h = \lbrace x \in V : x^*=x \rbrace$ and call the elements of $V_h$ {\it hermitian}. For any $x \in V$, we have that 
    \[
        x = {\rm Re}(x) + i {\rm Im}(x),
    \] where
    \[{\rm Re}(x) = \frac{x+x^*}{2} \text{~~~and~~~}{\rm Im}(x) = \frac{x-x^*}{2i} \]
    are hermitian.
    An {\it ordered *-vector space}\index{ordered $^*$-vector space} is a pair $(V,V^+)$, 
    where~$V$ is a *-vector space and~$V^+$ is a cone in~$V_h$ 
    (that is, a subset of $V_h$ closed under addition and multiplication by a non-negative scalar) 
    such that $V^+ \cap (-V^+) = \lbrace 0 \rbrace$. For $x, y \in V$, we write $x \leq y$ if $y-x \in V^+.$

    Let $(V,V^+)$ be an ordered *-vector space. We say that $e \in V^+$ is an {\it order unit} for $V$ 
    if for every $x \in V_h,$ there exists $r > 0$ such that $x \leq re$. 
    We call the triple $(V,V^+,e)$ an {\it order unit space}.
    An order unit $e \in V^+$ is called {\it Archimedean} if
    $$V^+ = \{x\in V_h : re+x \in V^+ \mbox{ for all } r > 0\}.$$ 
    A triple $(V,V^+, e)$, where $(V,V^+)$ is an ordered *-vector space for which $e$ is an Archimedean order unit,
is called an {\it Archimedean order unit space} (or an {\it AOU~space}, for short). We will often denote by $e_V$ 
the order unit with which an ordered *-vector space $(V,V^+)$ is equipped. 

    Let $(V,V^+)$ and $(W,W^+)$ be order unit spaces. 
    A linear map $\phi: V \rightarrow W$ is called {\it positive} if $\phi(V^+) \subseteq W^+$ and {\it unital} if $\phi(e_V)=e_W.$ 
    The map $\phi: V \rightarrow W$ is called an {\it order isomorphism} if it is bijective and $v \in V^+$ if and only if $\phi(v) \in W^+.$
    The complex field $\mathbb{C}$ will henceforth be equipped with its standard AOU space structure, 
    and linear maps $f : V\to \bb{C}$ will as usual be referred to as {\it functionals}. A \emph{state} 
    on $V$ is a unital positive functional. 
We write $S(V)$ for the set of all states on $V$ and call it the {\it state space}; note that $S(V)$ is a cone.

    Let $(V, V^+)$ be an ordered *-vector space with order unit $e$. 
    We introduce a seminorm on $V_h$, letting
    \[
    \norm{x}_h = \inf \lbrace r \in \mathbb{R} : -re \leq x \leq re \rbrace, \ \ \ x \in V_h.
    \]
    We call $\norm{\cdot}_h$ the {\it order seminorm} on $V_h$ determined by $e$. 
    We note that $\norm{\cdot}_h$ is a norm if $e$ is Archimedean \cite[Proposition 2.23]{PaulsenTomforde}.
    An {\it order seminorm} $\vertiii{\cdot}$ on $V$ is a seminorm such that $\vertiii{x^*} = \vertiii{x}$ for all $x \in V$ and $\vertiii{x} = \norm{x}_h$ for all $x \in V_h$. The set of order 
    seminorms on $V$ has a maximal and a minimal (with respect to point-wise ordering) element. 
    The minimal seminorm is given by letting 
\begin{equation}\label{eq_normx}
    \norm{x} = \sup \lbrace |f(x)|: f \in S(V) \rbrace,
\end{equation}
    and all order seminorms are equivalent to it. 
    If $(V, V^+, e)$ is an AOU space, $\norm{\cdot}$ is in fact a norm.

    By (\ref{eq_normx}), 
    the states of $V$ are contractive with respect to the minimal order seminorm. We denote by $V'$ the
    space of all functionals continuous in the topology defined by any of the order seminorms. If $V_1$ and $V_2$ are ordered 
    *-vector spaces with order units and $\phi : V_1 \rightarrow  V_2$ is a unital positive map then we let
    $\phi' : V_2' \rightarrow V_1'$ be the dual of $\phi$. 
    If $(V, V^+, e)$ is an AOU space then
    $S(V)$ is a compact topological space 
    with respect to the weak* topology inherited from the topology generated by any order norm on $V$. 
    Thus, the map $\phi: V \rightarrow C(S(V))$, given by $\angles{\phi(x)}{f} = \angles{f}{x}$, 
    is a unital injective map that is an order isomorphism onto its image. 
    Furthermore, $\phi$ is an isometry with respect to the minimal order norm on $V$ and the uniform norm on $C(S(V))$. 
The latter statement can be viewed as a complex version of Kadison's representation theorem (see \cite{Kadison} or \cite[Theorem~II.1.8]{AlfsenBook}); for a proof, we refer the reader to \cite[Theorem~5.2]{PaulsenTomforde}.

    The proof of the next remark is straightforward and is omited.

    \begin{remark} \label{rem:normouspace0}
      {\rm  Let $V$ be an ordered *-vector space with order unit and let $\norm{\cdot}'$ 
      be an order seminorm on $V$. 
      If $x\in V$, we have that $\norm{x}' = 0$ if and only if 
      $\norm{{\rm Re} (x)}_{h}=0$ and $\norm{{\rm Im} (x)}_{h}=0$.
%
            }
    \end{remark}

    In order to avoid excessive notation, we will sometimes denote the ordered *-vector space $(V, V^+, e)$ simply by $V$.

    Let us denote by $\catou$ the category whose objects are ordered *-vector spaces with order units and whose morphisms are unital positive maps, and by $\cataou$ the category whose objects are AOU spaces and whose morphisms are unital positive maps. 
Clearly, we have 
a forgetful functor $\bf{F}: \cataou \rightarrow \catou$. 
In~\cite[Section~3.2]{PaulsenTomforde}, the process of Archimedeanisation is discussed which provides us with a left adjoint to this functor. 
Let $(V, V^+, e)$ be an ordered *-vector space with order unit. Define
    \[
        D = \{ v \in V_h : re+v \in V^+ \text{~for every~} r>0\}
    \]
    and
    \begin{equation}\label{eq:N1}
        N = \lbrace v \in V : f(v)=0 \text{~for all~} f \in S(V) \rbrace.
    \end{equation}
Clearly, $D$ is a cone, while $N$ is a linear subspace of $V$. 
    Equip $V/N$ with the involution given by $(v+N)^* = v^*+N$, and
    set
    \[
        (V/N)^+ = \{v+N : v \in D\}.
    \]
    It was shown in \cite[Theorem 3.16]{PaulsenTomforde} that $(V/N, (V/N)^+, e+N)$ is an AOU space, 
    which was called therein the {\it Archimedeanisation} of $(V, V^+, e)$ and denoted by $V_{\text{Arch}}$. 
It satisfies the following universal property.

    \begin{theorem} \label{lem:archimedeanisation of positive map}
        Let $V$ be an ordered *-vector space with order unit. 
        The quotient map $\varphi: V \rightarrow V_{\rm Arch}$ is the unique 
        unital surjective positive linear map from $V$ onto $V_{\rm Arch}$
        such that, whenever $W$ is an Archimedean order unit space and $\phi: V \rightarrow W$ is a unital positive linear map, then there exists a unique positive linear map $\phi_{\rm Arch}: V_{\rm Arch} \rightarrow W$ such that $\phi = \phi_{\rm Arch} \circ \varphi$.

        Furthermore, if $(\tilde{V}, \tilde{\varphi})$ is a pair consisting of an AOU 
        space $\tilde{V}$ and a unital surjective positive linear map $\tilde{\varphi} : V\to \tilde{V}$ 
        such that, 
        whenever 
        $W$ is an Archimedean order unit space and $\phi: V \rightarrow W$ is a unital positive linear map
        there exists a unique positive linear map $\tilde{\phi} : \tilde{V} \rightarrow W$ with  
        $\phi = \tilde{\phi} \circ \tilde{\varphi}$,
        then there exists a unital order isomorphism $\psi : V_{\rm Arch} \rightarrow \tilde{V}$ such that $\psi \circ \varphi = \tilde{\varphi}.$
    \end{theorem}

    \subsection{Operator systems, operator spaces and tensor products}
    \subsubsection{Basic concepts}

    For a vector space $\cl S$, we let $M_{n,m}(\cl S)$ be the vector space of all $n$ by $m$ matrices with entries in $\cl S$. 
    We set $M_{n,m} = M_{n,m}(\bb{C})$, $M_{n}(\cl S) = M_{n,n}(\cl S)$ and $M_n = M_{n,n}$. 
    
    Let $\cl S$ be a *-vector space. We equip 
    $M_n(\cl S)$ with the involution $(s_{i,j})_{i,j}^*=(s_{j,i}^*)_{i,j}$; it turns $M_n(\cl S)$ into a *-vector space.
    A family $\lbrace C_n\rbrace_{n \in\bb{N}}$, where $C_n\subseteq M_n(V)$, 
is called a {\it matrix ordering} on $\cl S$ if
    \begin{enumerate}[(i)]
    \item $C_n$ is a cone in $M_n(\cl S)_h$ for each $n \in \mathbb{N}$,
    \item $C_n \cap (- C_n) = \lbrace 0 \rbrace$ for each $n \in \mathbb{N}$, and
    \item for each $n,m \in \mathbb{N}$ and each $\alpha \in M_{n,m}$ we have that $\alpha^* C_n \alpha \subseteq C_m$.
    \end{enumerate}
    A {\it matrix ordered} *-vector space  is a pair 
    $(\cl S, \lbrace C_n \rbrace_{n \in \bb{N}})$ where $\cl S$ is a *-vector space and $\lbrace C_n \rbrace_{n \in\bb{N}}$ is a matrix ordering. 
    We refer to condition (iii) as the {\it compatibility} of the family $\lbrace C_n \rbrace_{n \in\bb{N}}$
    and often write $M_n(\cl S)^+$ for $C_n$. 

    Let $(\cl S, \lbrace C_n \rbrace_{n \in \bb{N}})$ be a matrix ordered *-vector space. For each $e \in \cl S$ and $n \in \mathbb{N}$, let
    \[
    e^{(n)} :=
    \begin{pmatrix}
        e & & & & \\
        & &\ddots& & \\
        & &&& e\\
    \end{pmatrix} \in M_n(\cl S),
    \]
    where the off-diagonal entries are zero.
    We say that $e \in C_1$ is a {\it matrix order unit} if $e^{(n)}$ is an order unit for $(M_n(\cl S), C_n)$ for all $n \in \mathbb{N}$. 
    Similarly, we say that $e$ is an {\it Archimedean matrix order unit}
    if $e^{(n)}$ is an Archimedean order unit for $(M_n(\cl S), C_n)$ for each $n \in \mathbb{N}$.
    An {\it operator system} is a matrix ordered *-vector space with an Archimedean matrix order unit.
        In order to avoid excessive notation, we will sometimes denote the 
        triple $(\cl S, \{C_n\}_{n \in \bb{N}}, e)$ simply by $\cl S$; the unit is denoted by 
$e_{\cl S}$ if there is risk of confusion.

Let $\cl S$ and $\cl T$ be matrix ordered *-vector spaces with matrix order units and $\phi: \cl S \rightarrow \cl T$ be a linear map. 
We define $\phi^{(n)}: M_n(\cl S)\rightarrow M_n(\cl T)$ by letting
    \[
        \phi^{(n)}
        \left(
        \begin{pmatrix}
            s_{1,1} & \cdots & s_{1,n} \\
            \vdots & & \vdots\\
            s_{n,1} &\cdots & s_{n,n}   \\
        \end{pmatrix}
        \right)
        =
        \begin{pmatrix}
            \phi(s_{1,1}) & \cdots & \phi(s_{1,n}) \\
            \vdots & & \vdots\\
            \phi(s_{n,1}) & \cdots& \phi(s_{n,n})   \\
        \end{pmatrix},
    \]
$n \in \mathbb{N}$. We say that $\phi$ is {\it $n$-positive} if $\phi^{(n)}$ is positive and that $\phi$ is {\it completely positive}
    if $\phi$ is $n$-positive for all $n \in \mathbb{N}$. Furthermore, we say that $\phi$ is a {\it complete order isomorphism}
    if $\phi$ is a bijection and both $\phi$ and $\phi^{-1}$ are completely positive. We say that $\phi$ is a 
    {\it unital complete order embedding} if $\phi$ is a unital complete order isomorphism onto its image.

       Let $\BH$ denote the space of all bounded linear operators on a Hilbert space $\cl H$. 
       If $\cl S$ is a subset of $\BH$, we set $\cl S^*  = \{s \in \cl S: s^* \in \cl S\}$ and call $\cl S$ {\it selfadjoint} 
       if $\cl S = \cl S^*$. 
       A {\it concrete operator system} is a selfadjoint subspace of $\BH$ which contains the identity operator $I$. 
       If $\cl S\subseteq \cl B(\cl H)$ is a concrete operator system then
       $M_n(\cl S) \subseteq M_n(\BH) \cong \BHn$ and therefore $M_n(\cl S)$ inherits an order structure from 
       $\BHn$.
       Note that $I$ is an Archimedean matrix order unit for the matrix order structure thus defined. 
       Hence, a concrete operator system is an operator system. 
       The following fundamental result \cite[Theorem~4.4]{ChoiEffros} establishes the converse.

    \begin{theorem}[Choi--Effros \cite{ChoiEffros}]\label{choieffros}
        Let $(V, \{C_n\}_{n \in \bb{N}}, e)$ be an operator system.
        Then there exists a Hilbert space $\hilb$, a concrete operator system $\cl S \subseteq \BH$ and a complete order isomorphism $\Phi: V \rightarrow \cl S$ such that $\Phi(e) = I$.
    \end{theorem}

    \subsubsection{Operator spaces}\label{sss_os}
  Let $\cl X$ be a Banach space and $\norm{\cdot}_{n}$ be a norm on $M_n(\cl X)$, $n\in \bb{N}$. 
  We call 
  $(\cl X, \{\norm{\cdot}_{n}\}_{n \in \bb{N}})$ an {\it operator space} if the following are satisfied: 
  \begin{enumerate}[\rm(i)]
      \item
      $\left \lVert
      \begin{pmatrix}
      X & 0 \\
      0 & Y \\
    \end{pmatrix}
      \right \rVert_{n+m} = \max\{\norm{X}_{n}, \norm{Y}_{m}\}$ and
      \item $\norm{\alpha X \beta}_{n} \leq \norm{\alpha} \norm{X}_{n} \norm{\beta}$
  \end{enumerate}
  for all $X \in M_n(\cl X), Y \in M_m(\cl X)$ and $\alpha, \beta \in M_n$. 
 
   Let $(\cl X, \{\norm{\cdot}_{n}\}_{n \in \bb{N}})$ and $(\cl Y, \{\norm{\cdot}_{n}\}_{n \in \bb{N}})$
    be operator spaces and let $\phi: \cl X \rightarrow \cl Y$ be a linear map. 
    We let $\norm{\phi}_{\rm{cb}} = \sup \{\norm{\phi^{(n)}} : n \in \bb{N}\}$ and 
    say that $\phi$ is {\it completely bounded} 
    if $\norm{\phi}_{\rm cb}$ is finite, $\phi$ is {\it completely contractive} if $\norm{\phi}_{\rm cb} \leq 1$, and 
    a {\it complete isometry} if $\phi^{(n)}$ is an isometry for every $n$. 



  Let us denote by $\catosp$ the category whose objects are operator spaces and whose morphisms are completely bounded maps.
If $\cl X$ is an operator space and $\cl X'$ is the Banach space dual of $\cl X$
  then there is a natural way to induce an operator space structure on $\cl X'$ as follows 
  \cite{Blecher} (for more details see \cite[Section~3.2]{EffrosRuanBook}).
  If $S = (s_{i,j})_{i,j} \in M_n(\cl X')$ then~$S$ determines a linear mapping $\widetilde{S} : \cl X \rightarrow M_n$,
  given by $\widetilde{S}(x) = (\langle s_{i,j},x\rangle)_{i,j}$; 
  we set $\|S\|_n = \|\widetilde{S}\|_{\rm cb}$. 




It follows from the Choi-Effros Theorem that every operator system is, canonically, an operator space. 
The following result \cite[Lemma 3.1]{PaulsenBook} provides a characterisation of the 
norm in operator systems in terms of the matrix order structure. 

    \begin{lemma} \label{rem:norm of an operator system}
    Let $\cl S$ be an operator system and $x \in M_n(\cl S)$. Then
$        \norm{x} \leq 1 \text{~if and only if~}
        \begin{pmatrix}
            1_n & x \\
            x^* & 1_n \\
        \end{pmatrix} \in M_{2n}(\cl S)^+.$
    \end{lemma}

    If $\phi$ is a unital map between operator systems 
    then $\phi$ is completely contractive if and only if $\phi$ is completely positive (see \cite[Proposition 3.6]{PaulsenBook}).
    Thus, a unital linear map between operator systems is a unital complete isometry if and only if it is a unital complete order embedding.

    It is proved in \cite{PaulsenTomforde} that if $\cl A$ is a unital C*-algebra, then its C*-norm is an order norm. 
    Therefore, if $\cl S$ is an operator system, the operator system norm on $\cl S$ is an order norm. Indeed, if we choose a unital 
    C*-algebra $\cl A$ with unit $e_{\cl A}$ such that $\phi: \cl S \rightarrow \cl A$ is a unital complete order 
    embedding (see Theorem~\ref{choieffros} for the existence of $\cl A$ and $\phi$), then for any $r \in \bb{R}$ and $s \in \cl S$,
    \[
        -re_{\cl S} \leq s \leq re_{\cl S} \text{~~~~if and only if~~~~} -re_{\cl A} = \phi(-re_{\cl S}) \leq \phi(s) \leq \phi(re_{\cl S}) = re_{\cl A}.
    \]
    Therefore for any $s \in \cl S_h$, 
    \begin{eqnarray*}
    \norm{s}_{\cl S} = 
    \norm{\phi(s)}_{\cl A} & = &  \inf\{r \in \bb{R} : -re_{\cl A} \leq \phi(s) \leq re_{\cl A}\}\\
    & = & \inf\{r \in \bb{R} : -re_{\cl S} \leq s \leq re_{\cl S}\}.
\end{eqnarray*}

\subsubsection{Operator system tensor products}

    Suppose that $(\cl S, \{C_n\}_{n \in \bb{N}}, e_{\cl S})$ and $(\cl T,$  $\{D_n\}_{n \in \bb{N}}, e_{\cl T})$ are operator systems. 
    We denote by $\cl S \odot \cl T$ their algebraic tensor product. 
    We call a family $\mu = \{P_n\}_{n \in \bb{N}}$ of cones, where $P_n \subseteq M_n(\cl S \odot \cl T)$, 
    an {\it operator system structure} on $\cl S \odot \cl T$ if it satisfies the following properties: \index{operator system tensor product}
    \begin{enumerate}[\rm(i)]
    \item $(\cl S \odot \cl T, \{P_n\}_{n \in \bb{N}}, e_{\cl S} \otimes e_{\cl T})$ is an operator system, denoted $\cl S \otimes_{\mu} \cl T$,
    \item $C_n \otimes D_m \subseteq P_{nm}$ for all $n,m \in \bb{N}$, and
    \item if $m,n \in \bb{N}$ and $\phi: \cl S \rightarrow M_n$, $\psi: \cl T \rightarrow M_m$ are unital completely positive maps
     then $\phi \otimes \psi: \cl S \otimes_{\mu} \cl T \rightarrow M_{nm}$ is a completely positive map.
    \end{enumerate}

    An {\it operator system tensor product} \cite{KavrukPaulsenTodorovTomforde2} is a mapping 
    $\mu: \catos \times \catos \rightarrow \catos$ such that $\mu(\cl S, \cl T)$ is an operator system 
    with an underlying space $\cl S\odot \cl T$ for every $\cl S, \cl T \in \catos$. 
    We call an operator system tensor product {\it functorial} if for any four operator systems $\cl S_1, \cl S_2, \cl T_1$ and $\cl T_2$ we have that if $\phi_1: \cl S_1 \rightarrow \cl T_1$ and $\phi_2: \cl S_2 \rightarrow \cl T_2$ are unital completely positive maps then $\phi_1 \otimes \phi_2: \cl S_1 \otimes_{\mu} \cl S_2 \rightarrow \cl T_1 \otimes_{\mu} \cl T_2$ is a unital completely positive map. 
    An operator system tensor product is {\it injective} if whenever $\phi_1$ and $\phi_2$ are unital complete order 
    embeddings then $\phi_1 \otimes \phi_2$ is a unital complete order embedding.     
Let $\cl T$ be an operator system and $\mu$ be 
an operator system tensor product. We say that~$\cl T$ is {\it $\mu$-injective} if for any pair of operator systems $\cl S_1$ and $\cl S_2$ we have that if $\phi: \cl S_1 \rightarrow \cl S_2$ is a unital complete order embedding 
then $\phi \otimes {\id}_{\cl T} : \cl S_1 \otimes_{\mu} \cl T \rightarrow \cl S_2 \otimes_{\mu} \cl T$ is a unital complete order embedding.

    Let $(\cl S, \{C_n\}_{n \in \bb{N}}, e_{\cl S})$ and $(\cl T, \{D_n\}_{n \in \bb{N}}, e_{\cl T})$ 
    be operator systems and let $\iota_{\cl S}: \cl S \rightarrow \BH$ and 
    $\iota_{\cl T}: \cl T \rightarrow \BK$ be unital complete order embeddings. 
    The {\it minimal operator system tensor product} $\cl S \mitp \cl T$ of $\cl S$ and $\cl T$
    has operator system structure arising from the 
embedding $\iota_{\cl S} \otimes \iota_{\cl T}: \cl S \odot \cl T \rightarrow \mathcal{B}(\cl H \otimes \cl K)$. 
    It is proved in \cite[Theorem~4.4]{KavrukPaulsenTodorovTomforde2} that the minimal operator system tensor product is injective and functorial.
    
    Let
    \[
        P_n^{\max}(\cl S, \cl T) := \{\alpha(C \otimes D) \alpha^*: C \in C_k, D \in D_m, \alpha \in M_{n,km},  k,m \in \bb{N}\}.
    \]
    The {\it maximal operator system tensor product} of $\cl S$ and $\cl T$, denoted 
    $\cl S \mtp \cl T$, is the Archimedeanisation of $(\cl S \otimes \cl T, \{P_n^{\max}(\cl S, \cl T)\}_{n \in \bb{N}} , e_{\cl S} \otimes e_{\cl T})$. 
    Let $\hilb$ be a Hilbert space. A bilinear map $\phi: \cl S \times \cl T \rightarrow \BH$ is called 
    {\it jointly completely positive} if, for all $P = (x_{i,j})\in C_n$ and $Q = (y_{k,l})\in D_m$, the matrix
    $\phi^{(n,m)}(P,Q) := (\phi(x_{i,j},y_{k,l}))$ is a positive element of $M_{nm}(\BH)$ .

    \begin{theorem}\label{thm:universal property of mtp}
        Let $\cl S$ and $\cl T$ be operator systems. If $\phi: \cl S \times \cl T \rightarrow \BH$ is a jointly completely positive map, then its linearisation $\phi_{L}: \cl S \mtp \cl T \rightarrow \BH$ is completely positive.

        Furthermore if $\mu$ is an operator system structure on $\cl S \odot \cl T$ with the property that the linearisation of every jointly completely positive map $\phi: \cl S \times \cl T \rightarrow \BH$ is completely positive on $\cl S \otimes_{\mu} \cl T$ then $\cl S \otimes_{\mu} \cl T = \cl S \mtp \cl T$. \index{jointly completely positive}
    \end{theorem}

    The {\it commuting operator system tensor product} is the operator system arising from the inclusion of 
    $\cl S \odot \cl T$ into $C_u^*(\cl S) \mtp C_u^*(\cl T)$, denoted $\cl S \ctp \cl T$
    (see Subsection \ref{sect:C*extensions} for the definition of the universal 
    C*-algebra $C^*_u(\cl R)$ of an operator system $\cl R$). 
    It is proved 
    in \cite[Theorem~5.5~and~Theorem~6.3]{KavrukPaulsenTodorovTomforde2} that the maximal operator system tensor product and the commuting operator system tensor product are both functorial.

    \subsubsection{The Archimedeanisation of matrix ordered *-vector spaces}

    The process of Archimedeanisation for matrix ordered *-vector spaces was described in 
    \cite[Section 3.1]{PaulsenTodorovTomforde}. 
    Let $(\cl S, \{ C_n\}_{n \in \bb{N}},e)$ be a matrix ordered *-vector space with matrix order unit. 
    For each $n \in \bb{N}$, set
    \[
        N_n = \big\lbrace S \in M_n(\cl S) : f(S)=0 \text{~for all~} f \in S(M_n(\cl S))\big \rbrace.
    \]
    Recall the notation from (\ref{eq:N1}); it is proved in \cite[Lemma~3.14]{PaulsenTodorovTomforde} 
    that $M_n(N) = N_n$, $n \in \bb{N}$. Define
    \[
    \begin{split}
        C_n^{\text{Arch}} = \Big\lbrace S + M_n(N) \in &(M_n(\cl S)/M_n(N))_h : \\
        &r e^{(n)}+S + M_n(N) \in  C_n + M_n(N) \text{~for all~} r>0 \Big\rbrace.
      \end{split}
    \]
    Then $(\cl S/N, \{ C_n^{\rm{Arch}}\}_{n \in \bb{N}}, e+N)$ is an operator system. 
    We call this the {\it Archimedeanisation} of $\cl S$ and denote it by $\cl S_{\text{Arch}}$. 
    It has the following universal property.

    \begin{theorem}[\cite{PaulsenTodorovTomforde}]\label{prop: Archimedeanisation of mou space}
        Let $\cl S$ be an matrix ordered *-vector space with matrix order unit.
        The quotient map $\varphi: \cl S \rightarrow \cl S_{\rm Arch}$ is the unique
        unital surjective completely positive map 
        such that, 
        whenever $\cl T$ is an operator system and 
        $\phi: \cl S \rightarrow \cl T$ is a unital completely positive map, there exists 
        a unique completely positive map $\phi_{\rm Arch}: \cl S_{\rm Arch} \rightarrow \cl T$ 
        such that $\phi = \phi_{\rm Arch} \circ \varphi$.
        
        Furthermore, if $(\tilde{\cl S}, \tilde{\varphi})$ is a pair consisting of an operator system and unital surjective 
        completely positive map $\tilde{\varphi} : \cl S\to \tilde{\cl S}$ 
        with the property that, 
whenever $\cl T$ is an operator system and 
        $\phi: \cl S \rightarrow \cl T$ is a unital completely positive map, there exists 
        a unique completely positive map $\tilde{\phi} : \tilde{\cl S} \rightarrow \cl T$ 
        such that $\phi = \tilde{\phi} \circ \tilde{\varphi}$,        
      then there exists a unital complete order isomorphism 
        $\psi : \cl S_{\rm Arch} \rightarrow \tilde{\cl S}$ such that $\psi \circ \varphi = \tilde{\varphi}.$
    \end{theorem}

    \begin{remark} \label{rem:Arch is Arch at every matrix level}
  {\rm  It is shown in \cite[Remark 3.17]{PaulsenTodorovTomforde} that the Archimedeanisation of a matrix ordered *-vector space 
    with matrix order unit is precisely the operator system formed by taking the Archimedeanisation at every matrix level.}
  \end{remark}

    We will make use of the following facts in the sequel.

    \begin{lemma} \label{rem:linear map maintain operator system structure}
        Let $(\cl S, \{C_n\}_{n \in \bb{N}}, e)$ be an operator system, $V$ be a vector space and $\phi: \cl S \rightarrow V$ be an
        injective linear map. 
        Equip $\phi(\cl S)$ with the involution given by $\phi(x)^*\stackrel{def}{=} \phi(x^*)$. 
         Then $(\phi(\cl S), \{\phi^{(n)}(C_n)\}_{n \in \bb{N}}, \phi(e))$ is an operator system.
    \end{lemma}
\begin{proof}
The facts that the family $\{\phi^{(n)}(C_n)\}_{n \in \bb{N}}$ is compatible and that $\phi(e)$ is matrix order unit 
are straightforward.
To show that $\phi^{(n)}(e^{(n)})$ is Archime-dean, suppose that $x\in M_n(\cl S)$ is a selfadjont element such that
$\phi^{(n)}(x) + r \phi^{(n)}(e^{(n)}) \in \phi(C_n)$ for all $r > 0$. By the injectivity of $\phi$, $x + re^{(n)} \in C_n$ for all $r > 0$, 
and hence $x\in C_n$. Thus, $\phi^{(n)}(x)\in \phi^{(n)}(C_n)$ and the proof is complete. 
\end{proof}

    \begin{lemma} \label{rem:ucio conjgation}
        Let $\cl S, \cl T$ and $\cl P$ be operator systems and let $\phi: \cl S \rightarrow \cl T$ and $\psi: \cl T \rightarrow \cl P$ be unital linear maps. If $\psi$ and $\psi \circ \phi$ are complete order embeddings then $\phi$ is a complete order embedding.
        \end{lemma}
    \begin{proof}
        Let $n \in \bb{N}$ and $S \in M_n(\cl S)^+$. 
        Then $\psi^{(n)}\circ \phi^{(n)}(S) \in M_n(\cl P)^+$ and therefore $\phi^{(n)}(S) \in M_n(\cl T)^+$. 
        Suppose that $S\in M_n(\cl S)$ and $\phi^{(n)}(S) \in M_n(\cl T)^+$.
        Then $\psi^{(n)} \circ \phi^{(n)}(S) \in M_n(\cl P)^+$ and therefore $S \in M_n(\cl S)^+$.
    \end{proof}

    We denote by $\catmou$ 
    (resp. $\catos$) the category whose objects are matrix ordered *-vector spaces with matrix order unit 
    (resp. operator systems) and whose morphisms are unital completely positive maps.

\subsection{$\omin$ and $\omax$}\label{sec:pre.omin and omax}

Let $(V,V^+,e)$ be an AOU space. An \emph{operator system structure} on 
$(V,V^+,e)$ is a family $\{P_n\}_{n\in \bb{N}}$ such that 
$(V,\{P_n\}_{n\in \bb{N}},e)$ is an operator system and $P_1 = V^+$. 
There are two extremal operator system structures \cite{PaulsenTodorovTomforde} that will play a significant role in the sequel. 
The \emph{minimal} operator system structure on $(V,V^+,e)$ is the family $\lbrace C_n^{\min}(V)\rbrace_{n \in \bb{N}}$, 
where 
    \[
        C_n^{\min}(V) = \Big\lbrace (x_{i,j})_{i,j} \in M_n(V) : \sum_{i,j=1}^{n} \overline{\lambda_i} \lambda_j x_{i,j} \in V^+ ~\text{for all}~ \lambda_1, \ldots, \lambda_n \in \mathbb{C} \Big\rbrace.
    \]
We set $\omin(V) = (V,\lbrace C_n^{\min}(V)\rbrace_{n \in \bb{N}},e)$. 

    \begin{theorem}\label{thm:characterisation of omin}
    Let $(V, V^+, e)$ be an AOU space and $n \in \bb{N}$. Then $(x_{i,j})_{i,j} \in C_n^{\min}(V)$ if and only if $(\angles{f}{x_{i,j}})_{i,j} \in M_n^+$ for each $f \in S(V)$. 
    \end{theorem}

\noindent It follows from Theorem \ref{thm:characterisation of omin} that 
$\omin(V)$ is the operator system induced by the canonical inclusion of $V$ into $C(S(V))$.

We define $\omax(V)$ to be the Archimedeanisation of the matrix ordered space 
$(V, \lbrace D_n^{\max}(V)\rbrace_{n \in \bb{N}}, e)$, where
    \[
    D_n^{\max}(V) = \left \lbrace \sum_{i=1}^{k} a_i \otimes x_i : x_i \in V^+, ~a_i \in M_n^+, ~i=1, \ldots, k, ~k \in \mathbb{N} \right\rbrace.
    \]

    Let $\bf{F}: \catos \rightarrow \cataou$ be the forgetful functor. It can be seen that $\omin:\cataou \rightarrow \catos$ is a right adjoint functor to $\bf{F}$ and $\omax:\cataou \rightarrow \catos$ is a left adjoint functor to $\bf{F}$ (see \cite{MacLane} for relevant background 
    in Category Theory).

    \subsection{C*-covers} \label{sect:C*extensions}

Let $\cl S$ be an operator system. 
A {\it C*-cover} is a pair $(\cl A,\nu)$ consisting of a unital C*-algebra and a unital completely isometric embedding 
$\nu:\cl S \rightarrow \cl A$ such that $\nu(\cl S)$ generates $\cl A$ as a C*-algebra. 
The {\it universal C*-cover} $(C^*_u(\cl S),\iota)$ of $\cl S$ 
was defined in \cite{KirchbergWassermann} and is characterised by the following universal property:

    \begin{proposition} \label{prop:universal prop universal calg}
    Let $\cl S$ be an operator system, 
    $\cl A$ be a C*-algebra and $\phi: \cl S \rightarrow \cl A$ be a unital completely positive map. 
    Then there exists a *-homomorphism
    \[
        \widetilde{\phi} : C^*_u(\cl S) \rightarrow \cl A
    \]
    such that $\widetilde{\phi}\circ \iota =\phi$. Moreover, 
    if $(\cl B,\mu)$ is another C*-cover of $\cl S$ 
    such that, 
    whenever $\cl A$ is a C*-algebra and $\phi: \cl S \rightarrow \cl A$ be a unital completely positive map,
    there exists a *-homomorphism
    \[
        \widetilde{\phi} : \cl B \rightarrow \cl A
    \]
    such that $\widetilde{\phi}\circ \mu =\phi$,     
    then there exists a *-isomorphism $\rho : \cl B \rightarrow C_u^*(S)$  with $\rho \circ \mu = \iota$.
    \end{proposition}

    We call $C_u^*(\cl S)$ the {\it universal C*-algebra of $\cl S$}.
The {\it C*-envelope} of $\cl S$,
introduced in \cite{Hamana} (see also \cite[Section~4.3]{BlecherLeMerdy}) 
is, on the other hand, the C*-cover $(C^*_e(\cl S),\kappa)$,
characterised by the following universal property: if 
$(\cl A, \phi)$ is a C*-cover of $\cl S$, then there exists a *-homomorphism
    \[
        \widetilde{\phi} : \cl A \rightarrow C^*_e(\cl S)
    \]
    such that $\widetilde{\phi}\circ \phi =\kappa$. 
    Clearly, the pair $(C_e^*(\cl S), \kappa)$ is unique in the sense that if $(\cl B,\mu)$ 
    is another pair with the same property then there exists a *-isomorphism $\rho : \cl B \rightarrow C_e^*(S)$  with 
    $\rho \circ \mu = \kappa$.

    The following fact is a straightforward consequence of the universal property of C*-envelopes.

    \begin{remark} \label{cor:ucoi induces *iso in envelopes}
{\rm        Let $\cl S$ and $\cl T$ be operator systems and let $(C_e^*(\cl S), \iota_{\cl S})$ and $(C_e^*(\cl T), \iota_{\cl T})$ be the 
C*-envelopes of $\cl S$ and $\cl T$, respectively. If $\phi: \cl S \rightarrow \cl T$ is a unital complete order isomorphism, then there exists a 
unital *-isomorphism $\rho:  C_e^*(\cl S) \rightarrow C_e^*(\cl T)$ such that $\rho \circ \iota_{\cl S} = \iota_{\cl T}\circ \phi$.}
          \end{remark}

\subsection{Inductive limits}\label{subs_iil}

    We recall some basic categorical notions which will be necessary in the sequel; we refer
    the reader to \cite{MacLane} for further details.

    \begin{definition} \label{def:inductive limit}
    Let $\bf{C}$ be a category. An \emph{inductive system} in $\bf{C}$ is a 
pair $(\lbrace A_k \rbrace_{k \in \mathbb{N}}, \lbrace \alpha_k \rbrace_{k \in \mathbb{N}})$ where $A_k$ is an object in $\bf{C}$ and $\alpha_k: A_k \rightarrow A_{k+1}$ is a morphism for each $k \in \mathbb{N}$. An 
\emph{inductive limit}  
of the inductive system $(\lbrace A_k \rbrace_{k \in \mathbb{N}}, \lbrace \alpha_k \rbrace_{k \in \mathbb{N}})$ is a pair 
$(A, \lbrace \alpha_{k, \infty} \rbrace_{k \in \mathbb{N}})$ where $A$ is an object in $\bf{C}$ and $\alpha_{k,\infty}: A_k \rightarrow A$ is a morphism, 
$k \in \bb{N}$, such that

    \begin{enumerate}[(i)]
    \item $\alpha_{k+1,\infty} \circ \alpha_k = \alpha_{k,\infty}$, $k \in \bb{N}$, and
    \item if $(B, \lbrace \beta_{k} \rbrace_{k \in \mathbb{N}})$ is another pair such that $B$ is an object in $\bf{C}$,  $\beta_k: A_k \rightarrow B$ is a morphism and $\beta_{k+1} \circ \alpha_k = \beta_{k}$, $k \in \bb{N}$, then there exists a unique morphism $\mu : A \rightarrow B$ such that
    $\mu \circ \alpha_{k, \infty} = \beta_{k}$, $k \in \bb{N}$.
    \end{enumerate}
  \end{definition}
  
  Suppose that $(\lbrace A_k \rbrace_{k \in \mathbb{N}}, \lbrace \alpha_k \rbrace_{k \in \mathbb{N}})$ is an inductive system. 
If it exists, its inductive limit is unique and will be denoted by $\limcat (A_k, \alpha_k)$ or $\limcat A_k$ when the context is clear.
We will refer to $\alpha_k$, $k \in \bb{N}$, as the {\it connecting morphisms}, and set 
  $$\alpha_{k,l} = \alpha_{l-1}\circ \cdots \circ \alpha_k \mbox{ if } k<l \mbox{ and }  \alpha_{k,k} = \id\mbox{}_{A_k};$$
we thus have that $\alpha_{k,l}$ is a morphism from $A_k$ to $A_l$.
If every inductive system in the category ${\bf C}$ has an inductive limit, we say that ${\bf C}$ is a category 
{\it with inductive limits}.

  \begin{theorem} \label{thm:universal property infinitely many maps}
      Let {\bf C} be a category with inductive limits,
      and let \linebreak $(\{A_k\}_{k \in \mathbb{N}}, \{\phi_k\}_{k \in \mathbb{N}})$ 
      (resp. $(\{B_k\}_{k \in \mathbb{N}}, \{\psi_k\}_{k \in \mathbb{N}})$) be an inductive system in {\bf C} with 
      an inductive limit $(A, \{\phi_{k,\infty}\}_{k \in \bb{N}})$ (resp. $(B, \{\psi_{k,\infty}\}_{k \in \bb{N}})$). 
      Let $\{\theta_k\}_{k \in \mathbb{N}}$ be a sequence of morphisms such that the following diagram commutes:
      \begin{equation*} \label{eq:commdiagram1}
          \begin{CD}
              A_1 @>{\phi_1}>> A_2 @>{\phi_2}>> A_3 @>{\phi_3}>> A_4 @>{\phi_4}>> \cdots \\
              @V{\theta_1}VV @V{\theta_2}VV @V{\theta_3}VV @V{\theta_4}VV \\
              B_1 @>{\psi_1}>> B_2 @>{\psi_2}>> B_3 @>{\psi_3}>> B_4 @>{\psi_4}>> \cdots.
          \end{CD}
      \end{equation*}
      Then there exists a unique morphism $\theta : A \rightarrow  B $ such that $\theta \circ \phi_{k, \infty} = \psi_{k, \infty} \circ \theta_k$,
      $k \in \mathbb{N}$.
  \end{theorem}

    \begin{remark} \label{rem:indlimit subsequence}
{\rm        Let ${\bf C}$ be a category with inductive limits. 
        Let \linebreak $(\{A_k\}_{k \in \mathbb{N}}, \{\phi_k\}_{k \in \mathbb{N}})$ be an inductive system in ${\bf C}$ 
        with inductive limit $(A,$ $\{\phi_{k,\infty}\}_{k \in \bb{N}})$ and let $(n_k)_{k \in \mathbb{N}}\subseteq \mathbb{N}$ 
        be a subsequence. Then the inductive system
        $(\{A_{n_k}\}_{k \in \mathbb{N}}, \{\phi_{n_k, n_{k+1}}\}_{k \in \mathbb{N}})$
        has inductive limit $(A, \{\phi_{n_k,\infty}\}_{k \in \bb{N}})$. }
    \end{remark}

    \begin{proposition} \label{prop:iso ind limit infinitely many maps}
        Let {\bf C}, $(\{A_k\}_{k \in \mathbb{N}}, \{\phi_k\}_{k \in \mathbb{N}})$, 
        $(\{B_k\}_{k \in \mathbb{N}}, \{\psi_k\}_{k \in \mathbb{N}})$, $A$ and $B$
        be as in Theorem~\ref{thm:universal property infinitely many maps}.
        Suppose  
        $\{\theta_{2k-1}\}_{k \in \mathbb{N}}, \{\varphi_{2k}\}_{k \in \mathbb{N}}$ are sequences of morphisms 
        such that the following diagram commutes:
        \vspace{-1mm}
        \[
            \begin{CD}
                A_1 @>{\phi_1}>> A_2 @>{\phi_2}>> A_3 @>{\phi_3}>> A_4 @>{\phi_4}>> \cdots \\
                @V{\theta_1}VV @A{\varphi_2}AA @V{\theta_3}VV @A{\varphi_4}AA \\
                B_1 @>{\psi_1}>> B_2 @>{\psi_2}>> B_3 @>{\psi_3}>> B_4 @>{\psi_4}>> \cdots.
            \end{CD}
                \vspace{-2mm}
        \]
        Then $A$ is isomorphic to $B$.
    \end{proposition}

    \begin{remark}
{\rm      
Let ${\bf C}$ be a category with inductive limits 
and let \linebreak $(\{A_k\}_{k \in \mathbb{N}}, \{\phi_k\}_{k \in \mathbb{N}})$
(resp. $(\{B_k\}_{k \in \mathbb{N}}, \{\psi_k\}_{k \in \mathbb{N}})$) be an inductive system in {\bf C} 
with inductive limit $(A, \{\phi_{k,\infty}\}_{k \in \bb{N}})$ (resp. $(B, \{\psi_{k,\infty}\}_{k \in \bb{N}})$).
        By Remark~\ref{rem:indlimit subsequence} and Proposition~\ref{prop:iso ind limit infinitely many maps}, in order to show that $A$ and $B$ are isomorphic it suffices to find morphisms as in Proposition~\ref{prop:iso ind limit infinitely many maps} for subsystems
        \begin{equation*}
            A_{n_1}\stackrel{\phi_{{n_1}, {n_2}}}{\longrightarrow} A_{n_2} \stackrel{\phi_{n_2, n_3}}{\longrightarrow} A_{n_3} \stackrel{\phi_{n_3, n_4}}{\longrightarrow} A_{n_4} \stackrel{\phi_{n_4, n_5}}{\longrightarrow} \cdots
        \end{equation*}
        and
        \begin{equation*}
            B_{m_1}\stackrel{\psi_{{m_1}, {m_2}}}{\longrightarrow} B_{m_2} \stackrel{\psi_{m_2, m_3}}{\longrightarrow} B_{m_3} \stackrel{\psi_{m_3, m_4}}{\longrightarrow} B_{m_4} \stackrel{\psi_{m_4, m_5}}{\longrightarrow} \cdots.
        \end{equation*}
}
    \end{remark}

\smallskip

We next recall the notion of an inverse limit in the category 
$\cattop$ whose objects are topological spaces and whose morphisms are continuous maps. 
Suppose we have the following inverse system in $\cattop$: $X_1 \stackrel{f_1}{\longleftarrow} X_2 \stackrel{f_2}{\longleftarrow} X_3 \stackrel{f_3}{\longleftarrow} X_4 \stackrel{f_4}{\longleftarrow} \cdots$; 
this means that $X_k$ is a topological space and $f_k$ is a continuous map, $k\in \bb{N}$. 
        The inverse limit of this inverse system, denoted $\limtop X_k$, is the set
        \[
                \Big \lbrace (x_k)_{k \in \mathbb{N}} \in \prod_{k \in \mathbb{N}} X_k: f_k(x_{k+1}) = x_k \text{~for all~} k \in \mathbb{N} \Big\rbrace,
        \]
        equipped with the product topology.
        We note that if each of the spaces $X_k$ is compact and Hausdorff, then $\limtop X_k$ is a compact Hausdorff space.

\medskip

        We denote by $\catcalg$ the category whose objects are unital C*-algebras and whose morphisms are unital *-homomorphisms. Let
        \begin{equation}\label{eq_ccst}
            \cl A_1\stackrel{\pi_1}{\longrightarrow} \cl A_2 \stackrel{\pi_2}{\longrightarrow} \cl A_3 \stackrel{\pi_3}{\longrightarrow} \cl A_4 \stackrel{\pi_4}{\longrightarrow} \cdots
        \end{equation}
        be an inductive system in $\catcalg.$ Let $\prod_{k \in \bb{N}} \cl A_k$ be the space of sequences
        $a = (a_k)_{k \in \bb{N}}$
        such that
        \[
            \norm{a} = \sup\{\norm{a_k}_{\cl A_k} : k \in \mathbb{N}\}
        \]
        is finite. 
                Then $\prod_{k \in \bb{N}} \cl A_k$, equipped with pointwise addition,  multiplication 
        and the norm $\norm{\cdot}$, is a C*-algebra. Define
        \[
            {\cl A}_{\infty}^{0} = \Big\lbrace (a_k)_{k \in \bb{N}} \in \prod_{k \in \bb{N}} \cl A_k :  
            \exists m \in \bb{N} \text{~such that~} \pi_k(a_{k}) = a_{k+1} \text{~for all~} k \geq m \Big\rbrace
        \]
        and
        \[
            N = \Big\lbrace(a_k)_{k\in \bb{N}}\in {\cl A}_{\infty}^{0} : \lim_{k \rightarrow \infty} \norm{a_k}_{\cl A_k} = 0\Big\rbrace.
        \]
        Set 
        ${\cl A}_{\infty} ={\cl A}_{\infty}^{0}/N$ and let 
$q : {\cl A}_{\infty}^{0}\rightarrow {\cl A}_{\infty}$ be the canonical quotient map.
        Let $\pi_{k,\infty}^{0} : \cl A_k\rightarrow {\cl A}_{\infty}^{0}$ be the (linear)
        map given by $\pi_{k,\infty}^{0} (a) = (b_i)_{i \in \bb{N}}$, where
        \[
        b_i =
        \begin{cases}
        0 &\text{~if~} i < k \\
        \pi_{k,i}(a) &\text{~if~} i \geq k, \\
        \end{cases}
        \]
        and let $\pi_{k,\infty} = q \circ \pi_{k,\infty}^{0}$.
        We note that
            $\cl A_{\infty} = \cup_{k \in \bb{N}}\pi_{k,\infty}(\cl A_k)$
        and it is possible to show that
            $\norm{\pi_{k,\infty}(a_k)}_{\cl A_{\infty}} = 
            \lim_{m \rightarrow \infty} \norm{\pi_{k,m}(a_k)}_{\cl A_m}$ for any $a_k \in \cl A_k$.
        Let $\lcalg{\cl A}$ be the completion of~${\cl A}_{\infty}$;
        then $\lcalg{\cl A}$ is an inductive limit of the inductive system (\ref{eq_ccst}) in $\catcalg$ \cite[Section II.8.2]{Blackadar}. 
        Following our general notation, we will denote it by $\limcalg \cl A_k$.

        \begin{remark} \label{rem:universal prop Calg}
{\rm            If each $\pi_{k}$ is injective then $\pi_{k,\infty}$ is injective. Indeed, 
            suppose $\pi_{k,\infty}(a_k) = 0$; then
            $\norm{a_k}_{\cl A_k} = \lim_{m \rightarrow \infty}\norm{a_k}_{\cl A_k} = \lim_{m \rightarrow \infty} \norm{\pi_{k,m}(a_k)}_{\cl A_m} =0$ and therefore $a_k = 0$. }
        \end{remark}

    \begin{remark} \label{ex:InductiveLimitC(X)}
    {\rm 
      Let $X_1 \stackrel{f_1}{\longleftarrow} X_2 \stackrel{f_2}{\longleftarrow} X_3 \stackrel{f_3}{\longleftarrow} X_4 \stackrel{f_4}{\longleftarrow} \cdots$ be an inverse system in $\cattop$ such that each $X_k$ is compact and Hausdorff.
      Let  
      $C(X_1) \stackrel{\phi_1}{\longrightarrow} C(X_2) \stackrel{\phi_2}{\longrightarrow} C(X_3) \stackrel{\phi_3}{\longrightarrow} C(X_4) \stackrel{\phi_4}{\longrightarrow} \cdots$
      be the associated inductive system in~$\catcalg$. 
      We have that $\limcalg C(X_i)$ is unitally *-isomorphic to
      the C*-algebra $C(\limtop X_k)$ (see \cite[II.8.2.2]{Blackadar}).
      }
    \end{remark}






\section{Inductive limits of AOU spaces} \label{chapt:inductive limit AOU}

    We begin this section with the construction of
    the inductive limit in the category~$\catou$. 
    In Section~\ref{sec:inverse limit of state spaces}, we 
    identify the state space of such an inductive limit as the inverse limit of the state spaces of the 
    intermediate ordered *-vector spaces. Finally, in Section~\ref{sect:inductive limits AOU spaces}, we
consider inductive limits in the category $\cataou$ of AOU spaces.

\subsection{Inductive limits in the category {\bf OU}}

    Let $(V_k,V_k^+,e_k)_{k\in \bb{N}}$, be a sequence of ordered *-vector spaces with order units and let 
    $\phi_k : V_k\to V_{k+1}$ be a 
unital positive map, $k\in \bb{N}$; thus,
    \begin{equation}\label{eq_seq}
        V_1\stackrel{\phi_1}{\longrightarrow} V_2 \stackrel{\phi_2}{\longrightarrow} V_3 \stackrel{\phi_3}{\longrightarrow} V_4 \stackrel{\phi_4}{\longrightarrow} \cdots
    \end{equation}
    is an inductive system in $\catou$.
        We let
    $$V_{\infty}^{0} = \Big\lbrace(x_k)_{k\in \bb{N}} \in \prod_{k \in \bb{N}} V_k: \ \exists \ m \mbox{ such that } \phi_k(x_k) = x_{k+1}\text{~for all~} k\geq m\Big\rbrace$$
    and
    \begin{equation}\label{eq:N0}
    N^0_{(V_k)} = \big\lbrace(x_k)_{k\in \bb{N}}\in V_{\infty}^{0} : \ \exists \ m \mbox{ such that } x_k = 0 \text{~for all~}  k\geq m\big\rbrace.
    \end{equation}
    We simplify the notation and write $N^0$ in the place of $N^0_{(V_k)}$, when the context is clear. 
Clearly, $N^0$ is a subspace of $V_{\infty}^{0}$. We set
    $$\lou{V} =V_{\infty}^{0}/N^0,$$
    let $q_0 : V_{\infty}^{0}\rightarrow \lou{V}$ be the canonical quotient map
    and let $\phi_{k,\infty}^{0} : V_k\rightarrow V_{\infty}^{0}$ be the (linear)
    map given by $\phi_{k,\infty}^{0} (x) = (y_i)_{i \in \bb{N}}$ where
    \[
        y_i =
        \begin{cases}
        0 &\text{~if~} i < k \\
        \phi_{k,i}(x) &\text{~if~} i \geq k. \\
        \end{cases}
    \]
    Let
\begin{equation}\label{eq_kinfty}
    \ddot{\phi}_{k,\infty} = q_0 \circ \phi_{k,\infty}^{0};
\end{equation}
thus,  $\ddot{\phi}_{k,\infty}$ is a linear map from $V_k$ into $\lou{V}$.
Since 
$\phi_{k,\infty}^0 = \phi_{l,\infty}^0 \circ \phi_{k,l}$, we have that 
\begin{equation}\label{eq_kl0et}
\ddot{\phi}_{k,\infty} = \ddot{\phi}_{l,\infty} \circ \phi_{k,l}, \ \ \ k < l.
\end{equation}
        Note that
        \begin{equation}\label{eq_b}
            \lou{V} = \bigcup_{k \in \mathbb{N}}\ddot{\phi}_{k, \infty}(V_k).
        \end{equation}

    \begin{remark} \label{phik,infty behaves}
    {\rm
      Let $x_k \in V_k$ and $x_l\in V_l$; 
      then $\ddot{\phi}_{k,\infty}(x_k) = \ddot{\phi}_{l,\infty}(x_l)$ if and only if there exists $m>\max \{k, l\}$ 
      such that $\phi_{k,m}(x_k) =  \phi_{l,m}(x_l)$. }
    \end{remark}

    If $x_k\in V_k$ and $x_l\in V_l$ are such that
    $\ddot{\phi}_{k,\infty}(x_k) = \ddot{\phi}_{l,\infty}(x_l)$, 
    choose $m >\max \{k, l\}$ such that $\phi_{k,m}(x_k) =  \phi_{l,m}(x_l)$. 
Then 
$$\phi_{k,m}(x_k^*) = \phi_{k,m}(x_k)^* = \phi_{l,m}(x_l)^* = \phi_{l,m}(x_l^*).$$ 
Therefore, $\ddot{\phi}_{k,\infty}(x_k^*) = \ddot{\phi}_{l,\infty}(x_l^*)$, and 
we can define an involution on 
$\lou{V}$ by letting 
$\ddot{\phi}_{k,\infty}(x_k)^* \stackrel{def}{=} \ddot{\phi}_{k,\infty}(x_k^*)$. 
It follows that $\ddot{\phi}_{k,\infty}(x_k) \in (\lou{V})_h$ if and only if there exists $m>k$ such that $\phi_{k,m}(x_k) \in (V_m)_h$.

    Let
    $${\lou{V}}^{+} = \big\lbrace\ddot{\phi}_{k,\infty}(x_k) :
    x_k\in V_k  \mbox{ and there exists } m\geq k \mbox{ with } 
    \phi_{k,m}(x_k) \in V_m^+\big\rbrace. $$
To show that ${\lou{V}}^{+}$ is well-defined, 
suppose that $x_k\in V_k$ and $x_l\in V_l$ are such that 
$\ddot{\phi}_{k,\infty}(x_k) = \ddot{\phi}_{l,\infty}(x_l)$, and that $m\geq k$ 
is such that $\phi_{k,m}(x_k) \in V_m^+$. 
Let $p$ be such that $\phi_{k,p}(x_k) = \phi_{l,p}(x_l)$ and 
$q = \max\{m,p\}$. Since $\phi_{m,q}$ is positive, we have
$$\phi_{l,q}(x_l) = \phi_{k,q}(x_k) = \phi_{m,q}\circ \phi_{k,m}(x_k)\in V_q^+.$$

    \begin{lemma} \label{lem:positive cone}
    We have that
    
    (i) \  ${\lou{V}}^{+}$ is a cone in $({\lou{V}})_h$, and
    
    (ii) ${\lou{V}}^{+}\cap (-{\lou{V}}^{+}) = \{0\}$.
    \end{lemma}

    \begin{proof}
(i) 
Let $x_k\in V_k$ be such that 
$\ddot{\phi}_{k,\infty}(x_k) \in {\lou{V}}^{+}$. 
Then there exists $m>k$ such that $\phi_{k,m}(x_k) \in V_m^+ \subseteq {(V_m)}_h$,
and thus $\ddot{\phi}_{k,\infty}(x_k) \in {(\lou{V})}_{h}$. 
If $r \in [0, \infty)$ then $\phi_{k,m}(rx_k)= r\phi_{k,m}(x_k) \in V_m^+$, hence  $r\ddot{\phi}_{k,\infty}(x_k) = \ddot{\phi}_{k,\infty}(rx_k) \in {\lou{V}}^{+}$. 
If $\ddot{\phi}_{k,\infty}(x_k), \ddot{\phi}_{l,\infty}(x_l) \in {\lou{V}}^{+}$ then there exist $m_1>k$ and $m_2>l$ such that $\phi_{k,m_1}(x_k) \in V_{m_1}^+$ and $\phi_{l,m_2}(x_l) \in V_{m_2}^+$. Set $m= \max{\{m_1, m_2\}}$; then $\phi_{k,m}(x_k)+ \phi_{l,m}(x_l) \in V_m^+$. Therefore $\ddot{\phi}_{k,\infty}(x_k)+\ddot{\phi}_{l,\infty}(x_l)  = \ddot{\phi}_{m,\infty}(\phi_{k,m}(x_k)+ \phi_{l,m}(x_k) )\in {\lou{V}}^{+}$.
        
(ii)         Let 
$\ddot{\phi}_{k,\infty}(x_k) \in {\lou{V}}^{+} \cap (-{\lou{V}}^{+})$
for some $x_k\in V_k$. 
Then there exist $m_1, m_2 \geq k$ such that $\phi_{k,m_1}(x_k) \in V_{m_1}^+$ and  $-\phi_{k,m_2}(x_k) \in V_{m_2}^+$. 
Choose $m > \max{\{m_1, m_2\}}$; then 
$\phi_{k,m}(x_k) \in V_m^+\cap (-V_m^+)$, so 
$\phi_{k,m}(x_k) =0$, and 
hence $\ddot{\phi}_{k,\infty}(x_k) = \ddot{\phi}_{m,\infty} \circ \phi_{k,m}(x_k) = 0$.
    \end{proof}

    Observe that $\ddot{\phi}_{k,\infty}(x_k) \leq \ddot{\phi}_{l,\infty}(x_l)$ if and only if there exists $m>\max\{k,l\}$ such that $\phi_{k,m}(x_k) \leq \phi_{l,m}(x_l)$. 
    Furthermore, (\ref{eq_kl0et}) implies that 
    \begin{equation}\label{eq_c}
        {\lou{V}}^{+} = \bigcup_{k \in \mathbb{N}}\ddot{\phi}_{k, \infty}(V_k^+).
    \end{equation}

    By Remark~\ref{phik,infty behaves} and the unitality of 
    the connecting maps, 
    $\ddot{\phi}_{k,\infty}(e_k) = \ddot{\phi}_{l,\infty}(e_l)$ for all $k,l \in \bb{N}$. 
    Set $\ddot{e}_{\infty} = \ddot{\phi}_{k,\infty}(e_k)$ (for any $k\in \bb{N}$).
    We next show that $\ddot{e}_{\infty}$ is an order unit for $({\lou{V}},{\lou{V}}^{+})$.

    \begin{proposition}\label{prop:V_infty is ou}
        The triple $(\lou{V}, {\lou{V}}^{+}, \ddot{e}_{\infty})$ is an ordered *-vector space with order unit. 
        Furthermore, $\ddot{\phi}_{k,\infty}: V_k \rightarrow \lou{V}$ is a unital positive map such that $\ddot{\phi}_{k+1,\infty} \circ \phi_k = \ddot{\phi}_{k,\infty}$, $k \in \bb{N}$.
    \end{proposition}

    \begin{proof}
        To prove that $(\lou{V}, {\lou{V}}^{+}, \ddot{e}_{\infty})$ is an ordered 
        *-vector space with order unit, it suffices, by Lemma \ref{lem:positive cone}, 
        to show that $\ddot{e}_{\infty}$ is an order unit. Suppose 
        that $x_k\in V_k$ is such that $\ddot{\phi}_{k,\infty}(x_k) \in (\lou{V})_h$; 
        then there exists $m>k$ such that $\phi_{k,m}(x_k) \in (V_{m})_h$. Since $e_m$ is an order unit for $V_m$, there exists $r_m>0$ such that $\phi_{k,m}(x_k) \leq r_me_m = \phi_{k,m}(r_me_k)$. 
        By (\ref{eq_c}), 
        \[
            \ddot{\phi}_{k,\infty}(x_k) = \ddot{\phi}_{m,\infty}\circ \phi_{k,m}(x_k) \leq \ddot{\phi}_{m,\infty}(r_m e_m) = r_m \ddot{\phi}_{m,\infty}(e_m) = 
            r_m \ddot{e}_{\infty}.
        \]
The identity $\ddot{\phi}_{k+1,\infty} \circ \phi_k = \ddot{\phi}_{k,\infty}$, $k \in \bb{N}$,
is a special case of (\ref{eq_kl0et}). 
    \end{proof}

        So far we have ascertained that $(\ddot{V}_{\infty}, \{\ddot{\phi}_{k,\infty}\}_{k \in \bb{N}})$ is a suitable candidate for the inductive limit in $\catou$ of the inductive system (\ref{eq_seq}). 
        Theorem~\ref{thm:universal property *vectorspace2} will verify that this pair does indeed satisfy the universal property of the inductive limit. 
        First we take note of the special case when the maps in the inductive system are unital order isomorphisms.

    \begin{remark} \label{rem:Eachphi_kOIThenphi_k,inftyOI}
       Let
            $V_1\stackrel{\phi_1}{\longrightarrow} V_2 \stackrel{\phi_2}{\longrightarrow} V_3 \stackrel{\phi_3}{\longrightarrow} V_4 \stackrel{\phi_4}{\longrightarrow} \cdots$
        be an inductive system in $\catou$ such that $\phi_k$ is an order isomorphism onto its image for all $k \in \bb{N}$. Then $\ddot{\phi}_{k,\infty}$ is an order isomorphism onto its image for all $k \in \bb{N}$.
            \end{remark}
            
\begin{proof}
        Let $k \in \bb{N}$ and suppose $\ddot{\phi}_{k,\infty}(x_k) = 0$ for some $x_k \in V_k$. 
        Then there exists $m>k$ such that $\phi_{k,m}(x_k) = 0$. 
        By the assumption, 
        $\phi_{k,m}$ is injective and it follows that $x_k = 0.$ Thus, $\ddot{\phi}_{k,\infty}$ is injective. 
        
Suppose $\ddot{\phi}_{k,\infty}(x_k) \in {\lou{V}}^{+}$; then there exists $m > k$ such that $\phi_{k,m}(x_k) \in V_m^+.$ Since $\phi_{k, m}$ is 
an order isomorphism onto its image, $x_k \in V_k^+$. 
\end{proof}

    \begin{theorem} \label{thm:universal property *vectorspace2}
    The triple $(\lou{V}, \{\ddot{\phi}_{k,\infty}\}_{k \in \bb{N}}, \ddot{e}_{\infty})$ is an inductive limit of the 
    inductive system
    $V_1\stackrel{\phi_1}{\longrightarrow} V_2 \stackrel{\phi_2}{\longrightarrow} V_3 \stackrel{\phi_3}{\longrightarrow} V_4 \stackrel{\phi_4}{\longrightarrow} \cdots$ in $\catou$. 
    \end{theorem}

        \begin{proof}
        We check that $(\lou{V}, \{\ddot{\phi}_{k,\infty}\}_{k \in \bb{N}})$ satisfies the universal property of the inductive limit. 
      Suppose $(W, \lbrace \psi_k \rbrace_{k \in \mathbb{N}})$ is a pair consisting of an ordered 
      *-vector space and a family of unital positive maps $\psi_k:V_k \rightarrow W$ such that $\psi_{k+1} \circ \phi_{k} = \psi_k$ for all $k \in \mathbb{N}$. Let $k,l \in \mathbb{N}$, $x_k \in V_k$, $x_l \in V_l$ and suppose that $\ddot{\phi}_{k,\infty}(x_k) = \ddot{\phi}_{l,\infty}(x_l)$. 
      By Remark \ref{phik,infty behaves}, there 
      exists $m > \max \{k,l\}$ such that $\phi_{k,m}(x_k) = \phi_{l,m}(x_l)$. Consequently $\psi_k (x_k) = \psi_m \circ \phi_{k,m}(x_k) = \psi_m \circ \phi_{l,m}(x_l) = \psi_l (x_l).$ 
      Let $\ddot{\psi} : \lou{V} \rightarrow W$ be given by 
      $\ddot{\psi} \circ \ddot{\phi}_{k,\infty}  = \psi_k$;
      since $\lou{V} = \cup_{k \in \mathbb{N}}\ddot{\phi}_{k,\infty}(V_k)$, the map 
      $\ddot{\psi}$ is well-defined.
        Since $\psi_k$ is unital and  $\ddot{\psi} \circ \ddot{\phi}_{k,\infty}(e_k) = \psi_k (e_k)$, 
        the map $\ddot{\psi}$ is unital. Suppose that $\ddot{\phi}_{k,\infty}(x_k) \in \ddot{V}_{\infty}^+$; 
        then there exists $m>k$ such that $\phi_{k,m}(x_k) \in V_m^+$. Since $\psi_m$ is positive and
        $\ddot{\psi} \circ \ddot{\phi}_{k,\infty}(x_k) = \psi_k (x_k) = \psi_m \circ \phi_{k,m}(x_k)$, 
        we have that $\ddot{\psi} (\ddot{\phi}_{k,\infty}(x_k))\in W^+$ and hence $\ddot{\psi}$ is positive. 
            \end{proof}

According to our general notation, denote by $\limou V_k$ the inductive limit 
$(\lou{V}, \{\ddot{\phi}_{k,\infty}\}_{k \in \bb{N}})$. 

    \begin{remark} \label{rem:universal property *vectorspace}
{\rm      
Let $(\{V_k\}_{k \in \mathbb{N}}, \{\phi_k\}_{k \in \bb{N}})$ and $(\{W_k\}_{k \in \mathbb{N}}, \{\psi_k\}_{k \in \bb{N}})$ be 
inductive systems in $\catou$ and let $\{\theta_k\}_{k \in \mathbb{N}}$ be a sequence of unital positive maps such that the following diagram commutes:
      \begin{equation}\label{eq:limit in ou square diagram}
    \begin{CD}
        V_1 @>{\phi_1}>> V_2 @>{\phi_2}>> V_3 @>{\phi_3}>> V_4 @>{\phi_4}>> \cdots \\
        @V{\theta_1}VV @V{\theta_2}VV @V{\theta_3}VV @V{\theta_4}VV \\
        W_1 @>{\psi_1}>> W_2 @>{\psi_2}>> W_3 @>{\psi_3}>> W_4 @>{\psi_4}>> \cdots.
    \end{CD}
  \end{equation}
    It follows from Theorem~\ref{thm:universal property *vectorspace2} and Theorem~\ref{thm:universal property infinitely many maps} that there exists a unique unital positive map $\ddot{\theta} : \limou V_k \rightarrow \limou W_k$ such that $\ddot{\theta} \circ \ddot{\phi}_{k,\infty} = \ddot{\psi}_{k,\infty} \circ \theta_k$ for all $k \in \bb{N}$. 
        \begin{enumerate}[\rm(i)]
            \item If $\theta_k$ is injective for every $k \in \mathbb{N}$ then $\ddot{\theta}$ is injective. Indeed, if 
            $x_k\in V_k$ and 
            $\ddot{\theta} \circ \ddot{\phi}_{k,\infty}(x_k) = 0$,  then
            $\ddot{\psi}_{k,\infty} \circ \theta_k (x_k) = 0$. Therefore there exists $m>\max\{k,l\}$ such that $\psi_{k,m} \circ \theta_k (x_k) = 0$. 
            Since (\ref{eq:limit in ou square diagram}) commutes, $\theta_m \circ \phi_{k,m} (x_k) = 0$. 
            Since $\theta_m$ is injective, $\phi_{k,m} (x_k) =  0$ and hence 
            $\ddot{\phi}_{k,\infty}(x_k) = 0$.

            \item If $\theta_k$ is an order isomorphism onto its image for every $k \in \mathbb{N}$ then $\ddot{\theta}$ is an order isomorphism onto its image. 
            Indeed, suppose that $\ddot{\theta} \circ \ddot{\phi}_{k,\infty}(x_k) \in (\limou W_k)^+$ for some $x_k \in V_k$. Then $\ddot{\psi}_{k,\infty} \circ \theta_k(x_k) \in (\limou W_k)^+$ and it follows that there exists $m >k$ such that $\psi_{k,m} \circ \theta_k(x_k) \in W_m^+$. Since (\ref{eq:limit in ou square diagram}) commutes, this implies that $\theta_m \circ \phi_{k,m}(x_k) \in W_m^+$. Since $\theta_m$ is an order isomorphism, it follows that $\phi_{k,m}(x_k) \in V_m^+$, and hence $\ddot{\phi}_{k,\infty}(x_k) \in (\limou V_k)^+$.
        \end{enumerate}
        }
    \end{remark}

\subsection{The state space of the inductive limit in ${\bf OU}$}\label{sec:inverse limit of state spaces}

    Given the inductive system (\ref{eq_seq}), one can \lq\lq reverse the arrows'' to obtain a sequence
    \begin{equation*}\label{eq_seqstate} \index{state space}
        V_1'\stackrel{\phi_1'}{\longleftarrow} V_2' \stackrel{\phi_2'}{\longleftarrow} V_3'
        \stackrel{\phi_3'}{\longleftarrow} V_4' \stackrel{\phi_4'}{\longleftarrow} \cdots
    \end{equation*}
    of dual spaces and continuous maps
    (here we use the fact that unital positive maps between AOU spaces 
    are automatically continuous in the order norm \cite[Theorem 4.22]{PaulsenTomforde}).
    Since the maps $\phi_k$ are unital, we have that $\phi_k'(S(V_{k+1}))\subseteq S(V_k)$
    for all $k \in \bb{N}$,
    and thus we obtain the following inverse system in~$\cattop$:
    \begin{equation}\label{eq_seqs45}
        S(V_1)\stackrel{\phi_1'}{\longleftarrow} S(V_2) \stackrel{\phi_2'}{\longleftarrow} S(V_3)
        \stackrel{\phi_3'}{\longleftarrow} S(V_4) \stackrel{\phi_4'}{\longleftarrow} \cdots.
    \end{equation}

    \begin{proposition}\label{l_si}
        Let $V_1\stackrel{\phi_1}{\longrightarrow} V_2 \stackrel{\phi_2}{\longrightarrow} V_3 \stackrel{\phi_3}{\longrightarrow} V_4 \stackrel{\phi_4}{\longrightarrow} \cdots$
        be an inductive system in $\catou$.
        The state space $S(\limou V_k)$ is topologically homeomorphic to the inverse limit
        $\limtop S(V_k)$.
    \end{proposition}
    \begin{proof}
        Let $f \in S(\limou V_k)$ and define 
        $f_k :V_k \rightarrow \mathbb{C}$ by letting 
        $\angles{f_k}{x_k}= \angles{f}{\ddot{\phi}_{k,\infty}(x_k)}$, $x_k \in V_k$. For $x_k\in V_k$, 
        we have
        \begin{eqnarray*}
      \angles{f_{k+1} \circ \phi_k}{x_k} 
      & = & 
      \angles{f_{k+1}}{\phi_k(x_k)} = \angles{f}{\ddot{\phi}_{k+1,\infty} \circ \phi_k(x_k)}\\
      & = & \angles{f}{\ddot{\phi}_{k,\infty}(x_k)} = \angles{f_k}{x_k}.
         \end{eqnarray*}
        Therefore $\phi_{k}'(f_{k+1}) = f_k$ and so $(f_k)_{k \in \mathbb{N}} \in \limtop S(V_k)$. 
        Define a map $\theta : S(\limou V_k) \rightarrow \limtop S(V_k)$ by letting $\theta(f)=(f_k)_{k \in \bb{N}}$. 
        
        Suppose $f, g \in S(\limou V_k)$ are 
        such that $\theta(f)=\theta(g)$; that is, $f_k = g_k$ for all $k \in \bb{N}$. If $x_k\in V_k$ then
        \[
                \angles{f}{\ddot{\phi}_{k,\infty}(x_k)}  = \angles{f_k}{x_k} =  \angles{g_k}{x_k} = \angles{g}{\ddot{\phi}_{k,\infty}(x_k)}.
        \]
        By (\ref{eq_b}), $f=g$ and hence $\theta$ is injective.  
        
        Given a sequence $(f_k)_{k \in \bb{N}} \in \limtop S(V_k)$, define an element $f : {\lou{V}} \to \bb{C}$
        by setting $\angles{f}{\ddot{\phi}_{k,\infty}(x_k)} = \angles{f_k}{x_k}$, $x_k\in V_k$. Observe that $f$ is well-defined, 
        for if $\ddot{\phi}_{k,\infty}(x_k) = \ddot{\phi}_{l,\infty}(x_l)$ for some $x_k \in V_k$ and $x_l \in V_l$ then,
        by Remark \ref{phik,infty behaves}, 
        there exists $m>\max\{k,l\}$ such that $\phi_{k,m}(x_k) = \phi_{l,m}(x_l)$.
        Hence
        \[
        \begin{split}
        \angles{f_k}{x_k} = \angles{f_m \circ \phi_{k, m}}{x_k} &= \angles{f_m} {\phi_{k, m}(x_k)} \\& = \angles{f_m}{\phi_{l, m}(x_l)}= \angles{f_m \circ \phi_{l, m}}{x_l} = \angles{f_l}{x_l}.\\
      \end{split}
        \]
        Suppose that $x\in (\limou V_k)^+$. By (\ref{eq_c}), there exist $k\in \bb{N}$ and $x_k\in V_k^+$ such that 
        $x = \ddot{\phi}_{k,\infty}(x_k)$, and hence 
        $$\angles{f}{\ddot{\phi}_{k,\infty}(x_k)} = \angles{f_k}{x_k} \geq 0,$$ 
 showing that $f$ is positive. 
 Furthermore, $\angles{f}{\ddot{e}_{\infty}} = \angles{f_k}{e_k} = 1$ and thus $f \in S(\limou V_k)$. 
 Since $\theta(f) = (f_k)_{k\in \bb{N}}$, we conclude that $\theta$ is surjective.

        Finally, we prove that $\theta$ a homeomorphism. 
        Suppose that $(f^\lambda)_{\lambda \in \Lambda} \in S (\limou V_k)$ is a net such that 
        $f^\lambda \rightarrow_{\lambda \in \Lambda} f$ for some $f\in S (\limou V_k)$.
        Write
        \[
\theta(f) = (f_k)_{k \in \bb{N}}  \text{~and~}     \theta(f^\lambda) = (f^\lambda_k)_{k \in \bb{N}}, \ \ \ {\lambda \in \Lambda}.
        \]
        Since $\limtop S(V_k)$ is equipped with the product topology,
        \[
            ((f^\lambda_k)_{k \in \bb{N}})_{\lambda \in \Lambda} \rightarrow\mbox{}_{\lambda \in \Lambda} (f_k)_{k \in \bb{N}}
        \text{~if and only if~}
        (f^\lambda_k)_{\lambda \in \Lambda} \rightarrow\mbox{}_{\lambda \in \Lambda} f_k
        \text{~for all~}
        k \in \bb{N}.
        \]
        If $k \in \bb{N}$ and $x_k \in V_k$ then
        \[
            \angles{{f^{\lambda}_k}}{x_k} = \angles{{f_\lambda}}{\ddot{\phi}_{k,\infty}(x_k)} 
            \rightarrow\mbox{}_{\lambda \in \Lambda} \angles{f}{\ddot{\phi}_{k,\infty}(x_k)} = \angles{f_k}{x_k}.
        \]
        It follows that $\theta(f_\lambda) \rightarrow_{\lambda \in \Lambda} \theta(f)$ and so $\theta$ is continuous.

        Suppose that 
        $((f^\lambda_k)_{k \in \bb{N}})_{\lambda \in \Lambda} \in \limtop S(V_k)$ 
        is such that $(f^\lambda_k)_{k \in \bb{N}} \rightarrow_{\lambda \in \Lambda} (f_k)_{k \in \bb{N}}$. 
        For each $k \in \bb{N}$, $(f^\lambda_k)_{\lambda \in \Lambda} \rightarrow_{\lambda \in \Lambda} f_k$. Now,
        \[
            \theta^{-1}((f^\lambda_k)_{k \in \bb{N}}) = f_\lambda,  ~\text{~where~}~\angles{f_\lambda}{\ddot{\phi}_{k,\infty}(x_k)} = \angles{f^\lambda_k}{x_k}.
        \]
        If $x_k\in V_k$ then
        \[
             \angles{f_\lambda}{\ddot{\phi}_{k,\infty}(x_k)} = \angles{f^\lambda_k}{x_k} 
             \rightarrow\mbox{}_{\lambda \in \Lambda} \angles{f_k}{x_k} = \angles{f}{\ddot{\phi}_{k,\infty}(x_k)}.
        \]
By (\ref{eq_b}), $\theta^{-1}((f^\lambda_k)_{k \in \bb{N}}) \rightarrow_{\lambda \in \Lambda} \theta^{-1}((f_k)_{k \in \bb{N}})$ 
and therefore $\theta$ is a homeomorphism.
    \end{proof}

\subsection{Inductive limits in the category {\bf AOU}} \label{sect:inductive limits AOU spaces}

Let $(V_k,V_k^+,e_k)_{k \in \bb{N}}$ be a sequence of 
Archimedean order unit spaces and 
    \begin{equation}\label{eq:seqAOU}
        V_1\stackrel{\phi_1}{\longrightarrow} V_2 \stackrel{\phi_2}{\longrightarrow} V_3 \stackrel{\phi_3}{\longrightarrow} V_4 \stackrel{\phi_4}{\longrightarrow} \cdots
    \end{equation}
be an inductive system in the category $\cataou$. 
Recall that this means that $\phi_k : V_k\to V_{k+1}$ is a unital positive map, $k\in \bb{N}$. 
We may apply the forgetful functor $\bf F: \cataou \rightarrow \catou$ and consider 
the inductive limit $\limou {{\bf F}(V_k)}$. This is not necessarily an AOU space; we shall see in this subsection that 
its Archimedeanisation is an inductive limit in $\cataou$.

The proof of the following remark is straightforward and we omit it.

    \begin{remark} \label{rem:norm_on_N}
Let $\norm{\cdot}^k$ be an order norm on $V_k$, $k\in \bb{N}$.
For $x_k \in V_k$, we have that 
$\lim_{m \rightarrow \infty}\norm{\phi_{k,m}(x_k)}^m = 0$ if and only if
        \[
        \lim_{m \rightarrow \infty}\norm{{\rm Re} (\phi_{k,m}(x_k))}_h =0 \text{ and } \lim_{m \rightarrow \infty}\norm{{\rm Im} (\phi_{k,m}(x_k))}_h =0.\]
    \end{remark}

Let $\norm{\cdot}^k$ be any order norm on $V_k$ and 
$\norm{\cdot}^{\infty}$ be any order seminorm on $\limou V_k$.
Let 
\begin{equation}\label{eqNagain}
N = \{x\in \limou V_k : \norm{x}^{\infty} = 0\}
\end{equation}
be the kernel of $\norm{\cdot}^{\infty}$.

 \begin{proposition} \label{thm:characterisations of null space}
        Let $x_k\in V_k$ and $x = \ddot{\phi}_{k,\infty}(x_k) \in \limou V_k$.
        The following are equivalent:
        \begin{enumerate}[\rm(i)]
\item $x\in N$;
            \item $\lim_{m \rightarrow \infty} \norm{\phi_{k,m}(x_k)}^m = 0$.
        \end{enumerate}
\end{proposition}
      
\begin{proof}
By Remarks \ref{rem:normouspace0} and \ref{rem:norm_on_N},
we may assume that $x \in (\limou V_k)_h$. 

(i)$\Rightarrow$(ii)
We have that 
        \[
            \inf \lbrace \lambda \geq 0: -\lambda \ddot{e}_{\infty} \leq \ddot{\phi}_{k,\infty}(x_k) \leq \lambda \ddot{e}_{\infty} \rbrace = 0.
        \]
        Let $r > 0$; then there exists $m \in \bb{N}$ such that $-re_l \leq \phi_{k,l}(x_l) \leq re_l$ for all $l \geq m$. Therefore $\norm{\phi_{k,l}(x_k)}^l \leq r$ for all $l \geq m$. 
        Thus, $\lim_{m \rightarrow \infty}\norm{\phi_{k,m}(x_k)}^m = 0$.

(ii)$\Rightarrow$(i)
Assume, 
towards a contradiction, that $\norm{x}^{\infty}=\mu > 0$. 
There exists $m>k$ such that 
$$\inf \lbrace \lambda_l : -\lambda_l e_l \leq \phi_{k,l}(x_k) \leq \lambda_l e_l\rbrace < \frac{\mu}{2}, \ \ \ l \geq m.$$
Therefore, $-\frac{\mu}{2}e_l \leq \phi_{k,l}(x_k)\leq \frac{\mu}{2}e_l$ for all $l \geq m$ and so 
$-\frac{\mu}{2}\ddot{\phi}_{k,\infty}(e_k) \leq \ddot{\phi}_{k,\infty}(x_k)\leq \frac{\mu}{2}\ddot{\phi}_{k,\infty}(e_k)$. 
Thus $\norm{x}^{\infty} \leq \frac{\mu}{2} < \mu$, a contradiction.
    \end{proof}

In view of Proposition \ref{thm:characterisations of null space}, we will refer to 
$N$ defined by (\ref{eqNagain}) 
as the {\it null space of the sequence} $(V_k, V_k^+, e_k)_{k \in \mathbb{N}}$.

Let $(\laou{V},\laou{V}^+,e_{\infty})$ be the Archimedeanisation of $\limou V_k$; 
thus, 
    \[
    \laou{V} = (\limou V_k)/N,
    \]
the involution on $\laou{V}$ is given by 
$(\ddot{\phi}_{k,\infty}(x_k)+N)^* = \ddot{\phi}_{k,\infty}(x_k)^*+N$ 
(for $x_k\in V_k$), 
$$\laou{V}^+ = \{\ddot{\phi}_{k,\infty}(x_k)+N : x_k\in (V_k)_h, k\in \bb{N}, \mbox{ and } $$
$$\ddot{\phi}_{k,\infty}(x_k) + r\ddot{\phi}_{k,\infty}(e_k)  \in \ddot{V}_{\infty}^+, 
\mbox{ for all }  r > 0\},$$
and $e_{\infty} = \ddot{e}_{\infty} + N$.

  \begin{lemma}\label{lem:self adjoint element archimedeanisation}
    Let $x_k\in V_k$. The following are equivalent:
    \begin{enumerate}[\rm(i)]
    \item $\ddot{\phi}_{k,\infty}(x_k) + N \in (V_{\infty})_{h}$;
    \item $\ddot{\phi}_{k,\infty}(x_k) + N = \ddot{\phi}_{k,\infty}({\rm Re}(x_k)) + N$;
    \item $\ddot{\phi}_{k,\infty}(x_k) + N = \ddot{\phi}_{l,\infty}(x_l) + N$ for some $l \in \bb{N}$ and some $x_l \in (V_l)_h$.

    \end{enumerate}
     \end{lemma}

      \begin{proof}
(i)$\Rightarrow$ (ii) Suppose $\ddot{\phi}_{k,\infty}(x_k) + N \in (V_{\infty})_{h}$.
Then $\ddot{\phi}_{k,\infty}(x_k) + N = \ddot{\phi}_{k,\infty}(x_k)^* + N = \ddot{\phi}_{k,\infty}(x_k^*) +N$ and therefore
    \[
    \begin{split}
        \ddot{\phi}_{k,\infty}(x_k) + N &=  \frac{\ddot{\phi}_{k,\infty}(x_k) +\ddot{\phi}_{k,\infty}(x_k^*)}{2} + N \\ &= \ddot{\phi}_{k,\infty}\left(\frac{x_k + x_k^*}{2}\right) + N = \ddot{\phi}_{k,\infty}({\rm Re}(x_k)) + N.
      \end{split}
    \]

(ii)$\Rightarrow$ (iii) is trivial.

(iii)$\Rightarrow$ (i)  
       Suppose $\ddot{\phi}_{k,\infty}(x_k) + N = \ddot{\phi}_{l,\infty}(x_l) + N$ for some $x_l \in (V_l)_h$.
       Then
        \[
        \begin{split}
          (\ddot{\phi}_{k,\infty}(x_k) + N)^* &= (\ddot{\phi}_{l,\infty}(x_l) + N)^* = \ddot{\phi}_{l,\infty}(x_l)^* + N \\ &= \ddot{\phi}_{l,\infty}(x_l^*) + N =\ddot{\phi}_{l,\infty}(x_l) + N = \ddot{\phi}_{k,\infty}(x_k) + N.
        \end{split}
        \]
  \end{proof}

\begin{remark} \label{rem:positive cone V_infty}
{\rm  
We have that 
$$\laou{V}^+ = \{\ddot{\phi}_{k,\infty}(x_k)+N : x_k\in (V_k)_h, k\in \bb{N}, \mbox{ and } $$
$$ \mbox{ for every } r > 0 \mbox{ there exists } m \geq k \mbox{ such that }
\phi_{k,m}(x_k) + r e_m \in V_m^+\}.$$
An element $\ddot{\phi}_{k,\infty}(x_k)+N \in (V_{\infty})_h$
(where $x_k\in (V_k)_h$)  belongs to $V_{\infty}^+$ if and only if for every $r > 0$
there exist $l\in \bb{N}$ and $y_l\in V_l$  such that $\ddot{\phi}_{\l, \infty}(y_l) \in N$ and 
$\ddot{\phi}_{k,\infty}(r e_k + x_k) + \ddot{\phi}_{l,\infty}(y_l) \in \ddot{V}_{\infty}^+$.
Thus, $\ddot{\phi}_{k,\infty}(x_k)+N \in V_{\infty}^+$ if and only if for every $r > 0$
there exist $l\in \bb{N}$ and $y_l\in V_l$ such that $\ddot{\phi}_{l,\infty}(y_l) \in N$, and there exists 
$m > \max\{k,l\}$ with $re_m + \phi_{k,m}(x_k) + \phi_{l,m}(y_l)\in V_m^+$. 
We may assume without loss of generality that $l>k$ and that $y_l \in (V_l)_h$.
}
\end{remark}

    Let $q_V : \lou{V}\rightarrow \laou{V}$ be the canonical quotient map,
    and set
    \[
    \phi_{k,\infty} = q_V\circ \ddot{\phi}_{k,\infty};
    \]
    we have that $\phi_{k,\infty}$ is a unital positive map and 
$$\phi_{k+1,\infty} \circ \phi_k = \phi_{k,\infty}, \ \ \ k\in \bb{N}.$$
        Since $\ddot{V}_{\infty} = \cup_{k\in \bb{N}}\ddot{\phi}_{k,\infty}(V_k)$, we have that
    \[\laou{V} = \cup_{k\in \bb{N}}\phi_{k,\infty}(V_k).\]

The following lemma is certainly well-known; we record it since we were not able to find a precise reference.

\begin{lemma}\label{l_exten}
Let $(V,V^+,e)$ be an AOU space and $W\subseteq V$ be a linear subspace containing $e$. Set $W^+ = W\cap V^+$. 
Then $(W,W^+,e)$ is an AOU space and for every $f\in S(W)$ there exists $g\in S(V)$ such that $g|_{W} = f$. 
\end{lemma}
\begin{proof}
It is straightforward to check that $(W,W^+,e)$ is an AOU space. 
Recall the correspondence between complex functionals on $V$ and real functionals on $V_h$: given 
a real functional $\omega$ on $V_h$, one defines a functional $\tilde{\omega} : V \to \bb{C}$ by letting 
$\tilde{\omega}(x) = \omega({\rm Re}(x)) + i \omega({\rm Im}(x))$, $x\in V$. 
The second statement now follows from the fact that, by
\cite[Proposition~3.11]{PaulsenTomforde}, $\omega$ is positive if and only if $\tilde{\omega}$ is positive, 
and by \cite[Corollary~2.15]{PaulsenTomforde}, every positive real functional on a real ordered vector 
space can be extended to a positive real functional on a larger space. 
\end{proof}

    \begin{proposition} \label{prop:Eachphi_kOIthenN=0}
        Let
            $V_1\stackrel{\phi_1}{\longrightarrow} V_2 \stackrel{\phi_2}{\longrightarrow} V_3 \stackrel{\phi_3}{\longrightarrow} V_4 \stackrel{\phi_4}{\longrightarrow} \cdots$
        be an inductive system in $\cataou$ such that $\phi_k$ is an order isomorphism onto its image for each $k\in \bb{N}$.
        Then $N = \{0\}$
        and~$\phi_{k,\infty}$ is a unital order isomorphism onto its image for all $k \in \bb{N}$.
            \end{proposition}

    \begin{proof}
        Suppose that $x_k\in V_k$ and 
        $\ddot{\phi}_{k,\infty}(x_k) \in N$. 
        By Proposition \ref{thm:characterisations of null space}, 
        $\lim_{m \rightarrow \infty}\norm{\phi_{k,m}(x_k)}=0$. Since each $\phi_k$ is an order isomorphism onto $\phi_{k}(V_k)$, 
        using Lemma \ref{l_exten} we obtain that
        $\norm{\phi_{k,m}(x_k)} = \norm{x_k}$ for all $m \geq k$ and so $x_k=0$. 
        Thus, $\ddot{\phi}_{k,\infty}(x_k) = 0$. 
        It now follows that 
        $\phi_{k,\infty} = \ddot{\phi}_{k,\infty}$ and therefore, by Remark~\ref{rem:Eachphi_kOIThenphi_k,inftyOI}, 
        $\phi_{k,\infty}$ is a unital order isomorphism onto its image, $k \in \bb{N}$.
    \end{proof}

    \begin{theorem} \label{thm:universal property AOU 2}
    The triple $(V_{\infty}, \{\phi_{k,\infty}\}_{k \in \bb{N}},e_{\infty})$ is the inductive limit
of the inductive system
            $V_1\stackrel{\phi_1}{\longrightarrow} V_2 \stackrel{\phi_2}{\longrightarrow} V_3 \stackrel{\phi_3}{\longrightarrow} V_4 \stackrel{\phi_4}{\longrightarrow} \cdots$
            in the category $\cataou$. 
\end{theorem}
                
\begin{proof}
        Suppose $(W, \lbrace \psi_k \rbrace_{k \in \mathbb{N}})$ is a pair consisting of an AOU space and a family of unital positive maps $\psi_k : V_k \rightarrow W$ such that $\psi_{k+1} \circ \phi_{k} = \psi_k$ for all $k \in \mathbb{N}$.
        By Theorem \ref{thm:universal property *vectorspace2}, there exists a unique unital positive map $\ddot{\psi} : \limou V_k \rightarrow W$ such that
            $\ddot{\psi} \circ \ddot{\phi}_{k,\infty} = \psi_k$
        for all $k \in \bb{N}$. By Theorem~\ref{lem:archimedeanisation of positive map}, 
        there exists a unique unital positive map $\psi: \laou{V} \rightarrow W$
        such that $\psi \circ q_V = \ddot{\psi}.$
        Therefore $\psi \circ \phi_{k,\infty} = \psi \circ q_V \circ \ddot{\phi}_{k,\infty} = \ddot{\psi} \circ \ddot{\phi}_{k,\infty} = \psi_k$
        for all $k \in \bb{N}$
        and the proof is complete.
\end{proof}

We recall that, according to our general notation for 
inductive limits, $\limaou V_k$ will henceforth stand for the AOU space
$(V_{\infty},$ $\{\phi_{k,\infty}\}_{k \in \bb{N}},e_{\infty})$.

        \begin{remark} \label{rem:universal property AOU}
{\rm          For each $k \in \bb{N}$, let $(V_k, V_k^+, e_k)$ and  $(W_k, W_k^+, f_k)$ be AOU spaces such that $(\{V_k\}_{k\in \bb{N}},\{\phi_k\}_{k \in \bb{N}})$ and $(\{W_k\}_{k\in \bb{N}},\{\psi_k\}_{k \in \bb{N}})$ are inductive systems and let $\{\theta_k\}_{k \in \mathbb{N}}$ be a sequence of unital positive maps such that the following diagram commutes:
          \begin{equation} \label{eq:chtp1a}
        \begin{CD}
            V_1 @>{\phi_1}>> V_2 @>{\phi_2}>> V_3 @>{\phi_3}>> V_4 @>{\phi_4}>> \cdots \\
            @V{\theta_1}VV @V{\theta_2}VV @V{\theta_3}VV @V{\theta_4}VV \\
            W_1 @>{\psi_1}>> W_2 @>{\psi_2}>> W_3 @>{\psi_3}>> W_4 @>{\psi_4}>> \cdots.
        \end{CD}
        \end{equation}
        It follows from Theorem~\ref{thm:universal property AOU 2} and Theorem~\ref{thm:universal property infinitely many maps} that there exists a unique unital positive map $\theta : \limaou V_k \rightarrow \limaou W_k$ such that $\theta \circ \phi_{k,\infty} = \psi_{k,\infty} \circ \theta_k$ for all $k \in \bb{N}$. 
            If $\theta_k$ is an order isomorphism onto its image for each $k\in \bb{N}$
            then $\theta$ is injective. 
            Indeed, if $x_k\in V_k$ and 
            $\theta\circ \phi_{k,\infty}(x_k) = 0$ 
            then
            $\psi_{k,\infty} \circ \theta_k (x_k) = 0$. 
            By Proposition \ref{thm:characterisations of null space},             
              $\lim_{m \rightarrow \infty} \norm{\psi_{k,m} \circ \theta_k(x_k)} =0$ and,
            since (\ref{eq:chtp1a}) commutes,
            $\lim_{m \rightarrow \infty} \norm{\theta_m ({\phi_{k,m}}(x_k)}  = 0.$
        Since $\theta_m$ is a unital order isomorphism onto its image, 
        it follows, using Lemma  \ref{l_exten},  that
            $\lim_{m \rightarrow \infty} \norm{\phi_{k,m}(x_k)} =0$ and, 
by Proposition \ref{thm:characterisations of null space}, $\phi_{k,\infty}(x_k) = 0$.
            }
    \end{remark}

    \begin{proposition}\label{rem:state space of AOU space}
        Let $(\{V_k\}_{k \in \bb{N}}, \{\phi_k\}_{k \in \bb{N}})$ be an inductive system in $\cataou$. 
        Then $S(\limaou V_k)$ is homeomorphic to $\limtop S(V_k)$. 
    \end{proposition}
    
    \begin{proof}      
If $f \in S(\limou V_k)$ then, by Theorem~\ref{lem:archimedeanisation of positive map}, 
there exists a unique unital positive map 
$\widetilde{f} \in S(\limaou V_k)$ such that $f = \widetilde{f} \circ q_{V}$. Define $\theta: S(\limou V_k) \rightarrow S(\limaou V_k)$ by letting 
$\theta(f) = \widetilde{f}$; it is straightforward to check that $\theta$ is a homeomorphism
(recall that the state space is equipped with the weak* topology). 
By  Proposition~\ref{l_si}, $S(\limou V_k)$ is homeomorphic to $\limtop S(V_k)$, and 
the claim follows. 
\end{proof}


\section{Inductive limits of operator systems} \label{chpt:Inductive limit of operator systems}

    We begin this section with the construction of the inductive limit in the category~$\catmou$
    of matrix ordered spaces, and 
    in Section~\ref{sect:Inductive limit of operator systems} we consider 
the inductive limit in the category $\catos$ of operator systems. 
We devote the remainder of the chapter to proving various ``commutation theorems'' for the 
inductive limit in $\catos$. 
In particular, we prove that the inductive limit 
intertwines $\omax$ and commutes with the maximal operator system tensor product. Analogous results hold 
for $\omin$ and 
the minimal operator system tensor product, provided the connecting morphisms are complete order embeddings. 
We note that the commutation with the 
minimal tensor product in the case of complete operator systems was 
recently proved in \cite{LuthraKumar}. 
We also establish, under certain natural conditions, the commutation of the
inductive limit with the quotient construction.

\subsection{Inductive limits of matrix ordered *-vector spaces} 
\label{sect:Inductive limit of matrix ordered vector spaces}

In this subsection, let 
$(\cl S_k, \lbrace C_n^k\rbrace_{n \in \bb{N}}, e_k)_{k\in \bb{N}}$ be a sequence of matrix ordered *-vector 
spaces with matrix order unit and 
$\phi_k : \cl S_k\to \cl S_{k+1}$ be a unital completely positive map, $k\in \bb{N}$; 
thus,
    \begin{equation}\label{eq_seqmou2}
        \cl S_1\stackrel{\phi_1}{\longrightarrow} \cl S_2 \stackrel{\phi_2}{\longrightarrow} \cl S_3 \stackrel{\phi_3}{\longrightarrow} \cl S_4 \stackrel{\phi_4}{\longrightarrow} \cdots
    \end{equation}
    is an inductive system in $\catmou$. 
For each $n\in \bb{N}$, consider the induced inductive system in $\catou$:
    \begin{equation*}\label{eq_seqmou3}
        M_n(\cl S_1)\stackrel{\phi_1^{(n)}}{\longrightarrow} M_n(\cl S_2) \stackrel{\phi_2^{(n)}}{\longrightarrow}
        M_n(\cl S_3) \stackrel{\phi_3^{(n)}}{\longrightarrow} M_n(\cl S_4) \stackrel{\phi_4^{(n)}}{\longrightarrow} \cdots.
    \end{equation*}
    Denote by $\ddot{\phi}_{k,\infty}^n$ the unital positive map associated to 
    $\limou M_n(\cl S_k)$ through (\ref{eq_kinfty}),
    so that $\ddot{\phi}_{k,\infty}^n : M_n (\cl S_k) \rightarrow \limou M_n(\cl S_k)$ and 
    $\ddot{\phi}_{k+1,\infty}^n \circ \phi_{k}^{(n)} = \ddot{\phi}_{k,\infty}^n$ for all $k \in \bb{N}$.
    Note that $\ddot{\phi}_{k,\infty}^1 = \ddot{\phi}_{k,\infty}$. 
    We caution the reader about the difference between the maps $\ddot{\phi}_{k,\infty}^n$ and 
    $\ddot{\phi}_{k,\infty}^{(n)}$: 
    while their domains are both equal to $M_n(\cl S)$, their ranges are within $\limou M_n(\cl S_k)$ 
    and  $M_n(\limou \cl S_k)$, respectively.

    \begin{lemma} \label{lem:matrix over inductive limit}
We have that 
$M_n(\limou \cl S_k) = \bigcup_{k \in \bb{N}} {\ddot{\phi}_{k,\infty}}^{(n)} M_n(\cl S_k)$, $n \in \bb{N}$.
    \end{lemma}

    \begin{proof}
        Fix $n \in \bb{N}$. 
        It is clear that ${\ddot{\phi}_{k,\infty}}^{(n)} (M_n(\cl S_k))\subseteq M_n(\limou \cl S_k)$
        for all $k$. To show the reverse inclusion,        
let $(s_{i,j})_{i,j} \in M_n(\limou \cl S_k)$. 
For all $1 \leq i,j \leq n$, we have that 
$s_{i,j} = \ddot{\phi}_{k_{i,j}}(s_{k_{i,j}})$ for some $k_{i,j} \in \bb{N}$ and $s_{k_{i,j}} \in \cl S_{k_{i,j}}$. Let $k = \max \{k_{i,j}: 1 \leq i,j \leq n\}$ and $s_{i,j}^k = \phi_{k_{i,j}, k}(s_{k_{i,j}})$.
We have that
        $s_{i,j} = \ddot{\phi}_{k,\infty}(s_{i,j}^k)$
         for all $1 \leq i,j \leq n$ and hence $(s_{i,j})_{i,j} \in \ddot{\phi}_{k,\infty}^{(n)}(M_n(\cl S_k))$.
    \end{proof}

In the next lemma, $M_n(\limou \cl S_k)$ is equipped with its
canonical involution arising from the involution of $\limou \cl S_k$.

    \begin{lemma}\label{l_pin}
    The mapping 
    $\pi_n: M_n(\limou \cl S_k) \rightarrow \limou{M_n(\cl S_k)}$ given by 
    \[
            \pi_n \circ \ddot{\phi}_{k,\infty}^{(n)} = \ddot{\phi}_{k,\infty}^n, \ \ \ k \in \bb{N},
        \]
        is well-defined, bijective and involutive. 
    \end{lemma}

    \begin{proof}
        Fix $n \in \bb{N}$ and
       let $S \in M_n(\limou \cl S_k)$.
       By Lemma~\ref{lem:matrix over inductive limit}, 
       $S = \ddot{\phi}_{k,\infty}^{(n)}(S_k)$ for some $k \in \bb{N}$ and some $S_k \in M_n(\cl S_k)$. 
       Suppose that $S_k = (s_{i,j}^k)_{i,j} \in M_n(\cl S_k)$ and $S_l = (s_{i,j}^l)_{i,j}\in M_n(\cl S_l)$ are 
       such that $\ddot{\phi}_{k,\infty}^{(n)}(S_k) = \ddot{\phi}_{l,\infty}^{(n)}(S_l)$; 
       then, for all $1 \leq i,j \leq n$, there exists~$m_{i,j}$ such that ${\phi}_{k,m_{i,j}}(s_{i,j}^k) = {\phi}_{l,m_{i,j}}(s_{i,j}^l)$. 
       Let $m=\max \{m_{i,j} : 1 \leq i,j \leq n\}$; we have $\phi_{k,m}^{(n)}(S_k) = \phi_{l,m}^{(n)}(S_l)$. 
       Therefore $\ddot{\phi}_{k,\infty}^n(S_k) = \ddot{\phi}_{l,\infty}^n(S_l)$.
It follows that the mapping $\pi_n$ is well-defined. 
Since the mappings $\ddot{\phi}_{k,\infty}^{(n)}$ and $\ddot{\phi}_{k,\infty}^{n}$ are linear, 
we have that $\pi_n$ is linear. 

Suppose $S_k \in M_n(\cl S_k)$
is such that $\ddot{\phi}_{k,\infty}^n(S_k) = 0$. 
Then there exists $m > k$ such that $\phi_{k,m}^{(n)}(S_k) = 0$ and therefore
        $\ddot{\phi}_{k,\infty}^{(n)}(S_k) = \ddot{\phi}_{m,\infty}^{(n)} \circ \phi_{k,m}^{(n)}(S_k) = 0$.
This shows that $\pi_n$ is injective. 

Finally, let $S_k = (s_{i,j}^k)_{i,j}  \in M_n(\cl S_k)$. Then 
\begin{eqnarray*}
& & 
\ddot{\phi}_{k,\infty}^{n}(S_k)^*\\
& = & 
[(0,\dots,0,S_k,\phi_k^{(n)}(S_k),\dots)]^*
= 
[(0,\dots,0,S_k^*,\phi_k^{(n)}(S_k)^*,\dots)]\\
& = & 
[(0,\dots,0,S_k^*,\phi_k^{(n)}(S_k^*),\dots)] 
= 
\ddot{\phi}_{k,\infty}^{n}(S_k^*),
\end{eqnarray*}
and the proof is complete. 
    \end{proof}

We denote $\limou \cl S_k$ by $\ddot{\cl S}_{\infty}$ and let, as before, $\ddot{e}_{\infty} = \ddot{\phi}_{k,\infty}(e_k)$
for any $k\in \bb{N}$ (note that $\ddot{e}_{\infty}$ is thus well-defined).
    For each $n \in \bb{N}$, let $C_n \subseteq M_n(\ddot{\cl S}_{\infty})_h$ be given by
    \[
    C_n = \pi_n^{-1} \big((\limou M_n(\cl S_k))^+\big).
    \]

    \begin{proposition} \label{prop:S_infty is matrix ordered space}
        The triple 
       $(\ddot{\cl S}_{\infty}, \lbrace C_n\rbrace_{n \in \mathbb{N}},\ddot{e}_{\infty})$  
       is a matrix ordered *-vector space with matrix order unit.
     \end{proposition}

    \begin{proof}
Since $C_n$ is the inverse image of a proper cone under the injective mapping $\pi_n$ (Lemma \ref{l_pin}), 
we have that $C_n$ is a proper cone itself. 
We show that the family $\{C_n\}_{n \in \bb{N}}$ is compatible. 
Let $n,m \in \bb{N}$, $\alpha \in M_{n,m}$ and $\ddot{\phi}_{k,\infty}^{(n)}(S_k)\in C_n$,
where $S_k\in M_n(\cl S_k)$. 
        There exists $p \in \bb{N}$ such that 
        $\phi_{k,p}^{(n)}(S_k)\in M_n(\cl S_p)^+$. 
        We conclude that $\alpha^* \phi_{k,p}^{(n)}(S_k) \alpha \in M_m(\cl S_p)^+$ 
        and so $\ddot{\phi}_{p,\infty}^m(\alpha^* \phi_{k,p}^{(n)}(S_k) \alpha) \in (\limou M_m(\cl S_k))^+.$ 
        Therefore
        \[
            \alpha^* \ddot{\phi}_{k,\infty}^{(n)}(S_k)\alpha 
            = \ddot{\phi}_{p,\infty}^{(m)}(\alpha^* \phi_{k,p}^{(n)}(S_k) \alpha)\in C_m.
        \]
        Thus, $\{C_n\}_{n \in \bb{N}}$ is a matrix ordering for $\ddot{\cl S}_{\infty}$.

        Finally we show that $\ddot{e}_{\infty}$ is a matrix order unit. Observe that
        $
           \ddot{e}_{\infty}^{(n)} = \ddot{\phi}_{k,\infty}^{(n)}\big(e_k^{(n)}\big).
        $
       Suppose that 
       $\ddot{\phi}_{k,\infty}^{(n)}(S_k)\in ({M_n(\ddot{\cl S}_{\infty})})_h$. Then there exists $m >k$ such that
        \[
             \phi_{k,m}^{(n)}(S_k)\in (M_n(\cl S_m))_h.
        \]
        Since $e_m$ is a matrix order unit for $\cl S_m$, there exists $r > 0$ such that
        \[
            \phi_{k,m}^{(n)}(S_k) \leq r e_m^{(n)} = \phi_{k,m}^{(n)}\big(r e_k^{(n)}\big).
        \]
        Therefore
$\ddot{\phi}_{k,\infty}^n(S_k) \leq \ddot{\phi}_{k,\infty}^n\big(r e_k^{(n)}\big)$
        and thus
        \[
            \ddot{\phi}_{k,\infty}^{(n)}(S_k) 
            \leq \ddot{\phi}_{k,\infty}^{(n)}\big(r e_k^{(n)}\big) 
            = r \ddot{e}_{\infty}^{(n)}.
        \]
    \end{proof}

        For the remainder of this section, we denote by
        $\ddot{\cl S}_{\infty}$
        the matrix ordered *-vector space with matrix order unit $(\ddot{\cl S}_{\infty}, \{C_n\}_{n \in \bb{N}}, \ddot{e}_{\infty})$.

            \begin{remark} \label{rem:phi infinity cp1}
{\rm        
The map $\ddot{\phi}_{k,\infty} : \cl S_k \rightarrow \ddot{\cl S}_{\infty}$ is unital and completely positive. Indeed, suppose $S_k \in M_n(\cl S_k)^+$. Since $\ddot{\phi}_{k,\infty}^n$ is a unital positive map, 
$\ddot{\phi}_{k,\infty}^n(S_k) \in \big(\limou M_n(\cl S_k)\big)^+$ and therefore 
$\ddot{\phi}_{k,\infty}^{(n)}(S_k) \in C_n$.
}
    \end{remark}

    \begin{proposition}\label{rem:Eachphi_kCOIThenphi_k,inftyCOI}
        Let $\cl S_1\stackrel{\phi_1}{\longrightarrow} \cl S_2 \stackrel{\phi_2}{\longrightarrow} \cl S_3 \stackrel{\phi_3}{\longrightarrow} \cl S_4 \stackrel{\phi_4}{\longrightarrow} \cdots$
        be an inductive system in $\catmou$ such that $\phi_k$ is a complete order isomorphism onto its image for each $k \in \mathbb{N}$. Then $\ddot{\phi}_{k, \infty}$ is a complete order isomorphism onto its image for each $k \in \mathbb{N}$. 
    \end{proposition}        
  
  \begin{proof}      
        By Remarks~\ref{rem:Eachphi_kOIThenphi_k,inftyOI} and \ref{rem:phi infinity cp1}, it suffices to show that
        $
            \ddot{\phi}_{k, \infty}^{-1}
        $
        is completely positive.
       Suppose
        $
            \ddot{\phi}_{k, \infty}^{(n)}
            \left(
            S_k
            \right)
            \in C_n
        $
        for some $S_k \in M_n(\cl S_k)$.
        Then there exists $m>k$ such that
        $
            \phi_{k,m}^{(n)}
            \left(
            S_k
            \right)
            \in M_n(\cl S_m)^+.
        $
        Since 
        $\phi_{k,m}$ is a complete order isomorphism onto its image, $S_k\in M_n(\cl S_k)^+$.
\end{proof}

\begin{theorem} \label{thm:universal property  matrix *vectorspace2}
        The triple
        $(\ddot{\cl S}_{\infty}, \{C_n\}_{n \in \bb{N}}, \ddot{e}_{\infty})$ is an inductive limit of the inductive system
        $\cl S_1\stackrel{\phi_1}{\longrightarrow} \cl S_2 \stackrel{\phi_2}{\longrightarrow} \cl S_3 \stackrel{\phi_3}{\longrightarrow} \cl S_4 \stackrel{\phi_4}{\longrightarrow} \cdots$
         in $\catmou$. 
\end{theorem}

    \begin{proof}
        Suppose $(\cl T, \lbrace \psi_k \rbrace_{k \in \mathbb{N}})$ is a pair consisting of a matrix ordered 
        *-vector space with matrix order unit and a family of unital completely positive maps $\psi_k: \cl S_k \rightarrow \cl T$ 
        such that $\psi_{k+1} \circ \phi_{k} = \psi_k$ for all $k \in \mathbb{N}$.
        By Theorem~\ref{thm:universal property *vectorspace2}, there exists a unique unital positive map 
        $\ddot{\psi} : \ddot{\cl S}_{\infty} \rightarrow \cl T$ such that
            $\ddot{\psi} \circ \ddot{\phi}_{k, \infty} = \psi_k$
        for all $k \in \bb{N}$. We show that~$\ddot{\psi}$ is completely positive.
        Suppose $\ddot{\phi}_{k,\infty}^{(n)}(S_k) \in C_n$; 
        then there exists $m>k$ such that $\phi_{k,m}^{(n)}(S_k) \in M_n(\cl S_m)^+$. Since $\psi_m$ is completely positive,
        \[
            \ddot{\psi}^{(n)} \circ \ddot{\phi}_{k, \infty}^{(n)}(S_k) = \psi_k^{(n)}(S_k) = \psi_m^{(n)} \circ \phi_{k,m}^{(n)}(S_k) \in M_n(\cl T)^+.
        \]
    \end{proof}

    Following our general convention, we denote the triple $(\ddot{\cl S}_{\infty}, \{C_n\}_{n \in \bb{N}}, \ddot{e}_{\infty})$ 
    by $\limmou \cl S_k$.

    \begin{remark}\label{rem:mou in limit isomorphism morphisms}
    {\rm
      Let $(\{\cl S_k\}_{k \in \mathbb{N}}, \{\phi_k\}_{k \in \bb{N}})$ and $(\{\cl T_k\}_{k \in \mathbb{N}}, \{\psi_k\}_{k \in \bb{N}})$ be 
      inductive systems in $\catmou$ and let $\{\theta_k\}_{k \in \mathbb{N}}$ be a sequence of unital completely positive maps such that the following diagram commutes:
      \begin{equation}\label{eq:square diagram mou}
    \begin{CD}
        \cl S_1 @>{\phi_1}>> \cl S_2 @>{\phi_2}>> \cl S_3 @>{\phi_3}>> \cl S_4 @>{\phi_4}>> \cdots \\
        @V{\theta_1}VV @V{\theta_2}VV @V{\theta_3}VV @V{\theta_4}VV \\
        \cl T_1 @>{\psi_1}>> \cl T_2 @>{\psi_2}>> \cl T_3 @>{\psi_3}>> \cl T_4 @>{\psi_4}>> \cdots.
    \end{CD}
  \end{equation}
    It follows from Theorems~\ref{thm:universal property  matrix *vectorspace2} and \ref{thm:universal property infinitely many maps} that there exists a unique unital completely positive map $\ddot{\theta} : \limmou \cl S_k \rightarrow \limmou \cl T_k$ such that $\ddot{\theta} \circ \ddot{\phi}_{k,\infty} = \ddot{\psi}_{k,\infty} \circ \theta_k$ for all $k \in \bb{N}$. 
    
        We note that if 
        $\theta_k$ is a complete order isomorphism onto its image for each $k \in \bb{N}$, then $\ddot{\theta}$~is a complete order isomorphism onto its image. Indeed, by 
        Remark~\ref{rem:universal property *vectorspace}, it remains to check that $\ddot{\theta}^{-1}$ is completely positive. Suppose
        $
            \ddot{\theta}^{(n)}
            \circ
            \ddot{\phi}_{k,\infty}^{(n)}
            \left(
            S_k
            \right)
            \in M_n(\limmou \cl T_k)^+
        $. 
        Then
        $\ddot{\psi}_{k,\infty}^{(n)} \circ \theta_k^{(n)}(S_k) \in M_n(\limmou \cl T_k)^+$.
        It follows that there exists $m>k$ such that
        $
            \psi_{k,m}^{(n)}
            \circ
            \theta_k^{(n)}
            \left(
            S_k
            \right)
            \in M_n(\cl T_m)^+.
        $
        Since (\ref{eq:square diagram mou}) commutes,
        $
            \theta_m^{(n)}
            \circ
            \phi_{k,m}^{(n)}
            \left(
            S_k
            \right)
            \in M_n(\cl T_m)^+$. Since $\theta_m$ is a complete order isomorphism,
        $
            \phi_{k,m}^{(n)}
            \left(
            S_k
            \right)
            \in M_n(\cl S_m)^+
        $
         and therefore
        $
            \ddot{\phi}_{k,\infty}^{(n)}
            \left(
            S_k
            \right)
            \in M_n(\limmou \cl S_k)^+.
        $
        }
    \end{remark}

\subsection{Inductive limits of operator systems} \label{sect:Inductive limit of operator systems}

    We now proceed to the inductive limit in the category of operator systems.  
    Let $(\cl S_k, \lbrace C_n^k\rbrace_{n \in \bb{N}}, e_k)_{k\in \bb{N}}$ 
    be a sequence of operator systems and let 
    $\phi_k : \cl S_k\to \cl S_{k+1}$ be a unital completely positive map, $k\in \bb{N}$; thus, 
    \begin{equation}\label{eq_seqos1}
        \cl S_1\stackrel{\phi_1}{\longrightarrow} \cl S_2 \stackrel{\phi_2}{\longrightarrow} \cl S_3 \stackrel{\phi_3}{\longrightarrow} \cl S_4 \stackrel{\phi_4}{\longrightarrow} \cdots
    \end{equation}
    is an inductive system in $\catos$. 
    Let $\bf{F} : \catos \rightarrow \catmou$ be the forgetful functor;
    consider the inductive limit $\limmou {\bf F}(\cl S_k)$. 
    We will show that its Archimedeanisation
    is an inductive limit for the inductive system (\ref{eq_seqos1}). 
    
    Write $\limmou {\bf F}(\cl S_k) = (\ddot{\cl S}_{\infty}, \{C_n\}_{n \in \bb{N}}, \ddot{e}_{\infty})$
    (recall that $\ddot{e}_{\infty} = \ddot{\phi}_{k,\infty}(e_k)$, $k\in \bb{N}$). 
Let
    \[
        N = \big\lbrace s\in \ddot{\cl S}_{\infty} :  f(s)=0 \ \text{~for all~} \ f \in S(\ddot{\cl S}_{\infty})\big\rbrace
    \]
    be the {\it null space} of $\ddot{\cl S}_{\infty}$. 
Set
    \[
            \cl S_{\infty} = \ddot{\cl S}_{\infty}/N,
    \]
   write $q_{\cl S}: \ddot{\cl S}_{\infty} \rightarrow \cl S_{\infty}$ for the canonical quotient map and let
   $\phi_{k,\infty} = q_{\cl S} \circ \ddot{\phi}_{k,\infty}$. 
   We may identify $M_n(\ddot{\cl S}_{\infty} / N)$ with $M_n(\ddot{\cl S}_{\infty}) / M_n(N)$ in a natural way.
Note that, since $N$ is closed under the involution of $\ddot{\cl S}_{\infty}$, 
the space $M_n(N)$ is closed under the involution of $M_n(\ddot{\cl S}_{\infty})$.

      The proof of the next lemma is analogous to that of 
      Lemma~\ref{lem:self adjoint element archimedeanisation} and is omitted.

    \begin{lemma}\label{lem:self adjoint element archimedeanisation matrix}
        Let 
        $S_k\in M_n(\cl S_k)$. 
        The following are equivalent:
        \begin{enumerate}[\rm(i)]
            \item $\phi_{k,\infty}^{(n)}(S_k) \in (M_n(\ddot{\cl S}_{\infty}) / M_n(N))_h$;
            \item $\phi_{k,\infty}^{(n)}(S_k) = \phi_{k,\infty}^{(n)}({\rm Re}(S_k))$;
            \item $\phi_{k,\infty}^{(n)}(S_k) = \phi_{l,\infty}^{(n)}(S_l)$ for some $l \in \bb{N}$ and some $S_l \in (M_n(\cl S_l))_h$.
        \end{enumerate}
            \end{lemma}

    For each $n \in \bb{N}$, define

    \begin{equation} \label{eq:def positive matrix in limit}
      \begin{split}
    D_n = \Big\lbrace\phi_{k,\infty}^{(n)}(S_k) \in M_n&(\cl S_{\infty})_h :
S_k\in \cl S_k \mbox{ and for each } r > 0 
          \text{~there exist~}\\ 
          & l\in \bb{N}
          \text{~and~} T_l\in M_n(\cl S_l)
          \text{~with~} \ddot{\phi}_{\l, \infty}^{(n)}(T_l) \in M_n(N) \\
          &
          \text{~and~} 
          \ddot{\phi}_{k,\infty}^{(n)}(r e_k^{(n)} + S_k) + \ddot{\phi}_{l,\infty}(T_l) \in C_n \Big \rbrace.
          \end{split}
  \end{equation}

    \begin{remark}
    {\rm 
    Suppose $\ddot{\phi}_{k,\infty}^{(n)}(S_k) + M_n(N) \in (M_n(\cl S_{\infty}))_h$.
 We have that $\phi_{k,\infty}^{(n)}(S_k)\in  D_n$ if and only if 
 for all $r>0$ there exist $l\in \bb{N}$, $T_l\in M_n(\cl S_l)$ and $m > \max\{k,l\}$ such that 
$\ddot{\phi}_{l,\infty}^{(n)}(T_l) \in M_n(N)$  and 
$re_m^{(n)} + \phi_{k,m}^{(n)}(S_k) + \phi_{l,m}^{(n)}(T_l) \in M_n(\cl S_m)^+$. 
We may assume without loss of generality that $l>k$, $T_l \in (M_n(\cl S_l))_h$,
and 
$\phi_{k,m}(S_k)\in M_n(\cl S_m)_h$. 
}
    \end{remark}

  Note that the space 
  $(\cl S_{\infty}, \{D_n\}_{n \in \bb{N}}, e_{\infty})$,
  where $e_{\infty} = \phi_{k,\infty}(e_k)$ for some (and hence any) $k\in \bb{N}$,
  is the Archimedeanisation of the matrix ordered *-vector space 
  $(\ddot{\cl S}_{\infty}, \{C_n\}_{n \in \bb{N}}, \ddot{e}_{\infty})$.

    \begin{proposition} \label{prop:S_{infty} is operator space}
        The triple $(\cl S_{\infty}, \{D_n\}_{n \in \bb{N}}, e_{\infty})$ 
        is an operator system and $\phi_{k, \infty}$ is a unital completely positive map.
            \end{proposition}

    \begin{proof}
        Since $(\cl S_{\infty}, \{D_n\}_{n \in \bb{N}}, e_{\infty})$ is the Archimedeanisation of the matrix ordered 
        *-vector space $(\ddot{\cl S}_{\infty}, \{C_n\}_{n \in \bb{N}}, \ddot{e}_{\infty})$, 
        it follows from \cite[Proposition~3.16]{PaulsenTodorovTomforde} that it is an operator system.
        By Remark~\ref{rem:phi infinity cp1}, $\ddot{\phi}_{k, \infty}$ is a unital completely positive map.
        Since $q_{\cl S}$ is a unital completely positive map,
we have that $\phi_{k, \infty}$ is a unital completely positive map.
    \end{proof}

    \begin{theorem} \label{thm:universal property_operator_system2}
        The triple 
        $(\cl S_{\infty}, \{D_n\}_{n \in \bb{N}}, e_{\infty})$ is an inductive limit of the inductive system
            $$\cl S_1\stackrel{\phi_1}{\longrightarrow} \cl S_2 \stackrel{\phi_2}{\longrightarrow} \cl S_3 \stackrel{\phi_3}{\longrightarrow} \cl S_4 \stackrel{\phi_4}{\longrightarrow} \cdots$$
            in $\catos$. 
                \end{theorem}
           
                \begin{proof}
        Suppose $(\cl T, \lbrace \psi_k \rbrace_{k \in \mathbb{N}})$ is a pair consisting of an operator system and a family of unital completely positive maps $\psi_k: \cl S_k \rightarrow \cl T$ such that $\psi_{k+1} \circ \phi_{k} = \psi_k$ for all $k \in \mathbb{N}$.
       By Theorem~\ref{thm:universal property  matrix *vectorspace2}, 
       there exists a unique unital completely positive map $\ddot{\psi}:\ddot{\cl S}_{\infty} \rightarrow \cl T$ such that
        $\ddot{\psi} \circ \ddot{\phi}_{k, \infty} = \psi_k.$
        By Theorem~\ref{prop: Archimedeanisation of mou space}, there exists a unique unital completely positive map $\psi: \cl S_{\infty} \rightarrow \cl T$ such that
            $\ddot{\psi} = \psi \circ q_{\cl S}.$
        Thus
        \[
            \psi \circ \phi_{k,\infty} = \psi \circ q_{\cl S} \circ \ddot{\phi}_{k,\infty} = \ddot{\psi} \circ \ddot{\phi}_{k,\infty} = \psi_k, \ \ \ k \in \bb{N}.
        \]
    \end{proof}

 Using our general notational convention, we denote by $\limos \cl S_k$ the inductive limit 
  $(\cl S_{\infty}, \{\phi_{k,\infty}\}_{k \in \bb{N}})$ of 
 the inductive system $(\{\cl S_k\}_{k \in \mathbb{N}}, \{\phi_k\}_{k \in \bb{N}})$ 
  in the category $\catos$.
We often write $\cl S_{\infty} = \limos \cl S_k$.

    \begin{remark}\label{rem:both methods are same}
    {\rm 
          Let $(\{\cl S_k\}_{k \in \bb{N}}, \{\phi_{k}\}_{k \in \bb{N}})$ be an inductive system in~$\catos$. For each $n\in \bb{N}$, consider the induced inductive system
    \begin{equation*}\label{eq_seq3}
        M_n(\cl S_1)\stackrel{\phi_1^{(n)}}{\longrightarrow} M_n(\cl S_2) \stackrel{\phi_2^{(n)}}{\longrightarrow}
        M_n(\cl S_3) \stackrel{\phi_3^{(n)}}{\longrightarrow} M_n(\cl S_4) \stackrel{\phi_4^{(n)}}{\longrightarrow} \cdots
    \end{equation*}
    in $\cataou$.
    Let us denote by $\phi_{k,\infty}^n$ the unital positive map associated to $\limaou M_n(\cl S_k)$ so that $\phi_{k,\infty}^n : M_n (\cl S_k) \rightarrow \limaou M_n(\cl S_k)$ and 
    $\phi_{k+1,\infty}^n \circ \phi_{k}^{(n)} = \phi_{k,\infty}^n$ for all $k \in \bb{N}$. 
As a consequence of Remark~\ref{rem:Arch is Arch at every matrix level},    
    one can see that $\limos \cl S_k$ is the operator system 
    with underlying *-vector space 
    $\limaou \cl S_k$ such that $\phi_{k,\infty}^{(n)}(S_k) \in M_n(\limos \cl S_k)^+$ if and only if 
    $\phi_{k,\infty}^n(S_k) \in (\limaou M_n(\cl S_k))^+$.     
    }
    \end{remark}

\begin{proposition}\label{p_coios}
Let
$$\cl S_1\stackrel{\phi_1}{\longrightarrow} \cl S_2 \stackrel{\phi_2}{\longrightarrow} \cl S_3 \stackrel{\phi_3}{\longrightarrow} \cl S_4 \stackrel{\phi_4}{\longrightarrow} \cdots$$
           be an inductive system in $\catos$, and suppose that $\phi_k$ is a complete order embedding for each $k\in \bb{N}$. 
           Then $\phi_{k,\infty}$ is a complete order embedding.
\end{proposition}
\begin{proof}
The statement follows from Proposition \ref{prop:Eachphi_kOIthenN=0} and Remark \ref{rem:both methods are same}.
\end{proof}

    \begin{remark} \label{rem:theta coi_os}
    {\rm 
        Let 
        $(\{\cl S_k\}_{k \in \bb{N}}, \{\phi_k\}_{k \in \bb{N}})$ and $(\{\cl T_k\}_{k \in \bb{N}}, \{\psi_k\}_{k \in \bb{N}})$ be 
        inductive systems in $\catos$ and let $\{\theta_k\}_{k \in \mathbb{N}}$ be a sequence of unital completely positive maps such that the following diagram commutes:
          \begin{equation*} \label{eq:commuting diagram in lim opsys}
                \begin{CD}
                    \cl S_1 @>{\phi_1}>> \cl S_2 @>{\phi_2}>> \cl S_3 @>{\phi_3}>> \cl S_4 @>{\phi_4}>> \cdots \\
                    @V{\theta_1}VV @V{\theta_2}VV @V{\theta_3}VV @V{\theta_4}VV \\
                    \cl T_1 @>{\psi_1}>> \cl T_2 @>{\psi_2}>> \cl T_3 @>{\psi_3}>> \cl T_4 @>{\psi_4}>> \cdots.
                \end{CD}
            \end{equation*}
        It follows from Theorems \ref{thm:universal property  matrix *vectorspace2} and 
        \ref{thm:universal property infinitely many maps} that there exists a unique unital completely positive map $\theta : \limos \cl S_k \rightarrow \limos \cl T_k$ such that $\theta \circ \phi_{k,\infty} = \psi_{k,\infty} \circ \theta_k$ for all $k \in \bb{N}$. 
It follows from Remark~\ref{rem:universal property AOU} that if 
    each $\theta_k$ is a complete order isomorphism onto its image then $\theta$ is injective.
        }
    \end{remark}

    \begin{remark} \label{rem:all maps coi}
    {\rm 
Let $(\{\cl S_k\}_{k \in \bb{N}}, \{\phi_k\}_{k \in \bb{N}})$ and 
      $(\{\cl T_k\}_{k \in \bb{N}}, \{\psi_k\}_{k \in \bb{N}})$ be inductive systems in $\catos$, and assume that
      $\phi_k$ and $\psi_k$ are unital complete order embeddings, $k\in \bb{N}$.
      If $\{\theta_k\}_{k \in \mathbb{N}}$ is a sequence of unital complete order embeddings such that the following diagram commutes:
        \begin{equation*} \label{eq:commuting diagram in lim opsys}
              \begin{CD}
                  \cl S_1 @>{\phi_1}>> \cl S_2 @>{\phi_2}>> \cl S_3 @>{\phi_3}>> \cl S_4 @>{\phi_4}>> \cdots \\
                  @V{\theta_1}VV @V{\theta_2}VV @V{\theta_3}VV @V{\theta_4}VV \\
                  \cl T_1 @>{\psi_1}>> \cl T_2 @>{\psi_2}>> \cl T_3 @>{\psi_3}>> \cl T_4 @>{\psi_4}>> \cdots,
              \end{CD}
          \end{equation*}
                then $\theta : \limos \cl S_k \rightarrow \limos \cl T_k$ is a unital complete order embedding. 
Indeed, note first that, since the connecting maps are complete oder embeddings, 
the null spaces, associated with the two inductive systems, coincide with the zero spaces
(see Proposition \ref{prop:Eachphi_kOIthenN=0}). 
Suppose that $S_k\in M_n(\cl S_k)$ is such that 
        $
            \theta^{(n)}
            \circ
            \phi_{k,\infty}^{(n)}
            \left(
            S_k
            \right)
            \in M_n(\limmou \cl T_k)^+
        $. 
        Then
        $\psi_{k,\infty}^{(n)} \circ \theta_k^{(n)}(S_k) \in M_n(\limmou \cl T_k)^+$.
        Thus, for every $r > 0$, 
there exists $m > k$ such that 
        $
            r e_{\cl T}^{(n)} + \psi_{k,m}^{(n)}
            \circ
            \theta_k^{(n)}
            \left(
            S_k
            \right)
            \in M_n(\cl T_m)^+.
        $
        Since (\ref{eq:square diagram mou}) commutes, this means that 
        $
            \theta_m^{(n)}
            \left(\phi_{k,m}^{(n)}(S_k) + r e_{\cl S_k}^{(n)}\right)
            \in M_n(\cl T_m)^+$. 
            Since $\theta_m$ is a unital complete order isomorphism,            
        $
        \phi_{k,m}^{(n)}(S_k) + r e_{\cl S_k}^{(n)}
            \in M_n(\cl S_m)^+.
        $
Thus, 
        $
            \phi_{k,\infty}^{(n)}
            \left(
            S_k
            \right)
            \in M_n(\limmou \cl S_k)^+.
       $                
                }
    \end{remark}

    \begin{proposition}\label{prop:norm for inductive limit operator system}
        Let $\cl S_1\stackrel{\phi_1}{\longrightarrow} \cl S_2 \stackrel{\phi_2}{\longrightarrow} \cl S_3 \stackrel{\phi_3}{\longrightarrow} \cl S_4 \stackrel{\phi_4}{\longrightarrow} \cdots$
        be an inductive system in~$\catos$. For each $k \in \bb{N}$, let $\norm{\cdot}_k$ be the norm of 
        $\cl S_k$. If $s_k \in \cl S_k$ then
        \[
            \norm{\phi_{k,\infty}(s_k)} = \lim_{l \rightarrow \infty}\norm{\phi_{k,l}(s_k)}_l.
        \]
            \end{proposition}

    \begin{proof}
        Let $s_k \in \cl S_k$ and suppose that $\lim_{l \rightarrow \infty}\norm{\phi_{k,l}(s_k)}_l < 1$. Then there exists $m >k$ such that $\norm{\phi_{k,m}(s_k)}_m < 1$. By Lemma~\ref{rem:norm of an operator system},
        \[
            \begin{pmatrix}
                e_m & \phi_{k,m}(s_k) \\
                \phi_{k,m}(s_k)^* & e_m \\
            \end{pmatrix}
            \in M_2(\cl S_m)^+.
        \]
        Thus,
        \[
        \phi_{m,\infty}^{(2)}\left(
        \begin{pmatrix}
          e_m & \phi_{k,m}(s_k) \\
          \phi_{k,m}(s_k)^* & e_m \\
        \end{pmatrix}
        \right)
        =
            \begin{pmatrix}
                e_{\infty} & \phi_{k,\infty}(s_k) \\
                \phi_{k,\infty}(s_k)^* & e_{\infty} \\
            \end{pmatrix}
        \]
        is an element of $M_2(\limos \cl S_k)^+$
        and therefore $\norm{\phi_{k,\infty}(s_k)} \leq 1$. This proves that
        $\norm{\phi_{k,\infty}(s_k)} \leq \lim_{l \rightarrow \infty}\norm{\phi_{k,n}(s_k)}_l$.

To establish the reverse inequality, 
suppose that $\norm{\phi_{k,\infty}(s_k)} < 1$. 
Then 
$$
\begin{pmatrix}
                e_{\infty} & \phi_{k,\infty}(s_k) \\
                \phi_{k,\infty}(s_k)^* & e_{\infty} \\
            \end{pmatrix}
            \in M_2(\cl S_{\infty})^+.
$$
Let $r > 0$. Then there exist $q\geq k$, $T_q\in M_2(\cl S_q)$ and $m > q$ such that 
$\ddot{\phi}_{q,\infty}^{(2)}(T_q)\in M_2(N)$ and 
$$
\begin{pmatrix}
                (1 + r) e_m & \phi_{k,m}(s_k) \\
                \phi_{k,m}(s_k)^* & (1 + r) e_m \\
            \end{pmatrix} 
            + \phi_{q,m}^{(2)}(T_q)
            \in M_2(\cl S_{m})^+.
$$
By Proposition \ref{thm:characterisations of null space} and Remark \ref{rem:both methods are same}, 
we can choose $m$ to have the additional property that 
$$r e_m^{(2)} - \phi_{q,m}^{(2)}(T_q)  \in M_2(\cl S_{m})^+.$$
It now follows that 
$$
\begin{pmatrix}
                (1 + 2r) e_p & \phi_{k,p}(s_k) \\
                \phi_{k,p}(s_k)^* & (1 + 2r) e_p \\
            \end{pmatrix}
            \in M_2(\cl S_{p})^+, \ \ \ p\geq m,
$$
and hence 
$\norm{\phi_{k,p}(s_k)} \leq 1 + 2r$ for every $p\geq m$. 
Since $r$ is arbitrary, we conclude that 
$\lim_{l \rightarrow \infty}\norm{\phi_{k,l}(s_k)}_l \leq 1$. 
    \end{proof}

\subsection{Inductive limits of C*-algebras} \label{sect:Inductive limit c*alg}

    If $(\{\cl A_k\}_{k \in \bb{N}}, \{\phi_k\}_{k \in \bb{N}})$ is 
    an inductive system in $\catcalg$ then it
    is also an inductive system in $\catos.$ In the following theorem, we compare $\limos \cl A_k$ and $\limcalg \cl A_k$.

    \begin{theorem}\label{p_opsysst}
        Let
        $\cl A_1\stackrel{\phi_1}{\longrightarrow} \cl A_2 \stackrel{\phi_2}{\longrightarrow} \cl A_3 \stackrel{\phi_3}{\longrightarrow} \cl A_4 \stackrel{\phi_4}{\longrightarrow} \cdots$
        be an inductive system in~$\catcalg$, $\cl A_0 = \limos \cl A_k$ and 
        $\cl A = \limcalg \cl A_k$.
       Then $\cl A_0$ is unitally completely order isomorphic to a dense operator subsystem of 
       $\cl A$.
    \end{theorem}

    \begin{proof}
    Consider the commutative diagram
$$        \begin{CD}
        \cl A_1 @>{\phi_1}>> \cl A_2 @>{\phi_2}>> \cl A_3 @>{\phi_3}>> \cl A_4 @>{\phi_4}>> \cdots \\
        @V{\id}VV @V{\id}VV @V{\id}VV @V{\id}VV \\
        \cl A_1 @>{\phi_1}>> \cl A_2 @>{\phi_2}>> \cl A_3 @>{\phi_3}>> \cl A_4 @>{\phi_4}>> \cdots.
    \end{CD}
    $$
By Proposition \ref{prop:norm for inductive limit operator system} and 
the definition of the inductive limit in $\catcalg$, 
there exists an isometric linear map $\theta : \cl A_0 \to \cl A$ with dense range. 

We show that $\theta$ is a complete order isomorphism onto its image. 
        Suppose that $A_k\in M_n(\cl A_k)$ is such that 
        $\phi_{k,\infty}^{(n)}(A_k) \in M_n(\limos \cl A_k)^+$. 
        Fix $r>0$. 
        There exist $l > k$, $m > l$ and $B_l\in M_n(\cl A_l)$ such that 
        $\ddot{\phi}_{l,\infty}^{(n)}(B_l) \in M_n(N)$ and 
        $re_m^{(n)} + \phi_{k,m}^{(n)}(A_k) + \phi_{l,m}^{(n)}(B_l) \in M_n(\cl A_m)^+$.  
        Since $\phi_{m,\infty}$ is a unital 
        *-homomorphism and therefore a unital completely positive map, 
        we have that
        \[
            r \phi_{k,\infty}^{(n)}(e_k) + \phi_{k,\infty}^{(n)}(A_k) = \phi_{m,\infty}^{(n)}\big(r e_m^{(n)} + \phi_{k, m}^{(n)}(A_k) + \phi_{l,m}^{(n)}(B_l)\big) \in M_n(\limcalg \cl A_k)^+.
        \]
        Since $M_n(\limcalg \cl A_k)$ is an AOU space, $\phi_{k,\infty}^{(n)}(A_k) \in M_n(\limcalg \cl A_k)^+$.

       Now suppose that $\phi_{k,\infty}^{(n)}(A_k) \in M_n(\limcalg \cl A_k)^+$,
       where $A_k\in M_n(\cl A_k)$. 
       It follows that $\phi_{k,\infty}^{(n)}(A_k)= BB^*$ where $B \in M_n(\limcalg \cl A_k)$. 
       Assume $B = \lim_{p \rightarrow \infty} B^p$ where, for all $p \in \bb{N}$, 
       $B^p = \phi_{m_p, \infty}^{(n)} (B_{m_p})$ for some $m_p \in \bb{N}$ and some 
       $B_{m_p} \in \cl A_{m_p}$. 
       We may assume, without loss of generality, that $A_k \in M_n(\cl A_k)_h$ and $m_p > k$ for all $p \in \bb{N}$. 
        For all $r > 0$, there exists $p_0 \in \bb{N}$ such that
        \[
            \bignorm{ \phi_{k,\infty}^{(n)}(A_k) - \phi_{m_p, \infty}^{(n)} (B_{m_p}B_{m_p}^*)}_{M_n(\limcalg \cl A_k)} < r, \ \ \ p \geq p_0.
        \]
        Note that
        \[
            \begin{split}
            \bignorm{\phi_{k,\infty}^{(n)}(A_k) &- \phi_{m_p, \infty}^{(n)} (B_{m_p}B_{m_p}^*)}_{M_n(\limcalg \cl A_k)} \\
            &= \bignorm{\phi_{m_p,\infty}^{(n)}(\phi_{k,m_p}^{(n)}(A_k) - B_{m_p}B_{m_p}^*)}_{M_n(\limcalg \cl A_k)}\\
            &= \lim_{q \rightarrow \infty} \bignorm{\phi_{m_p, q}^{(n)}(\phi_{k,m_p}^{(n)}(A_k) - B_{m_p}B_{m_p}^*)}_{M_n(\cl A_q).}
            \end{split}
        \]
        Fix $r > 0$ and choose $p, q \in \bb{N}$ such that
        \[
            \bignorm{\phi_{m_{p}, q}^{(n)}(\phi_{k,m_{p}}^{(n)}(A_k) - B_{m_p}B_{m_p}^*)}_{M_n(\cl A_q)} < \frac{r}{2}.
        \]
        By \cite[Corollary~5.6]{PaulsenTomforde}, the norm $\norm{\cdot}_{M_n(\cl A_q)}$ agrees with the order norm on $M_n(\cl A_q)_h$; thus, 
        \[
            \frac{r}{2}e_q^{(n)} + \phi_{m_{p}, q}^{(n)}\big(\phi_{k,m_{p}}^{(n)}(A_k) - B_{m_p}B_{m_p}^*\big) \in M_n(\cl A_q)^+.
        \]
        Since
\begin{eqnarray*}
           \frac{r}{2}e_q^{(n)}  + \phi_{m_{p}, q}^{(n)}\big(\phi_{k,m_{p}}^{(n)}(A_k) - B_{m_p}B_{m_p}^*\big)
           & = & 
           \big(re_q^{(n)} + \phi_{k, q}^{(n)}(A_k)\big)\\
            & - & \big(\frac{r}{2}e_q^{(n)}
           + \phi_{m_{p}, q}^{(n)}(B_{m_p}B_{m_p}^*)\big)
\end{eqnarray*}
        and $\frac{r}{2}e_q^{(n)} + \phi_{m_{p}, q}^{(n)}(B_{m_p}B_{m_p}^*) \in M_n(\cl A_q)^+$, we have that 
        $$re_q^{(n)} + \phi_{k, q}^{(n)}(A_k) \in M_n(\cl A_q)^+.$$
        Therefore
        \[
            r\phi_{k,\infty}^{(n)}\big(e_{k}^{(n)}\big) + \phi_{k,\infty}^{(n)}(A_k) =  \phi_{q,\infty}^{(n)}\big(re_q^{(n)} + \phi_{k, q}^{(n)}(A_k))  \in M_n(\limos \cl A_k \big)^+.
        \]
        Since this holds for all $r > 0$, we have that $\phi_{k,\infty}^{(n)}(A_k) \in  M_n(\limos \cl A_k)^+$.
Thus, $\theta$ is a unital complete order isomorphism onto its image.
    \end{proof}

    \begin{corollary}\label{l_invdir}
        Let $X_1\stackrel{\alpha_1}{\longleftarrow} X_2 \stackrel{\alpha_2}{\longleftarrow} X_3 \stackrel{\alpha_3}{\longleftarrow} X_4 \stackrel{\alpha_4}{\longleftarrow} \cdots$
        be an inverse system in $\cattop$ such that $X_k$ is compact and Hausdorff, $k\in \bb{N}$. 
        Let $C(X_1) \stackrel{\phi_1}{\longrightarrow} C(X_2) \stackrel{\phi_2}{\longrightarrow} C(X_3) \stackrel{\phi_3}{\longrightarrow} C(X_4) \stackrel{\phi_4}{\longrightarrow} \cdots$
        be the canonically induced inductive system in $\catcalg$.
        Then there exists a unital completely order isomorphic embedding from $\limos C(X_k)$ into $C(\limtop X_k)$.
    \end{corollary}
    \begin{proof}
        This follows from Proposition~\ref{p_opsysst} and Remark~\ref{ex:InductiveLimitC(X)}.
    \end{proof}

\subsection{Inductive limits of $\omin$ and $\omax$} \label{sect:omin and omax}

    Let $V_1$ and $V_2$ be AOU spaces and 
    $\phi : V_1 \rightarrow V_2$ be a positive map. 
    It follows from \cite[Theorem~3.4]{PaulsenTodorovTomforde} that 
    $\phi$ is a completely positive map from $\omin(V_1)$ into $\omin(V_2)$ and, from 
    \cite[Theorem~3.22]{PaulsenTodorovTomforde}, that
    $\phi$ is a completely positive map from $\omax(V_1)$ into $\omax(V_2)$.
Therefore, given an inductive system
    $$V_1\stackrel{\phi_1}{\longrightarrow} V_2 \stackrel{\phi_2}{\longrightarrow} V_3 \stackrel{\phi_3}{\longrightarrow} V_4 \stackrel{\phi_4}{\longrightarrow} \cdots$$
    in $\cataou$, we have associated inductive systems
    $$\omin(V_1)\stackrel{\phi_1}{\longrightarrow} \omin(V_2) \stackrel{\phi_2}{\longrightarrow} \omin(V_3) \stackrel{\phi_3}{\longrightarrow} \omin(V_4) \stackrel{\phi_4}{\longrightarrow} \cdots$$
    and
    $$\omax(V_1)\stackrel{\phi_1}{\longrightarrow} \omax(V_2) \stackrel{\phi_2}{\longrightarrow} \omax(V_3) \stackrel{\phi_3}{\longrightarrow} \omax(V_4) \stackrel{\phi_4}{\longrightarrow} \cdots$$
in $\catos$. 
In this section we show that the inductive limit 
intertwines $\omax$ and that it intertwines $\omin$ when the connecting maps are order embeddings.

    \begin{lemma} \label{lem: order iso is complete order iso on omin}
      Let $V$ and $W$ be AOU spaces and let $\phi: V \rightarrow W$ be a unital order embedding. 
      Then $\phi: \omin(V) \rightarrow \omin(W)$ is a unital complete order embedding.
          \end{lemma}

      \begin{proof}
       Suppose that $\phi^{(n)}(X) \in M_n(\omin(W))^+$ for some $X =(x_{i,j})_{i,j} \in M_n(V)$,
       and let $g \in S(V)$.
       By Lemma \ref{l_exten}, there exists  
       $\widetilde{g} \in S(W)$ such that $\widetilde{g} \circ \phi = g$. 
       It follows that 
       $$(\angles{g}{x_{i,j}})_{i,j} = (\angles{\widetilde{g}}{\phi(x_{i,j})})_{i,j}  \in M_n^+.$$ 
       By Theorem~\ref{thm:characterisation of omin}, $X \in M_n(\omin(V))^+$.
      \end{proof}

    \begin{theorem}\label{th_ominin}
        Let\index{$\omin$}
        $V_1\stackrel{\phi_1}{\longrightarrow} V_2 \stackrel{\phi_2}{\longrightarrow} V_3 \stackrel{\phi_3}{\longrightarrow} V_4 \stackrel{\phi_4}{\longrightarrow} \cdots$
        be an inductive system in $\cataou$ such that each $\phi_k$ is a unital order embedding. 
        Then $\omin(\limaou V_k)$ is unitally completely order isomorphic to $\limos \omin(V_k)$.
    \end{theorem}
    
    \begin{proof}
   Let
        \begin{equation*}\label{eq_omin1}
            S(V_1)\stackrel{\phi_1'}{\longleftarrow} S(V_2) \stackrel{\phi_2'}{\longleftarrow} S(V_3) \stackrel{\phi_3'}{\longleftarrow} S(V_4) \stackrel{\phi_4'}{\longleftarrow} \cdots
        \end{equation*}
        be the corresponding inverse system in $\cattop$.
        Note that each $\phi_k'$ is surjective.
        By Proposition~\ref{rem:state space of AOU space}, there exists a homeomorphism
            $\alpha: S(\limaou V_k) \rightarrow \limtop S(V_k).$
            Let $\hat{\alpha}: C(\limtop S(V_k)) \rightarrow C(S(\limaou V_k))$ be the 
       unital *-isomorphism induced  by $\alpha$.
            
Consider, in addition, the induced inductive system in $\catcalg$ with *-isomorphic embeddings
        \[
            C(S(V_1))\stackrel{\alpha_1}{\longrightarrow} C(S(V_2)) \stackrel{\alpha_2}{\longrightarrow} C(S(V_3)) \stackrel{\alpha_3}{\longrightarrow} C(S(V_4)) \stackrel{\alpha_4}{\longrightarrow} \cdots.
        \]
By Corollary~\ref{l_invdir}, there exists a unital complete order embedding
            $$\beta: \limos C(S(V_k)) \rightarrow C(\limtop S(V_k)).$$
        By Theorem~\ref{thm:characterisation of omin}, for each $k \in \mathbb{N}$ the natural inclusion $\iota_k : \omin(V_k) \rightarrow C(S(V_k))$ is a unital completely order isomorphic embedding.
        
The diagram
        \[
            \begin{CD}
                \omin(V_1) @>{\phi_1}>> \omin(V_2) @>{\phi_2}>> \omin(V_3) @>{\phi_3}>> \cdots \\
                @V{\iota_1}VV @V{\iota_2}VV @V{\iota_3}VV \\
                C(S(V_1))
                 @>{\alpha_1}>> C(S(V_2)) @>{\alpha_2}>> C(S(V_3)) @>{\alpha_3}>> \cdots
            \end{CD}
        \]
commutes since ${\alpha_k}|_{\omin(V_k)} = \phi_k$. By Remark~\ref{rem:all maps coi}, there exists a 
unital complete order embedding
        $$\iota: \limos \omin(V_k) \rightarrow \limos C(S(V_k)).$$
        Therefore
        \[
             \hat{\alpha} \circ \beta \circ \iota : \limos \omin(V_k) \rightarrow C(S(\limaou V_k))
        \]
        is a unital completely order isomorphic embedding. Thus, 
        $ \limos \omin(V_k)$ is completely order isomorphic to an operator subsystem $\cl T$ of  
        the C*-algebra $C(S(\limaou V_k))$. 
        By Theorem~\ref{thm:characterisation of omin}, $\cl T = \omin (\limaou V_k)$.
    \end{proof}

    \begin{theorem}\label{th_omax}
        Let\index{$\omax$}
        $V_1\stackrel{\phi_1}{\longrightarrow} V_2 \stackrel{\phi_2}{\longrightarrow} V_3 \stackrel{\phi_3}{\longrightarrow} V_4 \stackrel{\phi_4}{\longrightarrow} \cdots$
        be an inductive system in $\cataou$. Then $\omax(\limaou V_k)$ is unitally completely order isomorphic to $\limos \omax(V_k)$.
    \end{theorem}
    \begin{proof}
        Recall that $\phi_{k,\infty}: V_k \rightarrow \limaou V_k$ is a unital positive map for all $k \in \bb{N}$. 
        By \cite[Theorem~3.22]{PaulsenTodorovTomforde}, the 
        map $\hat{\phi}_{k,\infty}$ that formally coincides with $\phi_{k,\infty}$, 
        but considered from $\omax(V_k)$ into $\omax(\limaou V_k)$, 
        is unital and completely positive, $k \in \bb{N}$. 
        Denote temporarily by $\tilde{\phi}_{k,\infty}$ the canonical map from 
        $\omax(V_k)$ into $\limos \omax(V_k)$.
        By Theorem~\ref{thm:universal property_operator_system2}, there exists 
        a unique unital completely positive map $\iota: \limos \omax(V_k) \rightarrow \omax(\limaou V_k)$ 
        such that $\iota \circ \tilde{\phi}_{k,\infty} = \hat{\phi}_{k,\infty}$. 
        
        Note that the natural map 
        $j : \limaou V_k \rightarrow \limos \omax(V_k)$ is a unital positive map. 
        By \cite[Theorem~3.22]{PaulsenTodorovTomforde}, 
        $j : \omax(\limaou V_k) \rightarrow \limos \omax(V_k)$ 
        is a unital completely positive map. 
        Finally note that 
        $$ j \circ \iota \circ \tilde{\phi}_{k,\infty} = j \circ \hat{\phi}_{k,\infty} = \tilde{\phi}_{k,\infty}$$ and 
        $$\iota \circ j \circ \hat{\phi}_{k,\infty} = \iota \circ \tilde{\phi}_{k,\infty} = \hat{\phi}_{k,\infty}.$$
        It follows that $\iota$ is a complete order isomorphism.
    \end{proof}

    \begin{remark}
    {\rm 
        Let $\omax: {\cataou} \rightarrow {\catos}$ be the functor sending
        $V$ to $\omax(V)$. As pointed out 
        in Section~\ref{sec:pre.omin and omax}, $\omax$ is a left adjoint to the forgetful functor ${\bf F}: \catos \rightarrow \cataou$. Thus, 
        Theorem~\ref{th_omax} is a consequence of 
        the well-known fact that left adjoints commute with colimits \cite{MacLane}.
        We have decided to include a proof relying on the features of the considered categories since 
        it clarifies the concrete workings in the case of interest. }
    \end{remark}

\subsection{Inductive limits of universal C*-algebras} \label{sect: universal c*Alg}

      In this section, we consider the universal C*-algebra of an inductive limit operator system;
     we show in Theorem \ref{thm:Cu(S)} that $C_u^*$ commutes with $\limos$ 
     when the connecting maps are complete order embeddings. 
     The result is well-known in the case of closed operator systems (see \cite[Proposition 2.4]{LuthraKumar}).
     We have decided to include complete arguments in order to keep the exposition self-contained.  
  
The following lemma was established in \cite{KirchbergWassermann}.

      \begin{lemma}
      [\cite{KirchbergWassermann}]\label{lem:op system coi then univer calg *iso}
  Let $\cl S$ and $\cl T$ be operator systems with universal C*-algebras $(C_u^*(\cl S), \iota_{\cl S})$ and $(C_u^*(\cl T), \iota_{\cl T})$,
  respectively, and let $\phi:\cl S \rightarrow \cl T$ be a unital complete order embedding. 
  Then the 
  *-homomorphism $\widetilde{\phi}: C_u^*(\cl S) \rightarrow C_u^*(\cl T)$ 
  with the property that $\widetilde{\phi} \circ \iota_{\cl S} = \iota_{\cl T} \circ \phi$ is injective. 
    \end{lemma}

Clealry, if 
$\cl S_1\stackrel{\phi_1}{\longrightarrow} \cl S_2 \stackrel{\phi_2}{\longrightarrow} \cl S_3 \stackrel{\phi_3}{\longrightarrow} \cl S_4 \stackrel{\phi_4}{\longrightarrow} \cdots$
        is an inductive system in $\catos$ then
        \begin{equation*} \label{eq:univ1}
            C^*_u(\cl S_1)
            \stackrel{\widetilde{\phi}_1}{\longrightarrow}
            C^*_u(\cl S_2)
            \stackrel{\widetilde{\phi}_2}{\longrightarrow}
            C^*_u(\cl S_3)
            \stackrel{\widetilde{\phi}_3}{\longrightarrow}
            C^*_u(\cl S_4)
            \stackrel{\widetilde{\phi}_4}{\longrightarrow}
            \cdots
        \end{equation*}
        is an inductive system in $\catcalg$.
    Let
        $\pi_{k} :\cl C^*_u(\cl S_k) \rightarrow  \limcalg C_u^*({\cl S_k})$ be the canonical unital *-homomorphism,
    $k \in \bb{N}$.

    \begin{theorem} \label{thm:Cu(S)}
               Let $\cl S_1\stackrel{\phi_1}{\longrightarrow} \cl S_2 \stackrel{\phi_2}{\longrightarrow} \cl S_3 \stackrel{\phi_3}{\longrightarrow} \cl S_4 \stackrel{\phi_4}{\longrightarrow} \cdots$
        be an inductive system in $\catos$ such that each $\phi_k$ is a unital complete order embedding. Then
         $C^*_u(\limos \cl S_k)$ is *-isomorphic to $\limcalg C^*_u(\cl S_k)$.
    \end{theorem}
    \begin{proof}
        Set $\cl S_{\infty} = \limos \cl S_k$ and 
    let $\iota_{\los{\cl S}} :  \cl S_{\infty} \to C^*_u(\cl S_{\infty})$ be the canonical embedding.
    Consider the following commutative diagram
    \begin{equation}\label{eq:univ2000}
        \begin{CD}
            \cl S_1 @>{\phi_1}>> \cl S_2 @>{\phi_2}>> \cl S_3 @>{\phi_3}>> \cl S_4 @>{\phi_4}>> \cdots \\
            @V{\iota_1}VV @V{\iota_2}VV @V{\iota_3}VV @V{\iota_4}VV \\
             C^*_u(\cl S_1)
             @>{\widetilde{\phi}_1}>> C^*_u(\cl S_2) @>{\widetilde{\phi}_2}>>  C^*_u(\cl S_3) @>{\widetilde{\phi}_3}>> C^*_u(\cl S_4) @>{\widetilde{\phi}_4}>> \cdots.
        \end{CD}
    \end{equation}
    By Lemma~\ref{lem:op system coi then univer calg *iso},
    all maps in (\ref{eq:univ2000}) are 
    unital complete order embeddings. 
    By Remark~\ref{rem:all maps coi}, there exists a unique unital complete order embedding 
    $\iota : \cl S_{\infty} \rightarrow \limos C_u^*({\cl S_k})$ such that
    \begin{equation} \label{eq:univ4}
        \iota \circ \phi_{k, \infty} = \pi_{k} \circ \iota_k, \ \ \ k \in \bb{N}.
    \end{equation}
By Proposition~\ref{p_opsysst}, the natural map 
        $\id: \limos C_u^*({\cl S_k}) \rightarrow \limcalg C_u^*({\cl S_k})$
    is a unital complete order embedding; thus, 
        $\iota : \cl S_{\infty} \rightarrow \limcalg C_u^*({\cl S_k})$
    is a unital complete order embedding.

    By Proposition~\ref{prop:universal prop universal calg}, there exists a unique unital *-homomorphism
        $$\nu: C_u^*(\cl S_{\infty}) \rightarrow \limcalg C_u^*({\cl S_k})$$
    such that
    \begin{equation}\label{eq:univ5}
        \nu \circ \iota_{\los{\cl S}} = \iota.
    \end{equation}
    Note that
        $\iota_{\los{\cl S}} \circ \phi_{k,\infty} : \cl S_k \rightarrow C_u^*(\cl S_{\infty})$
    is a unital completely order isomorphic embedding, $k \in \bb{N}$. 
    By Proposition~\ref{prop:universal prop universal calg}, there exists a  unital *-homomorphism
        $$\widetilde{\iota_{\los{\cl S}} \circ \phi_{k,\infty}} :  C_u^*(\cl S_k) \rightarrow C_u^*(\cl S_{\infty})$$
    such that
    \begin{equation} \label{eq:univ6}
        (\widetilde{\iota_{\los{\cl S}} \circ \phi_{k,\infty}} ) \circ \iota_k = \iota_{\los{\cl S}} \circ \phi_{k,\infty}, \ \ \ k \in \bb{N}.
    \end{equation}
    By (\ref{eq:univ6}),
    \begin{equation*}
    \begin{split}
       (\widetilde{\iota_{\los{\cl S}} \circ \phi_{k+1,\infty}} ) &\circ  \widetilde{\phi}_k \circ \iota_k
                = (\widetilde{\iota_{\los{\cl S}} \circ \phi_{k+1,\infty}} ) \circ \iota_{k+1} \circ \phi_k \\
                &=  \iota_{\los{\cl S}} \circ \phi_{k+1,\infty} \circ \phi_k
                =  \iota_{\los{\cl S}} \circ \phi_{k,\infty}
                =  (\widetilde{\iota_{\los{\cl S}} \circ \phi_{k,\infty}} ) \circ \iota_k \\
    \end{split}
    \end{equation*}
    for all $k \in \bb{N}$. By the universal property of the inductive limit in the category of C*-algebras, there exists a unique unital *-homomorphism
        $\mu: \limcalg C_u^*(\cl S_k) \rightarrow C_u^*(\cl S_{\infty})$
    such that
    \begin{equation} \label{eq:univ7}
        \mu \circ \pi_{k}  = (\widetilde{\iota_{\los{\cl S}} \circ \phi_{k,\infty}} ), \ \ \  k \in \bb{N}.
    \end{equation}

Note that $\mu \circ \nu = \id_{C_u^*(\cl S_{\infty})}$ and $\nu \circ \mu = \id_{\limcalg C_u^*(\cl S_k)}$. Indeed, by (\ref{eq:univ4}), (\ref{eq:univ5}), (\ref{eq:univ6}) and (\ref{eq:univ7}),
    \begin{eqnarray*}
        \mu \circ \nu \circ \iota_{\los{\cl S}} \circ \phi_{k,\infty}
       & = & 
       \mu \circ \iota \circ \phi_{k,\infty} = \mu \circ \pi_{k} \circ \iota_k
        = (\widetilde{\iota_{\los{\cl S}} \circ \phi_{k, \infty}}) \circ \iota_k\\
        & = &  \iota_{\los{\cl S}} \circ \phi_{k,\infty}
      \end{eqnarray*}
    and
    \begin{eqnarray*}
        \nu \circ \mu \circ \pi_{k} \circ \iota_k
       & = & \nu \circ (\widetilde{\iota_{\los{\cl S}} \circ \phi_{k,\infty}}) \circ \iota_k
        = \nu \circ \iota_{\los{\cl S}} \circ \phi_{k,\infty} = \iota \circ \phi_{k,\infty}\\
        & = & \pi_{k} \circ \iota_k.
      \end{eqnarray*}
    Since $\mu\circ \nu$ and $\nu\circ \mu$ coincide with the identities on dense operator systems, generating the 
    corresponding C*-algebras, we have that $\mu$ is a *-isomorphism.
    Finally, note that $\mu$ is unital, since 
    $$\mu \circ \pi_{k} \circ \iota_k(e_k) = (\widetilde{\iota_{\los{\cl S}} \circ \phi_{k,\infty}}) \circ \iota_k(e_k) 
    = \iota_{\los{\cl S}} \circ \phi_{k,\infty}(e_k).$$
    \end{proof}

    \begin{corollary} \label{cor:universalCalgCor}
        Let
        $\cl S_1\stackrel{\phi_1}{\longrightarrow} \cl S_2 \stackrel{\phi_2}{\longrightarrow} \cl S_3 \stackrel{\phi_3}{\longrightarrow} \cl S_4 \stackrel{\phi_4}{\longrightarrow} \cdots$
        be an inductive system in $\catos$ such that each $\phi_k$ is a unital completely order isomorphic embedding. Then
        $\limos C^*_u(\cl S_k)$ is unitally completely order isomorphic to an operator subsystem of 
        $C^*_u(\limos \cl S_k)$.
    \begin{proof}
        By Theorem~\ref{p_opsysst}, $\id: \limos C^*_u(\cl S_k) \rightarrow \limcalg C^*_u(\cl S_k)$ 
        is a unital completely order isomorphic embedding.
         By Theorem~\ref{thm:Cu(S)}, there exists a unital complete order isomorphism
        $\mu: \limcalg C^*_u(\cl S_k) \rightarrow C^*_u(\limos \cl S_k)$. Therefore $\mu \circ {\id}: \limos C^*_u(\cl S_k) \rightarrow C^*_u(\limos \cl S_k)$ is a unital completely order isomorphic embedding.
    \end{proof}
    \end{corollary}

\subsection{Quotients of inductive limits of operator systems} \label{sect:Quotients}

In this subsection, we relate inductive limits with the quotient theory of operator systems. 
We first recall the basic facts about quotient operator systems, as 
developed in \cite{KavrukPaulsenTodorovTomforde}. 

    Let $\cl S$ be an operator system and let $J \subseteq \cl S$  be a subspace. If there exists an operator system $\cl T$ and a unital completely positive map $\phi: \cl S \rightarrow \cl T$ such that $J = \ker \phi$, then we say that $J$ is a {\it kernel}. 
    If $J$ is a kernel, we let $q : \cl S \to \cl S/J$ be the quotient map and 
    equip the quotient vector space $\cl S/J$ with the involution given by $(x + J)^* = x^* + J$. 
    For $n\in \bb{N}$, let
        \[
    \begin{split} \index{quotient operator system}
        C_n(\cl S / J) = \Big\lbrace (x_{i,j}+ J) &\in M_n(\cl S/ J):  \forall r > 0 ~\exists \ k_{i,j} \in J  
        \\
        & \text{~such that}~ r e^{(n)} + (x_{i,j} + k_{i,j})_{i,j} \in M_n(\cl S)^+\Big\rbrace.
    \end{split}
    \]
It was shown in \cite[Section~3]{KavrukPaulsenTodorovTomforde} that 
$(\cl S/J, \{ C_n(\cl S / J)\}_{n \in \bb{N}}, e+J)$ is an operator system
(called henceforth a \emph{quotient operator system}); moreover, the following holds:

    \begin{theorem}\label{thm:quotient op sys universal property}
        Let $\cl S$ and $\cl T$ be operator systems and let $J$ be a kernel in $\cl S$. If $\phi: \cl S \rightarrow \cl T$ is a unital completely positive map with $J \subseteq \ker \phi$ then the map 
        $\widetilde{\phi}: \cl S/ J \rightarrow \cl T$, 
        defined by the identity $\widetilde{\phi} \circ q = \phi$, is unital and completely positive.

        Furthermore, if $\cl P$  is an operator system and $\psi: \cl S \rightarrow \cl P$ is a unital completely positive map such that whenever $\cl T$ is an operator system and $\phi: \cl S \rightarrow \cl T$ is a unital completely positive map with $J \subseteq \ker \phi$ there exists a unique unital completely positive map $\widetilde{\phi}: \cl P \rightarrow \cl T$ 
        with the property that $\widetilde{\phi} \circ \psi = \phi$, 
        then there exists a complete order isomorphism $\varphi: \cl P \rightarrow \cl S/J$ such that $\varphi \circ \psi = q$.
    \end{theorem}

    If $\cl X$ is a (not necessarily complete) operator space and $Y$ is a closed subspace of $\cl X$, 
    then the quotient $\cl X/Y$ has a canonical operator space structure given by assigning $M_n(\cl X/Y)$ the norm arising from the identification $M_n(\cl X/Y) = M_n(\cl X)/M_n(Y)$, that is, by setting
    \begin{equation} \label{eq:quotient1}
      \norm{(x_{i,j}+Y)}_{M_n(\cl X/Y)} = \inf\big \lbrace\norm{x_{i,j} + y_{i,j}}_{M_n(\cl X)} : y_{i,j} \in Y\big \rbrace, ~~~(x_{i,j})  \in M_n(\cl X).
    \end{equation}
    If $\cl S$ is an operator system and $J$ is a kernel, then $\cl S/J$ can be equipped, 
    on one hand, with the operator space structure inherited from the quotient operator system $\cl S/J$
    and, on the other hand, with the operator space structure given by (\ref{eq:quotient1}). It is proved in \cite[Section~4]{KavrukPaulsenTodorovTomforde} that the matrix norms obtained 
    {\it via} these two methods are in general distinct. 
    If $J$ is a kernel in $\cl S$ such that the operator space quotient and the operator system quotient are completely isometric then we call $J$ {\it completely biproximinal}. 
    
\medskip

Suppose that 
$\cl S_1\stackrel{\phi_1}{\longrightarrow} \cl S_2 \stackrel{\phi_2}{\longrightarrow} \cl S_3 \stackrel{\phi_3}{\longrightarrow} \cl S_4 \stackrel{\phi_4}{\longrightarrow} \cdots$ is an inductive system in~$\catos$ and that,
for each $k \in \bb{N}$, $J_k$ is a kernel in $\cl S_k$ such that $\phi_{k}(J_k) \subseteq J_{k+1}$. 
Let $q_k : \cl S_k\to \cl S_k/J_k$ be the quotient map. 
By Theorem~\ref{thm:quotient op sys universal property}, there is a natural inductive system in $\catos$,
  \begin{equation}\label{eq:quotient11}
  \cl S_1/J_1\stackrel{\psi_1}{\longrightarrow} \cl S_2/J_2 \stackrel{\psi_2}{\longrightarrow} \cl S_3/J_3 \stackrel{\psi_3}{\longrightarrow} \cl S_4/J_4 \stackrel{\psi_4}{\longrightarrow} \cdots,
\end{equation}
such that 
        \begin{equation}\label{eq_ker_1}
            \psi_k \circ q_k = q_{k+1} \circ \phi_k, \ \ \ k \in \bb{N}.
        \end{equation}

   In this subsection we prove that if each of the $J_k$ is completely biproximinal, then the inductive limit of (\ref{eq:quotient11}) is a quotient operator system.

    \begin{lemma}\label{lem:limJ_i_kernel}
        Let $\cl S_1\stackrel{\phi_1}{\longrightarrow} \cl S_2 \stackrel{\phi_2}{\longrightarrow} \cl S_3 \stackrel{\phi_3}{\longrightarrow} \cl S_4 \stackrel{\phi_4}{\longrightarrow} \cdots$
        be an inductive system in $\catos$. 
        For each $k \in \bb{N}$, let $J_k$ be a completely biproximinal kernel in $\cl S_k$ such that $\phi_{k}(J_k) \subseteq J_{k+1}$.
        Then $\underrightarrow{\lim}J_k \stackrel{def}{=} \overline{\cup_{k \in \mathbb{N}}\phi_{k, \infty}(J_k)}$ is a kernel in $\limos \cl S_k$.
            \end{lemma}

    \begin{proof}
    Set $\cl S_{\infty} = \limos \cl S_k$ and 
$J = \underrightarrow{\lim} J_k$; clearly, $J$ is a closed subspace of $\cl S_{\infty}$. 
Note that 
$q_{k+1} \circ \phi_k : \cl S_k \rightarrow \cl S_{k+1} /J_{k+1}$ is a unital completely positive map.
        Consider the commuting diagram
        \[
            \begin{CD}
                \cl S_1 @>{\phi_1}>> \cl S_2 @>{\phi_2}>> \cl S_3 @>{\phi_3}>> \cl S_4 @>{\phi_4}>> \cdots \\
                @V{q_1}VV @V{q_2}VV @V{q_3}VV @V{q_4}VV \\
                  \cl S_1/ J_1
                 @>{\psi_1}>> \cl S_2/ J_2
                 @>{\psi_2}>> \cl S_3/ J_3
                 @>{\psi_3}>> \cl S_4/ J_4
                 @>{\psi_4}>> \cdots.
            \end{CD}
        \]
        By Theorem~\ref{thm:universal property_operator_system2}, there exists a (unique) 
        unital completely positive map 
        $q: \cl S_{\infty} \rightarrow \limos \left(\cl S_k / J_k\right)$,
        such that
        \begin{equation}\label{eq_ker_2}
            q \circ \phi_{k, \infty} = \psi_{k, \infty} \circ q_k, \ \ \ k \in \bb{N}.
        \end{equation}
        We show that $\ker q = J$.
        Since $\ker q$ is closed, in order to prove that $J \subseteq \ker q$,
        it suffices to show that $\cup_{k \in \mathbb{N}}\phi_{k, \infty}(J_k) \subseteq \ker q.$ 
        But, if $y_k \in J_k$ then $q \circ \phi_{k,\infty}(y_k) = \psi_{k,\infty} \circ q_k (y_k) = 0$. 
        Now suppose that 
        $\phi_{k,\infty}(s_k) \in \ker q$ for some $s_k\in \cl S_k$;
        then $\psi_{k,\infty} \circ q_k (s_k) = q \circ \phi_{k, \infty}(s_k) = 0$. 
        By Proposition \ref{prop:norm for inductive limit operator system},
        \[
          \lim_{m \rightarrow \infty} \norm{q_m \circ \phi_{k,m}(s_k)}_{\cl S_m/J_m} =   \lim_{m \rightarrow \infty} \norm{\psi_{k,m} \circ q_k(s_k)}_{\cl S_m/J_m}  =0.
        \]
        For $l \in \bb{N}$, let $m_l \in \bb{N}$ be such that
        \[
            \norm{q_m \circ \phi_{k,m}(s_k)}_{\cl S_m/J_m} < \frac{1}{l}, \ \ \ m \geq m_l.
        \]
Since $J_{m_l}$ is completely biproximinal, there exists $y_{m_l} \in J_{m_l}$ such that
        \[
            \norm{\phi_{k,m_l}(s_k) + y_{m_l}}_{\cl S_{m_l}} < \frac{1}{l}.
        \]
                The map $\phi_{m_l, \infty}$ is unital and completely positive; therefore it is contractive
        and hence, for all $l \in \bb{N}$,
        \[
                \norm{\phi_{m_l, \infty}(\phi_{k,m_l}(s_k) + y_{m_l})}_{\cl S_{\infty}} 
                \leq \norm{\phi_{k,m_l}(s_k) + y_{m_l}}_{\cl S_{m_l}} < \frac{1}{l}.
        \]
Thus, $\phi_{m_l, \infty}(y_{m_l})\in J$ and 
$\phi_{m_l, \infty}(y_{m_l})\to_{l\to\infty} \phi_{k,\infty}(s_k)$, showing that 
$\ker q \subseteq J.$
    \end{proof}

In view of Lemma \ref{lem:limJ_i_kernel}, the operator system 
$(\limos \cl S_k) / (\underrightarrow{\lim} J_k)$ is well-defined. We let 
$\gamma : \limos \cl S_k \to (\limos \cl S_k) / (\underrightarrow{\lim} J_k)$ be the corresponding quotient map.

    \begin{theorem}\label{Ind_Lim_Quotient}
        Let \index{quotient operator system}
        $\cl S_1\stackrel{\phi_1}{\longrightarrow} \cl S_2 \stackrel{\phi_2}{\longrightarrow} \cl S_3 \stackrel{\phi_3}{\longrightarrow} \cl S_4 \stackrel{\phi_4}{\longrightarrow} \cdots$
        be an inductive system in $\catos$. 
        Let $J_k$ be a completely biproximinal kernel in $\cl S_k$ such that 
        $\phi_{k}(J_k) \subseteq J_{k+1}$, $k \in \bb{N}$.
        Then there exists a unital complete order isomorphism $\rho :  
        \limos \left(\cl S_k/ J_k \right) \to (\limos \cl S_k) / (\underrightarrow{\lim} J_k)$ such that 
        $$\rho \circ \psi_{k,\infty} \circ q_k = \gamma \circ \phi_{k,\infty}, \ \ k\in \bb{N}.$$ 
            \end{theorem}

    \begin{proof}
    Set $\cl S_{\infty} = \limos \cl S_k$. 
        Let $\cl T$ be an operator system and 
        $\theta : \cl S_{\infty} \rightarrow \cl T$
        be a unital completely positive map
        such that $\underrightarrow{\lim}J_k \subseteq \ker \theta$;
        then $\theta \circ \phi_{k, \infty}: \cl S_k \rightarrow \cl T$ is a unital completely positive map, $k \in \mathbb{N}$.
        Let $k \in \bb{N}$ and suppose $y_k \in J_k$; 
        by definition, $\phi_{k, \infty}(y_k) \in \underrightarrow{\lim} J_k$ and so $\theta \circ \phi_{k, \infty}(y_k) = 0$. 
        Thus, $J_k \subseteq \ker (\theta \circ \phi_{k, \infty})$. By Theorem~\ref{thm:quotient op sys universal property}, there exists a unique unital completely positive map $\big(\widetilde{\theta \circ \phi_{k, \infty}}\big): \cl S_k / J_k \rightarrow \cl T$ such that
        \begin{equation}\label{eq_ker_3}
            \big(\widetilde{\theta \circ \phi_{k, \infty}}\big) \circ q_k=\theta \circ \phi_{k, \infty}, \ \ \ k \in \bb{N}.
        \end{equation}
        By (\ref{eq_ker_1}) and (\ref{eq_ker_3}),
        \[
            \begin{split}
                \big(\widetilde{\theta \circ \phi_{k+1, \infty}}\big) \circ \psi_k \circ q_k
                &= \big(\widetilde{\theta \circ \phi_{k+1, \infty}}\big) \circ q_{k+1} \circ \phi_k
                = \theta  \circ \phi_{k+1, \infty} \circ \phi_k \\
                &= \theta \circ \phi_{k, \infty}
                = \big(\widetilde{\theta \circ \phi_{k, \infty}}) \circ q_k \\
            \end{split}
        \]
        for every $k \in \bb{N}$.
        By Theorem~\ref{thm:universal property_operator_system2}, there exists a unique unital completely positive map $\widetilde{\theta}: \limos (\cl S_k/ J_k) \rightarrow \cl T$ such that
        \begin{equation}\label{eq_ker_4}
            \widetilde{\theta} \circ \psi_{k, \infty} = \big(\widetilde{\theta \circ \phi_{k, \infty}}\big), \ \ \ k \in \bb{N}.
        \end{equation}
        By (\ref{eq_ker_2}), (\ref{eq_ker_3}) and (\ref{eq_ker_4}), 
        \[
                \widetilde{\theta} \circ q \circ \phi_{k, \infty}
                = \widetilde{\theta} \circ \psi_{k, \infty} \circ q_k
                = \big(\widetilde{\theta \circ \phi_{k, \infty}}\big) \circ q_k
                = \theta \circ \phi_{k, \infty}, \ \ \ k \in \bb{N},
        \]
        where $q: \cl S_{\infty} \rightarrow \limos \left(\cl S_k / J_k\right)$ is the map defined through (\ref{eq_ker_2}). 
        Thus, $\widetilde{\theta} \circ q = \theta$. By Theorem~\ref{thm:quotient op sys universal property}, 
        there exists a unital complete order isomorphism $\rho: \limos \left({\cl S_k / J_k}\right) \rightarrow ({\limos \cl S_k})/({\underrightarrow{\lim} J_k})$
such that $\rho\circ q = \gamma$. This implies that 
$\rho\circ q \circ \phi_{k,\infty} = \gamma \circ \phi_{k,\infty}$
which, by virtue of (\ref{eq_ker_2}), means that 
$\rho\circ \psi_{k,\infty} \circ q_k = \gamma \circ \phi_{k,\infty}$, $k\in \bb{N}$.
    \end{proof}


\subsection{Inductive limits and tensor products} \label{sect:tensor products}

    Let
    \begin{equation}\label{eq:tensor product1}
        \cl S_1\stackrel{\phi_1}{\longrightarrow} \cl S_2 \stackrel{\phi_2}{\longrightarrow} \cl S_3 \stackrel{\phi_3}{\longrightarrow} \cl S_4 \stackrel{\phi_4}{\longrightarrow} \cdots
    \end{equation}
    be an inductive system in $\catos$. Let $\cl T$ be an operator system; for any functorial operator system tensor product $\mu$, we may define the following inductive system in $\catos$:
    \begin{equation}\label{eq:tensor product2}
        \cl S_1 \otimes_{\mu} \cl T \stackrel{\phi_1 \otimes \id_{\cl T}}{\longrightarrow}
        \cl S_2 \otimes_{\mu} \cl T \stackrel{\phi_2 \otimes \id_{\cl T}}{\longrightarrow}
        \cl S_3 \otimes_{\mu} \cl T \stackrel{\phi_3 \otimes \id_{\cl T}}{\longrightarrow}
        \cl S_4 \otimes_{\mu} \cl T \stackrel{\phi_4 \otimes \id_{\cl T}}{\longrightarrow}
        \cdots.
    \end{equation}
    We are interested to know if $\limos (\cl S_k \otimes_{\mu} \cl T)$ is completely order isomorphic to $(\limos\cl S_k) \otimes_{\mu} \cl T$. 
    We first discuss the canonical linear isomorphism between these vector spaces.

    Recalling the notation from Subsection~\ref{sect:inductive limits AOU spaces}, 
    let $N$ be the null space for the inductive system~(\ref{eq:tensor product1}) and let $N_{\mu}$ be the null space for the inductive system~(\ref{eq:tensor product2}). 
Let $\psi_{k} = \phi_{k} \otimes \id_{\cl T}$ and 
$\psi_{k,\infty}: \cl S_k \otimes_{\mu} \cl T \rightarrow \limos (\cl S_k \otimes_{\mu} \cl T)$ be the unital completely positive map 
associated to the inductive system (\ref{eq:tensor product2}).

    \begin{lemma}\label{lem:notation tensor product of inductive limit}
     If $x \in (\limos \cl S_k)\odot \cl T$ 
     then there exist $k, n \in \bb{N}$, $s_{k}^i \in \cl S_k$ and $t^{i} \in \cl T$, $1 \leq i \leq n$, such that
     the set $\{t^i\}_{i=1}^n$ is linearly independent and 
              $$x = \sum_{i=1}^n \phi_{k,\infty} (s_k^i) \otimes t^i. $$
    \begin{proof}
        Since $\limos \cl S_k = \cup_{k\in \bb{N}} \phi_{k,\infty}(\cl S_k)$, 
        there exists $n \in \bb{N}$, $k_i\in\bb{N}$, $s_{k_i}\in \cl S_{k_i}$ and $t^i\in \cl T$, $i = 1,\dots,n$, 
        such that
              $x = \sum_{i=1}^n \phi_{k_i,\infty} (s_{k_i}) \otimes t^i.$
              Let $k= \max\{k_i : 1 \leq i \leq n\}$ and $s_k^i = \phi_{k_i, k}(s_{k_i})$, $i = 1,\dots,n$.
Choosing $n$ to be minimal with this property ensures that $\{t^i\}_{i = 1}^n$ is linearly independent. 
    \end{proof}
    \end{lemma}

    \begin{proposition} \label{prop:TensorProdCommute-LinearSpaces}
        Let
        $\cl S_1\stackrel{\phi_1}{\longrightarrow} \cl S_2 \stackrel{\phi_2}{\longrightarrow} \cl S_3 \stackrel{\phi_3}{\longrightarrow} \cl S_4 \stackrel{\phi_4}{\longrightarrow} \cdots$
        be an inductive system in~$\catos$. Let $\cl T$ be an operator system and $\mu$ be a functorial operator system tensor product. 
        Then the mapping 
        $\widetilde{\alpha}: (\limos \cl S_k) \odot \cl T \rightarrow \limos (\cl S_k \otimes_{\mu} \cl T)$ given by
        \begin{equation}\label{eq:alpha equation}
          \widetilde{\alpha} \circ (\phi_{k,\infty} \otimes {\id}_{\cl T}) = \psi_{k,\infty}, \ \ \ k\in \bb{N},
        \end{equation}
        is a well-defined linear bijection.
    \end{proposition}

    \begin{proof}
        Suppose that $(\phi_{k,\infty}(s_k),t_1 )=(\phi_{l,\infty}(s_l),t_2)$
        for some $s_k \in \cl S_k, s_l\in \cl S_l$ where $k<l$ and $t_1, t_2 \in \cl T$. Then 
        $\phi_{l,\infty}(\phi_{k,l}(s_k) - s_l) = 0$ 
        and $t_1 = t_2$.
         By Proposition \ref{prop:norm for inductive limit operator system}, 
         $\lim_{p \rightarrow \infty}\norm{\phi_{l,p}(\phi_{k,l}(s_k) - s_l)}_{\cl S_p}=0$ and thus
        \[
        \begin{split}
           &  \lim_{p \rightarrow \infty}\norm{\psi_{l,p}(\psi_{k,l}((s_k \otimes t_1) - s_l \otimes t_2))}_{\cl S_p \otimes_{\mu} \cl T} \\
            &=\lim_{p \rightarrow \infty}\norm{\psi_{l,p}((\phi_{k,l}(s_k)-s_l) \otimes t_1)}_{\cl S_p \otimes_{\mu} \cl T} \\
            &= \lim_{p \rightarrow \infty}\norm{\phi_{l,p}(\phi_{k,l}(s_k) - s_l) \otimes t_1}_{\cl S_p \otimes_{\mu} \cl T} \\
            &\leq \norm{t_1}_{\cl T} \lim_{p \rightarrow \infty}\norm{\phi_{l,p}(\phi_{k,l}(s_k) - s_l)}_{\cl S_p}
            = 0,
          \end{split}
        \]
        where the last inequality follows from \cite[Proposition~3.4]{KavrukPaulsenTodorovTomforde2}.
        By Proposition \ref{prop:norm for inductive limit operator system},
$\psi_{l,\infty}(\psi_{k,l}(s_k\otimes t_1) - s_l\otimes t_2) = 0$ and hence $\psi_{k,\infty}(s_k\otimes t_1) = \psi_{l,\infty}(s_l\otimes t_2)$. 
It follows that the map
        $\alpha: (\limos \cl S_k) \times \cl T \rightarrow \limos (\cl S_k \otimes_{\mu} \cl T)$, given by
              $  \alpha(\phi_{k,\infty}(s_k),t )= \psi_{k,\infty}(s_k \otimes t)$, is well-defined. 
The map $\alpha$ is clearly bilinear, and its linearisation 
 $\widetilde{\alpha}: (\limos\cl S_k) \odot \cl T \rightarrow \limos (\cl S_k \otimes_{\mu} \cl T)$ satisfies
        \[
                \widetilde{\alpha} \left(\phi_{k,\infty}(s_k) \otimes t\right) = \psi_{k,\infty}(s_k \otimes t), \ \ s_k\in \cl S_k, t\in \cl T, k\in \bb{N}.
        \]

        We show that $\widetilde{\alpha}$ is bijective. To show that $\widetilde{\alpha}$ is surjective, suppose that
        $y \in \limos (\cl S_k \otimes_{\mu} \cl T)$ and write 
        \[
                y = \psi_{k,\infty}\Big( \sum_{i=1}^n s_k^i \otimes t^i \Big),
        \]
        where $s_k^i \in \cl S_k$, $k\in \bb{N}$, and $t^i \in \cl T$, $1 \leq i \leq n$. Then
        \[
                \sum_{i=1}^n \phi_{k,\infty}(s_k^i) \otimes t^i \in (\limos \cl S_k) \odot \cl T
        \]
        and
        \[
        \begin{split}
                \widetilde{\alpha}\Big(\sum_{i=1}^n \phi_{k,\infty}(s_k^i) \otimes t^i \Big)
                &= \widetilde{\alpha}\circ (\phi_{k,\infty}\otimes\id) \Big(\sum_{i=1}^n s_k^i \otimes t^i \Big)\\
                &= \sum_{i=1}^n \widetilde{\alpha}\big(\phi_{k,\infty}(s_k^i) \otimes t^i \big)
                = \sum_{i=1}^n \psi_{k,\infty}(s_k^i \otimes t^i)\\
                & = \psi_{k,\infty}\Big(\sum_{i=1}^n s_k^i \otimes t^i\Big) = y. \\
        \end{split}
        \]
        To see that $\widetilde{\alpha}$ is injective, 
        let $x \in (\limos \cl S_k) \odot \cl T$ with $\widetilde{\alpha}(x) =0$. 
        Using Lemma~\ref{lem:notation tensor product of inductive limit}, write
                $x = \sum_{i=1}^n \phi_{k,\infty} (s_k^i) \otimes t^i$
        for some $k \in \bb{N}, s_{k}^i \in \cl S_k$, $1 \leq i \leq n$, 
        and a linearly independent family $\{t^{i}\}_{i=1}^n \subseteq \cl T$. 
        Since
        \[
            \widetilde{\alpha} \Big( \sum_{i=1}^n \phi_{k,\infty} (s_k^i) \otimes t^i \Big)
            = \psi_{k,\infty} \left(\sum_{i=1}^n  s_k^i \otimes t^i \right), 
        \]
        it follows by Proposition \ref{prop:norm for inductive limit operator system} that
        \[
                            \lim_{p \rightarrow \infty} \Big\lVert\sum_{i=1}^n \phi_{k,p}(s_k^i) \otimes t^i\Big\rVert_{\cl S_p \otimes_{\mu} \cl T}= \lim_{p \rightarrow \infty} \Big\lVert\psi_{k,p}\Big(\sum_{i=1}^n s_k^i \otimes t^i\Big)\Big\rVert_{\cl S_p \otimes_{\mu} \cl T}=0.
        \]
        Let $W  = {\rm span}\{t^1, t^2, \ldots, t^n\}\subseteq \cl T$
        and define, for each $l = 1, \ldots, n$, a linear functional $f_l : W \rightarrow \mathbb{C}$ by letting
        \[
                    f_l(t^i) =
            \begin{cases}
                1 &\mbox{if } i=l \\
                0 & \mbox{if } i \neq l.
            \end{cases}
        \]
        Each $f_l$ is bounded and may be extended to a bounded functional $\widetilde{f_l}: \cl T \rightarrow \mathbb{C}$.
        It follows from \cite[Proposition~3.7]{KavrukPaulsenTodorovTomforde2} that for any $k \in \mathbb{N}$ and 
        $1 \leq l \leq n$, $\norm{{\id}_{\cl S_k} \otimes \widetilde{f_l}} \leq \norm{\widetilde{f_l}}$. Therefore, for each $l = 1, \ldots, n$,
        \[
            \begin{split}
                \lim_{p \rightarrow \infty} \norm{\phi_{k,p}({s_k^l})}_{\cl S_p}
                &= \lim_{p \rightarrow \infty} \Big\lVert({\id}_{\cl S_k} \otimes \widetilde{f_l})\Big(\sum_{i=1}^{n} \phi_{k,p}(s_k^i) \otimes t^i\Big)\Big\rVert_{\cl S_p}\\
                &\leq \norm{\widetilde{f_l}} \lim_{p \rightarrow \infty} \Big\lVert\sum_{i=1}^n \phi_{k,p}(s_k^i) \otimes t^i\Big\rVert_{\cl S_p \otimes_{\mu} \cl T} =0. \\
            \end{split}
        \]
        By Proposition \ref{prop:norm for inductive limit operator system}, $\phi_{k,\infty}(s_k^l) = 0$ for each
$l =1, \ldots, n$ and hence $x = 0$.
    \end{proof}

    Throughout this section, unless otherwise specified, we let $\widetilde{\alpha}$ denote the map defined 
by (\ref{eq:alpha equation}).

    \begin{remark}\label{rem:nR=nS}
Let $k\in \bb{N}$ and $R \in M_n(\cl S_k \otimes_{\mu}\cl T)$. We have that 
$\ddot{\psi}_{k,\infty}^{(n)}(R) \in M_n(N_{\mu})$ if and only if $(\phi_{k,\infty} \otimes {\id}_{\cl T})^{(n)}(R) = 0$.
    \end{remark}
\begin{proof}
If $R = (r_{i,j})_{i,j} \in M_n(\cl S_k \otimes_{\mu}\cl T)$ and $\ddot{\psi}_{k,\infty}^{(n)}(R) \in M_n(N_{\mu})$ 
then 
$\psi_{k,\infty}(r_{i,j})$ $= 0$ for all $i,j$ and hence, by the injectivity of the map $\widetilde{\alpha}$, established in 
Proposition~\ref{prop:TensorProdCommute-LinearSpaces}, we have that 
$(\phi_{k,\infty} \otimes {\id}_{\cl T})(r_{i,j}) =0$ for all $i,j$. Thus, $(\phi_{k,\infty} \otimes {\id}_{\cl T})^{(n)}(R) = 0$.
Conversely, if $(\phi_{k,\infty} \otimes {\id}_{\cl T})^{(n)}(R) = 0$ then 
$\psi_{k,\infty}(r_{i,j})$ $=$ $\widetilde{\alpha}((\phi_{k,\infty} \otimes {\id}_{\cl T})(r_{i,j})) = 0$ for all $i,j$ and hence 
$\ddot{\psi}_{k,\infty}^{(n)}(R) \in M_n(N_{\mu})$.
\end{proof}


    \begin{theorem}\label{thm:alpa-1 is ucp}
        Let
        $\cl S_1\stackrel{\phi_1}{\longrightarrow} \cl S_2 \stackrel{\phi_2}{\longrightarrow} \cl S_3 \stackrel{\phi_3}{\longrightarrow} \cl S_4 \stackrel{\phi_4}{\longrightarrow} \cdots$
        be an inductive system in $\catos$. Let $\cl T$ be an operator system and $\mu$ be a functorial operator system tensor product.
        Then the inverse 
            $\widetilde{\alpha}^{-1}: \limos (\cl S_k \otimes_{\mu} \cl T) \rightarrow (\limos \cl S_k) \otimes_{\mu} \cl T$ 
            of the map $\widetilde{\alpha}$ 
        is a unital completely positive map.
            \end{theorem}

    \begin{proof}
        Suppose $\psi_{k,\infty}^{(n)}(R) \in M_n(\limos (\cl S_k \otimes \cl T))^+$ for some
        $R \in M_n(\cl S_k \otimes_{\mu} \cl T)_h$, $k\in\bb{N}$.
        Then for every $r > 0$ there exist $l\in \bb{N}$, $P\in \cl S_l \otimes_{\mu} \cl T$ and $m > \max\{k,l\}$ such that 
        $\ddot{\psi}_{l,\infty}^{(n)}(P) \in M_n(N_{\mu})$ and 
        \[
                r (e_m \otimes e_{\cl T})^{(n)} + \psi_{k,m}^{(n)}(R) + \psi_{l,m}^{(n)}(P) \in M_n(\cl S_m \otimes_{\mu}\cl T)^+.
        \]
        By Remark~\ref{rem:nR=nS},
        \[
        \begin{split}
                r&(\phi_{k,\infty}(e_k) \otimes e_{\cl T})^{(n)}+ (\phi_{k,\infty} \otimes {\id}_{\cl T})^{(n)}(R) \\
                &=
                (\phi_{m,\infty} \otimes {\id}_{\cl T})^{(n)}\big(r (e_m \otimes e_{\cl T})^{(n)} + \psi_{k,m}^{(n)}(R) + \psi_{l,m}^{(n)}(P)\big) \\
                &  \in M_n((\limos \cl S_k) \otimes_{\mu} \cl T)^+. \\
        \end{split}
        \]
        Since this holds for all $r>0$, it follows that 
$$(\phi_{k,\infty} \otimes {\id}_{\cl T})^{(n)}(R) \in M_n((\limos \cl S_k) \otimes_{\mu} \cl T)^+.$$ 
Since $\widetilde{\alpha}^{-1} \circ \psi_{k,\infty}= \phi_{k,\infty} \otimes {\id}_{\cl T}$, the proof is complete.
    \end{proof}

    \begin{theorem} \label{thm:TensorCommutesWithIndLimitWhenCOI}
        Let
            $\cl S_1\stackrel{\phi_1}{\longrightarrow} \cl S_2 \stackrel{\phi_2}{\longrightarrow} \cl S_3 \stackrel{\phi_3}{\longrightarrow} \cl S_4 \stackrel{\phi_4}{\longrightarrow} \cdots$
        be an inductive system in~$\catos$ such that each $\phi_k$ is a complete order isomorphism onto its image. Let $\cl T$ be an operator system and $\mu$ be a functorial, injective operator system tensor product.  
        Then the map $\widetilde{\alpha} : (\limos \cl S_k) \otimes_{\mu}
        \cl T \to \limos (\cl S_k \otimes_{\mu} \cl T)$ is a unital complete order isomorphism.
    \end{theorem}

    \begin{proof}
Note that the maps $\phi_{k,\infty}\otimes\id_{\cl T}$, $k\in \bb{N}$, are completely positive and 
$$(\phi_{k+1,\infty}\otimes\id\mbox{}_{\cl T}) \circ (\phi_{k}\otimes\id\mbox{}_{\cl T}) = \phi_{k,\infty}\otimes\id\mbox{}_{\cl T}, 
\ \ k\in \bb{N}.$$
        We will show that the pair
        $$\left((\limos \cl S_k) \otimes_{\mu} \cl T, \{\phi_{k,\infty}\otimes\id\mbox{}_{\cl T}\}_{k\in \bb{N}}\right)$$
        satisfies the universal property of the inductive limit $\limos (\cl S_k \otimes_{\mu} \cl T)$. 
        Suppose that $(\cl R, \lbrace \rho_k \rbrace_{k \in \mathbb{N}})$ is another pair consisting of an operator system and a family of unital completely positive maps $\rho_k: \cl S_k \otimes_{\mu} \cl T \rightarrow \cl R$ such that
        \begin{equation}\label{eq_Ten2}
            \rho_{k+1} \circ \psi_k  = \rho_k, \ \ \ k \in \mathbb{N}.
        \end{equation}
        Suppose that $(\phi_{k, \infty}(s_k), t_1)=(\phi_{l, \infty}(s_l), t_2)$ for some 
        $k, l \in\mathbb{N}, s_k \in \cl S_k, s_l \in \cl S_l$ and $t_1, t_2 \in \cl T$. 
        By Proposition \ref{prop:Eachphi_kOIthenN=0}, 
        there exists $m > \max\lbrace k, l \rbrace$ such that $\phi_{k, m}(s_k) = \phi_{l, m}(s_l)$.
        By (\ref{eq_Ten2}),
        \[
            \begin{split}
                \rho_k(s_k \otimes t_1) &= \rho_m \circ (\phi_{k, m} \otimes {\id}_{\cl T})(s_k \otimes t_1)
                = \rho_m(\phi_{k, m}(s_k) \otimes t_1)\\
                &= \rho_m(\phi_{l, m}(s_l) \otimes t_2)
                = \rho_m \circ (\phi_{l, m} \otimes {\id}_{\cl T})(s_l \otimes t_2)
                = \rho_l(s_l \otimes t_2). \\
            \end{split}
        \]
       It follows that the map $\theta: (\limos \cl S_k) \times \cl T \rightarrow \cl R$, given by
        \[
          \theta  (\phi_{k, \infty}(s_k), t) = \rho_k(s_k \otimes t), \ \ \ k \in \bb{N},
        \]
        is well-defined. 
        Clearly, $\theta$ is bilinear; let 
      $\widetilde{\theta}: (\limos \cl S_k) \otimes_{\mu} \cl T \rightarrow \cl R$ be its linearisation. 
      Thus, $\widetilde{\theta} \circ (\phi_{k,\infty} \otimes {\id}_{\cl T}) = \rho_k$, $k\in \bb{N}$. 
      Since $\rho_k$ is unital, $k\in \bb{N}$, we have that $\widetilde{\theta}$ is unital.      
      
      We check that $\widetilde{\theta}$ is completely positive.
        Suppose that $X \in M_n(\cl S_k \otimes_{\mu} \cl T)$ is such that
        \[
              (\phi_{k,\infty} \otimes {\id}_{\cl T})^{(n)}(X) \in M_n((\limos \cl S_k)\otimes_{\mu} \cl T)^+.
        \]
         By Proposition~\ref{p_coios},
         $\phi_{k,\infty}$ is a unital complete order embedding.
         Since $\mu$ is an injective functorial tensor product, $\phi_{k,\infty} \otimes {\id}_{\cl T}$ is a complete order embedding.
        Therefore $X \in M_n(\cl S_k \otimes_{\mu} \cl T)^+$ and, since $\rho_k$ is completely positive,
        \[
            \widetilde{\theta}^{(n)} \circ (\phi_{k,\infty} \otimes {\id}_{\cl T})^{(n)}(X) = \rho_k^{(n)}(X) \in M_n(\cl R)^+.
        \]
        It follows that $\widetilde{\theta}$ is completely positive, and the proof is complete. 
    \end{proof}

As a direct consequence of Theorem \ref{thm:TensorCommutesWithIndLimitWhenCOI}, we 
obtain the following fact, which was observed in \cite{LuthraKumar} in the case of complete operator systems.

    \begin{corollary}\label{cor: min tensor product commutes when maps are coi}
        Let\index{operator system tensor product!minimal}
            $\cl S_1\stackrel{\phi_1}{\longrightarrow} \cl S_2 \stackrel{\phi_2}{\longrightarrow} \cl S_3 \stackrel{\phi_3}{\longrightarrow} \cl S_4 \stackrel{\phi_4}{\longrightarrow} \cdots$
        be an inductive system in~$\catos$ such that each $\phi_k$ is a complete order isomorphism onto its image, 
        and let $\cl T$ be an operator system.  Then $\limos (\cl S_k \mitp \cl T)$ is unitally completely order isomorphic to $(\limos \cl S_k) \mitp \cl T$.
    \end{corollary}

Although the maximal operator system tensor product is not injective, 
the conclusion of Theorem \ref{thm:TensorCommutesWithIndLimitWhenCOI} still holds for it, as we show 
in the next theorem. We note that, in the case where the connecting maps
are complete order embeddings, this result was first stated in \cite{Li}.

    \begin{theorem} \label{thm:MaxTensorCommutesWithIndLimit}
        Let \index{operator system tensor product!maximal}
        $\cl S_1\stackrel{\phi_1}{\longrightarrow} \cl S_2 \stackrel{\phi_2}{\longrightarrow} \cl S_3 \stackrel{\phi_3}{\longrightarrow} \cl S_4 \stackrel{\phi_4}{\longrightarrow} \cdots$
        be an inductive system in $\catos$ and let $\cl T$ be an operator system. Then $\limos (\cl S_k \mtp \cl T)$ is unitally completely order isomorphic to $(\limos \cl S_k) \mtp \cl T$.
            \end{theorem}

    \begin{proof}
        By Proposition~\ref{prop:TensorProdCommute-LinearSpaces}, 
        $\widetilde{\alpha} : (\limos \cl S_k) \mtp \cl T \to \limos (\cl S_k \mtp \cl T)$ is a linear bijection.
        Set $D_n = (\widetilde{\alpha}^{-1})^{(n)}( M_n(\limos( \cl S_k \mtp \cl T))^+)$, $n\in \bb{N}$.
        By Lemma~\ref{rem:linear map maintain operator system structure}, $\{D_n\}_{n\in \bb{N}}$
                is an operator system structure on $(\limos \cl S_k)\odot \cl T$.
We claim that $\{D_n\}_{n\in \bb{N}}$ is a tensor product operator system structure. 
Suppose that $P\in M_p(\limos \cl S_k)^+$ and $Q\in M_q(\cl T)^+$. 
For every $r > 0$ there exist $k,l\in \bb{N}$, $R\in M_p(\cl S_l)$, $S\in  M_p(\cl S_k)_h$ and $m > \max\{k,l\}$ such that
$\ddot{\phi}_{l,\infty}^{(p)}(R) \in M_p(N)$, $\phi_{k,\infty}^{(n)}(S) = P$ and 
$$\frac{r}{\|Q\|} e_m^{(p)} + \phi_{k,m}^{(p)}(S) + \phi_{l,m}^{(p)}(R) \in M_p(\cl S_m)^+.$$
We have that $S\otimes Q\in M_{pq}(\cl S_k\otimes_{\max} \cl T)_h$,  
$$(\phi_{k,\infty}^{(p)}\otimes\id\mbox{}_{\cl T}^{(q)})(S\otimes Q) = P\otimes Q$$ 
and, by Remark \ref{rem:nR=nS}, 
$\ddot{\psi}_{l,\infty}(R\otimes Q)\in M_{pq} (N_{\mu})$. Moreover,
$$\frac{r}{\|Q\|} e_m^{(p)}\otimes Q + (\phi_{k,m}\otimes\id\mbox{}_{\cl T})^{(pq)}(S\otimes Q) + 
(\phi_{l,m} \otimes\id\mbox{}_{\cl T})^{(pq)}(R\otimes Q)$$
belongs to $M_{pq}(\cl S_m\otimes_{\max}\cl T)^+$,
that is,
$$\frac{r}{\|Q\|} e_m^{(p)}\otimes Q + \psi_{k,m}^{(pq)}(S\otimes Q) + 
\psi_{l,m}^{(pq)}(R\otimes Q) \in M_{pq}(\cl S_m\otimes_{\max}\cl T)^+.$$
Since $Q\leq \|Q\|e_{\cl T}^{(q)}$, we conclude that 
$$r (e_m^{(p)} \otimes e_{\cl T}^{(q)}) + \psi_{k,m}^{(pq)}(S\otimes Q) + 
\psi_{l,m}^{(pq)}(R\otimes Q) \in M_{pq}(\cl S_m\otimes_{\max}\cl T)^+.$$ 
Thus, 
$$r (e_{\infty}^{(p)} \otimes e_{\cl T}^{(q)}) + \psi_{k,\infty}^{(pq)}(S\otimes Q)
\in M_{pq}(\limos( \cl S_k \mtp \cl T)^+).$$ 
Since this holds for every $r > 0$, we have that 
$$\psi_{k,\infty}^{(pq)}(S\otimes Q) \in M_{pq}(\limos( \cl S_k \mtp \cl T)^+).$$
However, 
$\psi_{k,\infty}^{(pq)}(S\otimes Q) = \widetilde{\alpha}^{(pq)}(P\otimes Q)$, and we conclude that 
$P\otimes Q \in D_{pq}$. 

Suppose next that $f : \limos\cl S_k \to M_p$ and $g : \cl T\to M_q$ are unital completely positive maps
and that $L\in D_t$, for some $t\in \bb{N}$. We will show that $(f\otimes g)^{(t)}(L)\in M_{pqt}^+$, 
thus obtaining that $\{D_n\}_{n\in \bb{N}}$ is an operator system tensor product structure on 
$(\limos\cl S_k)\odot \cl T$. 
Let $T = \widetilde{\alpha}^{(t)}(L)$; we have that $T\in M_t(\limos(\cl S_k\otimes_{\max}\cl T))^+$. 
Fix $r > 0$. Then there exist $k,l\in \bb{N}$, $m > \max\{k,l\}$, 
$R\in M_t(\cl S_l)$ and $S\in M_t(\cl S_k\otimes_{\max}\cl T)_h$
such that $\ddot{\psi}_{k,\infty}(S) = T$, $\psi_{l,\infty}(R)\in M_t(N_{\max})$ and 
$$\frac{r}{2} e^{(t)} + \psi_{k,m}^{(t)}(S) + \psi_{l,m}^{(t)}(R)\in M_t(\cl S_m\otimes_{\max}\cl T)^+.$$
By the definition of the maximal operator system structure, there exist 
$a,b\in \bb{N}$, $A\in M_{ab,t}$, $P\in M_a(\cl S_m)^+$ and $Q\in M_b(\cl T)^+$ such that 
$$r e^{(t)} + \psi_{k,m}^{(t)}(S) + \psi_{l,m}^{(t)}(R) = A^*(P\otimes Q)A.$$
The last identity can be rewritten as
$$r e^{(t)} + (\phi_{k,m}^{(t)}\otimes\id\mbox{}_{\cl T})(S) 
+ (\phi_{l,m}^{(t)}\otimes\id\mbox{}_{\cl T})(R) = A^*(P\otimes Q)A.$$
Note that 
$$\widetilde{\alpha}^{(t)}\circ (\phi_{m,\infty}^{(t)}\otimes \id\mbox{}_{\cl T})(\phi_{k,m}^{(t)}\otimes\id\mbox{}_{\cl T})(S) = 
\psi_{k,\infty}^{(t)}(S) = T = \widetilde{\alpha}^{(t)}(L);$$
the injectivity of $\widetilde{\alpha}$ implies that 
$$(\phi_{k,\infty}^{(t)}\otimes \id\mbox{}_{\cl T})(S) = 
(\phi_{m,\infty}^{(t)}\otimes \id\mbox{}_{\cl T})(\phi_{k,m}^{(t)}\otimes\id\mbox{}_{\cl T})(S) = L.$$
Using Remark \ref{rem:nR=nS}, we have that 
\begin{eqnarray*}
& & r I_{pqt} + (f\otimes g)^{(t)}(L)\\
& = & 
(f\otimes g)^{(t)}(r e^{(t)}) + 
(f\otimes g)^{(t)}\circ (\phi_{m,\infty}^{(t)}\otimes\id\mbox{}_{\cl T}) 
((\phi_{k,m}^{(t)}\otimes\id\mbox{}_{\cl T})(S))\\
& + & (f\otimes g)^{(t)}\circ (\phi_{m,\infty}^{(t)}\otimes\id\mbox{}_{\cl T}) (\phi_{l,m}^{(t)}\otimes\id\mbox{}_{\cl T})(R)) \\
& = & (f\otimes g)^{(t)}((\phi_{m,\infty}^{(t)}\otimes\id\mbox{}_{\cl T})(A^*(P\otimes Q)A))\\
& = & (f\otimes g)^{(t)}(A^*(\phi_{m,\infty}^{(a)}(P)\otimes Q)A)\\
& = & A^*(f^{(a)}(\phi_{m,\infty}^{(a)}(P))\otimes g^{(b)}(Q))A \in M_{pqt}^+.
\end{eqnarray*}

        Suppose $\hilb$ is a Hilbert space and
            $\theta: (\limos \cl S_k) \times \cl T \rightarrow \BH$
        is a unital jointly completely positive map.
        Let $\widetilde{\theta}$ denote the linearisation of $\theta$. Then
            $\widetilde{\theta} : (\limos \cl S_k) \mtp \cl T \rightarrow \BH$
        is a unital completely positive map.
        Since
            $\phi_{k, \infty} \otimes {\id}_{\cl T} : \cl S_k \mtp \cl T \rightarrow (\limos \cl S_k) \mtp \cl T$
        is a unital completely positive map, we have that
            $\widetilde{\theta} \circ (\phi_{k, \infty} \otimes {\id}_{\cl T}) : \cl S_k \mtp \cl T \rightarrow \BH$
        is a unital completely positive map, $k \in \mathbb{N}$. Furthermore,
$$\widetilde{\theta} \circ (\phi_{k+1, \infty} \otimes {\id}_{\cl T}) \circ (\phi_k \otimes {\id}_{\cl T}) 
= \widetilde{\theta} \circ (\phi_{k, \infty} \otimes \id), \ \ k \in \bb{N}.$$
       By Theorem~\ref{thm:universal property_operator_system2}, 
       there exists a unique unital completely positive map
            $\eta : \limos (\cl S_k \mtp \cl T) \rightarrow \BH$
        such that
            $\eta \circ \psi_{k, \infty} =\widetilde{\theta} \circ (\phi_{k, \infty} \otimes {\id}_{\cl T}).$
        Thus,
$$\widetilde{\theta} \circ (\phi_{k, \infty} \otimes {\id}_{\cl T}) = \eta \circ \psi_{k, \infty} 
= \eta \circ \widetilde{\alpha} \circ (\phi_{k, \infty} \otimes {\id}_{\cl T}), \ \ k \in \bb{N}.$$
        Therefore $\widetilde{\theta} = \eta \circ \widetilde{\alpha}$; that is, $\widetilde{\theta} \circ \widetilde{\alpha}^{-1} = \eta$. It follows that $\widetilde{\theta} \circ \widetilde{\alpha}^{-1}$ is a unital completely positive map; that is, $\widetilde{\theta}$ is completely positive 
        for the operator system structure $\{D_n\}_{n\in \bb{N}}$. 
By Theorem~\ref{thm:universal property of mtp}, $\widetilde{\alpha}$ is a completely positive map.
    \end{proof}

Our next aim is to identify conditions that guarantee that the inductive limit intertwines the commuting tensor product.

        \begin{lemma}\label{lem:norm closure of op system}
            Let $(\cl S, \{C_n\}_{n \in \bb{N}}, e)$ be an operator system and let $\widehat{\cl S}$ be the completion of $\cl S$. If
            $\widehat{C}_n$ is the completion of $C_n$, $n \in \bb{N}$ then $(\widehat{\cl S}, \{\widehat{C}_n\}_{n \in \bb{N}}, e)$ is an operator system.
            Moreover, if $\rho : \cl S\to \cl B(H)$ is a unital complete isometry then $\widehat{\cl S}$ is 
            unitally completely order isomorphic to the concrete operator system $\overline{\rho(\cl S)}$.
          \end{lemma}
          
          \begin{proof}
Let $\rho : \cl S\to \cl B(H)$ is a unital complete isometry, and let $\cl T = \overline{\rho(\cl S)}$. 
We equip $\cl T$ with the canonical operator system structure arising from its inclusion $\cl T\subseteq \cl B(H)$. 
We claim that 
$M_n(\cl T)^+ = \overline{M_n(\cl S)^+}$, $n\in \bb{N}$. 
It suffices to establish the identity in the case $n = 1$. 
Suppose that $x\in \cl T^+$, $r > 0$, and let $(x_k)_{k\in \bb{N}}\subseteq \cl S_h$ be a sequence such that 
$r I + x = \lim_{k\to\infty} x_k$. By \cite[Theorem~2]{Murphy}, there exists $k_0\in \bb{N}$ such that 
$x_k\geq 0$, $k\geq k_0$. 
It follows that $rI + x\in \overline{\cl S^+}$, for every $r > 0$. Thus, $x\in \overline{\cl S^+}$. 
The statements of the lemma are now evident. 
          \end{proof}



    \begin{lemma} \label{lem:injectivity of mtp with closure}
        Let $\cl S$ and $\cl T$ be an operator systems and let $\widehat{\cl S}$ be the completion of $\cl S$. 
        Then $\id_{\cl S} \otimes \id_{\cl T}: \cl S \mtp \cl T \rightarrow \widehat{\cl S} \mtp \cl T$ is a complete order isomorphism onto its image.
    \begin{proof}
        Fix $n \in \bb{N}$ and suppose that 
        $U \in M_n(\cl S \mtp \cl T) \cap M_n(\widehat{\cl S} \mtp \cl T)^+$. 
        Since the set of hermitian elements is closed, $U = U^*$. 
        For all $r > 0$, we have that 
        $r(e_{\cl S} \otimes e_{\cl T})^{(n)}+ U = \alpha (P^r \otimes Q^r) \alpha^*$ where $\alpha \in M_{n, km}$, $P^r \in M_k(\widehat{\cl S})^+$ and $Q^r \in M_m(\cl T)^+$ for some $k,m \in \bb{N}$. 
By Lemma \ref{lem:norm closure of op system}, 
$P^r = \lim_{l \rightarrow \infty} P_l^r$, for some sequence $(P_l^r)_{l\in \bb{N}} \subseteq M_n(\cl S)^+$. 
Let $X_l^{r} = \alpha (P_l^r \otimes Q^r) \alpha^*$, $l \in \bb{N}$. It follows that $r(e_{\cl S} \otimes e_{\cl T})^{(n)}+ U = \lim_{l \rightarrow \infty} X_l^{r}$ 
        with $X_l^{r} \in M_n(\cl S\mtp \cl T)^+$ for all $r>0$.

        Fix $r > 0$ and choose $l \in \bb{N}$ such that
        \[
          \Bignorm{\frac{r}{2}(e_{\cl S} \otimes e_{\cl T})^{(n)} + U -  X_l^{\frac{r}{2}}}_{M_n(\cl S \mtp \cl T)} < \frac{r}{2}.
        \]
        We have
        \[
          \frac{r}{2}(e_{\cl S} \otimes e_{\cl T})^{(n)} + \frac{r}{2}(e_{\cl S} \otimes e_{\cl T})^{(n)} + U -  X_l^{\frac{r}{2}} \in M_n(\cl S \mtp \cl T)^+.
        \]
        Thus
        $r(e_{\cl S} \otimes e_{\cl T})^{(n)} + U \in   M_n(\cl S \mtp \cl T)^+.$
        Since  this holds for all $r > 0$ and $M_n(\cl S \mtp \cl T)$ is an AOU space, $U \in   M_n(\cl S \mtp \cl T)^+$.
    \end{proof}
    \end{lemma}

In the case the inductive limit is taken in the category of complete operator systems, 
Theorem \ref{thm:ComTensorCommutesWithIndLimit} below follows from \cite[Proposition 4.1]{LuthraKumar}.
In our proof, we also supply some details that were not fully provided in \cite{LuthraKumar}. 
First we need a lemma that may be interesting in its own right.

\begin{lemma}\label{l_univct}
Let $\cl S$ and $\cl T$ be operator systems, and 
$\iota : \cl S\otimes_{\rm c}\cl T\to C_u^*(\cl S\otimes_{\rm c} \cl T)$ and
$j : \cl S\otimes_{\rm c}\cl T\to C_u^*(\cl S)\otimes_{\max} C^*_u(\cl T)$
be the canonical embeddings. Then there exists a 
*-isomorphism $\delta : C_u^*(\cl S)\otimes_{\max} C^*_u(\cl T) \to C_u^*(\cl S\otimes_{\rm c} \cl T)$ such that 
$\delta \circ j = \iota$. 
\end{lemma}

\begin{proof}
Let $H$ be a Hilbert space and $\rho : \cl S \ctp \cl T \to \cl B(H)$ be a unital completely positive map.
Let $\rho_{\cl S} : \cl S\to \cl B(H)$ and $\rho_{\cl T} : \cl T\to \cl B(H)$ be the unital completely positive maps
such that $\rho(x\otimes y) = \rho_{\cl S}(x)\rho_{\cl T}(y)$, $x\in \cl S$, $y\in \cl T$. 
Let $\tilde{\rho}_{\cl S} : C^*_u(\cl S) \to \cl B(H)$ and $\tilde{\rho}_{\cl T} : C^*_u(\cl T) \to \cl B(H)$
be their canonical *-homomorphic extensions. 
Since the ranges of $\rho_{\cl S}$ and $\rho_{\cl T}$ commute, so do the ranges of 
$\tilde{\rho}_{\cl S}$ and $\tilde{\rho}_{\cl T}$. 
Let $\theta : C_u^*(\cl S)\otimes_{\max} C^*_u(\cl T) \to \cl B(H)$ be the *-homomorphism 
given by 
$\theta(x\otimes y) = \tilde{\rho}_{\cl S}(x) \tilde{\rho}_{\cl T}(y)$, $x\in C_u^*(\cl S)$, $y\in C_u^*(\cl T)$.
Note that $\theta \circ j = \rho$. 
Thus, the pair $(C_u^*(\cl S)\otimes_{\max} C^*_u(\cl T), j)$ satisfies the universal property of 
$C_u^*(\cl S\otimes_{\rm c} \cl T)$. The conclusion follows from the uniqueness of the universal C*-algebra. 
\end{proof}

    \begin{theorem} \label{thm:ComTensorCommutesWithIndLimit}
        Let 
        $\cl S_1\stackrel{\phi_1}{\longrightarrow} \cl S_2 \stackrel{\phi_2}{\longrightarrow} \cl S_3 \stackrel{\phi_3}{\longrightarrow} \cl S_4 \stackrel{\phi_4}{\longrightarrow} \cdots$
        be an inductive system in $\catos$ such that each $\phi_k$ is a complete order embedding, and let $\cl T$ be an operator system.
      Assume that the map
     $\phi_k\otimes\id_{\cl T}$ is a complete order embedding of $\cl S_k \ctp \cl T$ into $\cl S_{k+1} \ctp \cl T$, $k\in \bb{N}$. 
        Then $\limos (\cl S_k \ctp \cl T)$ is unitally completely order isomorphic to $(\limos \cl S_k) \ctp \cl T$.
     \end{theorem}

    \begin{proof}
        Let $\iota_{\cl T} : \cl T \to C^*_u(\cl T)$ and $\iota_k : \cl S_k\to C_u^*(\cl S_k)$,
        $k\in \bb{N}$, be the corresponding canonical embeddings. 
        Consider the following inductive system in $\catcalg$, and therefore in $\catos$:
        \begin{equation} \label{eq:ctp1}
             C^*_u(\cl S_1)
            \stackrel{\eta_1}{\longrightarrow}
             C^*_u(\cl S_2)
            \stackrel{\eta_2}{\longrightarrow}
             C^*_u(\cl S_3)
            \stackrel{\eta_3}{\longrightarrow}
             C^*_u(\cl S_4)
            \stackrel{\eta_4}{\longrightarrow}
            \cdots,
        \end{equation}
        where $\eta_k$ is the extension of $\phi_k$, $k \in \bb{N}$, guaranteed by the universal property of the universal C*-algebra. 
        By Lemma~\ref{lem:op system coi then univer calg *iso}, $\eta_k$ is a *-isomorphic embedding for all $k \in \bb{N}$.
        Let
        \[
            \eta_{k,\infty}: C_u^*(\cl S_k) \rightarrow \limos C_u^*(\cl S_k)
        \]
        be the unital complete order embeddings associated with the inductive system (\ref{eq:ctp1}).
        Consider the inductive system 
        \begin{equation} \label{eq:ctp2}
             C^*_u(\cl S_1) \mtp C^*_u(\cl T)
            \stackrel{\rho_1}{\longrightarrow}
             C^*_u(\cl S_2) \mtp C^*_u(\cl T)
            \stackrel{\rho_2}{\longrightarrow}
             C^*_u(\cl S_3) \mtp C^*_u(\cl T)
            \stackrel{\rho_3}{\longrightarrow}
            \cdots
        \end{equation}
        in $\catos$,
        where $\rho_k = \eta_k \otimes \id_{C^*_u(\cl T)}$, $k \in \bb{N}$. 
By assumption, the map 
$\phi_k\otimes\id_{\cl T} :  \cl S_k \ctp \cl T\to \cl S_{k+1} \ctp \cl T$ is a complete order isomorphic embedding, $k \in \bb{N}$. 
By Lemmas \ref{lem:op system coi then univer calg *iso} and \ref{l_univct},     
$\rho_k$ is a complete order embedding, $k \in \bb{N}$.
Let
        \[
            \rho_{k,\infty}:C^*_u(\cl S_k) \mtp C^*_u(\cl T) \rightarrow \limos (C^*_u(\cl S_k) \mtp C^*_u(\cl T)), \ \ \ k\in \bb{N},
        \]
        be the unital completely order isomorphic embeddings associated with the inductive system (\ref{eq:ctp2}), and 
        let
            $\widetilde{\alpha}: (\limos \cl S_k) \odot \cl T \rightarrow \limos(\cl S_k \ctp \cl T)$
        be the linear bijection from Proposition~\ref{prop:TensorProdCommute-LinearSpaces}. 
        Note that
        \begin{equation}\label{eq:ctp3}
            \widetilde{\alpha} \circ (\phi_{k,\infty} \otimes {\id}_{\cl T}) = \psi_{k,\infty}, \ \ \ k \in \bb{N},
        \end{equation}
        where $\{\psi_{k,\infty}\}_{k \in \bb{N}}$ are the unital completely order isomorphic embeddings associated to 
        $ \limos(\cl S_k \ctp \cl T)$ (with connecting mappings $\psi_k = \phi_k \otimes {\id}$, $k\in \bb{N}$). 
        Let
        \[
            \widetilde{\beta}: (\limos C_u^*(\cl S_k)) \mtp C_u^*(\cl T) \rightarrow \limos (C_u^*(\cl S_k) \mtp C_u^*(\cl T))
        \]
        be the unital complete order isomorphism such that
        \begin{equation}\label{eq:ctp5}
            \widetilde{\beta} \circ \big(\eta_{k,\infty} \otimes {\id}_{C_u^*(\cl T)}\big) = \rho_{k,\infty}, \ \ \ k \in \bb{N},
        \end{equation}
        whose existence is guaranteed by Theorem~\ref{thm:MaxTensorCommutesWithIndLimit}.

Consider the commutative diagram
        \begin{equation}\label{eq: ctp eq 4}
            \begin{CD}
                \cl S_1 \ctp \cl T @>{\psi_1}>>
                \cl S_2 \ctp \cl T @>{\psi_2}>>
                \cdots \\
                @V{\iota_1 \otimes \iota_{\cl T}}VV
                @V{\iota_2 \otimes \iota_{\cl T}}VV  \\
                C^*_u(\cl S_1) \mtp C^*_u(\cl T) @>{\rho_1}>>
                C^*_u(\cl S_2) \mtp C^*_u(\cl T) @>{\rho_2}>>
                \cdots,
            \end{CD}
        \end{equation}
and note that all the maps appearing in it are unital complete order embeddings.
        By Remark~\ref{rem:all maps coi}, there exists a unique unital complete order embedding
        \[
            \iota: \limos (\cl S_k \ctp \cl T) \rightarrow \limos \big(C^*_u(\cl S_k) \mtp C_u^*(\cl T)\big)
        \]
        such that
        \begin{equation}\label{eq:ctp4}
            \iota \circ \psi_{k,\infty} =  \rho_{k,\infty} \circ (\iota_k \otimes \iota_{\cl T}), \ \ \ k \in \bb{N}.
        \end{equation}

        By Lemma~\ref{lem:injectivity of mtp with closure} 
        and Theorem \ref{p_opsysst}, the canonical map
        \[
     \gamma : \left(\limos C_u^*(\cl S_k)\right) \mtp C_u^*(\cl T) \rightarrow \left(\limcalg C_u^*(\cl S_k)\right) \mtp C_u^*(\cl T)
        \]
        is a completely order isomorphic embedding.
        By Theorem~\ref{thm:Cu(S)}, there exists a unital 
        *-isomorphism
            $\mu : \limcalg C_u^*(\cl S_k) \rightarrow C_u^*(\limos \cl S_k)$
        such that
        \begin{equation}\label{eq:ctp6}
            \mu \circ \eta_{k,\infty} \circ \iota_k = \iota_{\los{\cl S}} \circ \phi_{k, \infty}
        \end{equation}
        for all $k \in \bb{N}$, where $\iota_{\los{\cl S}} : \limos \cl S_k \to C_u^*(\limos \cl S_k)$ 
        is the canonical embedding. We have that 
        \[
            \mu \otimes {\id}_{C_u^*(\cl T)} : \left(\limcalg C_u^*(\cl S_k)\right) \mtp C_u^*(\cl T) \rightarrow C_u^*(\limos \cl S_k) \mtp C_u^*(\cl T)
        \]
        is a unital *-isomorphism.
        By the definition of the commuting tensor product,
        \[
            \iota_{\los{\cl S}} \otimes \iota_{\cl T}: (\limos \cl S_k) \ctp \cl T \rightarrow C_u^*(\limos \cl S_k) \mtp C_u^*(\cl T)
        \]
        is a unital complete order isomorphism onto its image.

        We will show that
        \begin{equation}\label{eq_gammmma}
            (\iota_{\los{\cl S}} \otimes \iota_{\cl T}) \circ \widetilde{\alpha}^{-1} 
            = \big(\mu \otimes {\id}_{C_u^*(\cl T)}\big) 
            \circ  \gamma
            \circ \widetilde{\beta}^{-1} \circ \iota;
        \end{equation}
since $\big(\mu \otimes {\id}_{C_u^*(\cl T)}\big) \circ \gamma \circ \widetilde{\beta}^{-1} \circ \iota$  and 
$\iota_{\los{\cl S}} \otimes \iota_{\cl T}$ are  complete order embeddings, it will follow 
from Lemma~\ref{rem:ucio conjgation} that 
$\widetilde{\alpha}$ is a complete order embedding.
By (\ref{eq:ctp3}), (\ref{eq:ctp5}), (\ref{eq:ctp4}) and (\ref{eq:ctp6}), for every $k\in \bb{N}$, we have
        \[
            \begin{split}
                \big(\mu &\otimes {\id}_{C_u^*(\cl T)}\big) \circ \gamma \circ \widetilde{\beta}^{-1} \circ \iota \circ \psi_{k,\infty}\\
                &= \big(\mu \otimes {\id}_{C_u^*(\cl T)}\big) \circ \widetilde{\beta}^{-1} \circ \iota \circ \psi_{k,\infty}\\
                &=  \big(\mu \otimes {\id}_{C_u^*(\cl T)}\big) \circ \widetilde{\beta}^{-1} \circ \rho_{k,\infty} \circ (\iota_k \otimes \iota_{\cl T})\\
                &= \big(\mu \otimes {\id}_{C_u^*(\cl T)}\big) \circ  (\eta_{k,\infty} \otimes {\id}_{C^*_u(\cl T)})  \circ (\iota_k \otimes \iota_{\cl T})\\
                &= \Big( \mu \circ \eta_{k,\infty} \circ \iota_k \Big)  \otimes \big({\id}_{C_u^*(\cl T)} \circ \iota_{\cl T}\big) 
                = (\iota_{\los{\cl S}} \circ \phi_{k,\infty}) \otimes \iota_{\cl T} \\
                &= (\iota_{\los{\cl S}} \otimes \iota_{\cl T}) \circ \big(\phi_{k,\infty} \otimes {\id}_{\cl T}\big)
                = (\iota_{\los{\cl S}} \otimes \iota_{\cl T}) \circ \widetilde{\alpha}^{-1} \circ \psi_{k,\infty}.\\
            \end{split}
        \]
        This establishes (\ref{eq_gammmma}), and the proof is complete. 
        \end{proof}

Recall \cite{KavrukPaulsenTodorovTomforde} that an operator system $\cl S$ is 
said to possess the \emph{double commutant expectation property} 
if, for every complete order embedding $\cl S\subseteq \cl B(\cl H)$ (where $H$ is a Hilbert space), 
there exists a completely positive map from $\cl B(\cl H)$ into the double commutant $\cl S''$ of $\cl S$
that fixes $\cl S$ element-wise.

\begin{corollary}
        Let 
        $\cl S_1\stackrel{\phi_1}{\longrightarrow} \cl S_2 \stackrel{\phi_2}{\longrightarrow} \cl S_3 \stackrel{\phi_3}{\longrightarrow} \cl S_4 \stackrel{\phi_4}{\longrightarrow} \cdots$
        be an inductive system in $\catos$ such that each $\phi_k$ is a completely order isomorphic embedding, and let $\cl T$ be an operator system.
Assume that $\cl S_k$ satisfies the double commutant expectation property for each $k\in \bb{N}$.
        Then $\limos (\cl S_k \ctp \cl T)$ is unitally completely order isomorphic to $(\limos \cl S_k) \ctp \cl T$.
\end{corollary}
\begin{proof}
Since $\cl S_k$ satisfies the double commutant expectation property, 
\cite[Theorems 7.1 and 7.3]{KavrukPaulsenTodorovTomforde} imply that 
the map 
$\phi_k\otimes\id_{\cl T} :  \cl S_k \ctp \cl T\to \cl S_{k+1} \ctp \cl T$ is a complete order embedding, $k \in \bb{N}$. 
The claim now follows from Theorem \ref{thm:ComTensorCommutesWithIndLimit}.
\end{proof}


\section{Inductive limits of operator C*-systems} \label{chapt: operator systems additional structure}

    In this section, 
    we adapt our construction of the inductive limit of operator systems to 
    the category of operator C*-systems.
      We recall some notions and results that will be required shortly.
    Let $(\cl S, \{C_n\}_{n \in \mathbb{N}}, e)$ be a complete 
    operator system and $\cl A$ be 
    a unital C*-algebra such that $\cl S$ is an $\cl A$-bimodule. Let us denote the bimodule action by $\cdot$ so that $(a_1 a_2) \cdot s = a_1 \cdot (a_2 \cdot s)$ whenever $s \in \cl S$ and $a_1, a_1 \in \cl A$. 
We assume that $a \cdot e = e \cdot a$, $a \in \cl A$, and equip $M_n(\cl S)$ with a bimodule action of $M_n(\cl A)$ by letting 
$(a_{i,j})\cdot (s_{i,j}) = \left(\sum_{k=1}^n a_{i,k}\cdot s_{k,j}\right)$ and 
$(s_{i,j})\cdot (a_{i,j}) = \left(\sum_{k=1}^n s_{i,k}\cdot a_{k,j}\right)$. 
If
    \[
       A^* \cdot X\cdot A \in C_n \ \mbox{ whenever } \ X\in C_n,
    \]
we say that 
$\cl S$ is an {\it operator} $\cl A${\it -system} or that the pair $(\cl S, \cl A)$ is an {\it operator C*-system}. 
Let $(\cl S, \cl A)$ and $(\cl T, \cl B)$ be operator C*-systems.  A pair $(\phi, \pi)$ will be 
called an {\it operator C*-system homomorphism} 
    if $\phi: \cl S \rightarrow \cl T$ is a unital completely positive map, $\pi: \cl A \rightarrow \cl B$ is a unital *-homomorphism and $\phi(a_1 \cdot s \cdot a_2) = \pi(a_1) \cdot \phi(s) \cdot \pi(a_2)$ for all $a_1, a_2 \in \cl A$ and $s \in \cl S$. We write $(\phi, \pi) : (\cl S, \cl A) \rightarrow (\cl T, \cl B)$. We call the operator C*-system homomorphism $(\phi, \pi)$ an {\it operator C*-system monomorphism} if $\phi$ is completely isometric. 
    If $(\phi, \pi) : (\cl S, \cl A) \rightarrow (\cl T, \cl B)$ and $(\psi, \rho) : (\cl T, \cl B) \rightarrow (\cl R, \cl C)$ are operator 
    C*-system homomorphisms, we write $(\phi, \pi) \circ (\psi, \rho)$ for the pair $(\phi \circ \psi, \pi \circ \rho)$;
    it is straightforward to see that the latter is an operator C*-system homomorphism. 
    The following theorem is contained in \cite[Chapter~15]{PaulsenBook}.

    \begin{theorem}\label{thm:characterisation of operator A modules}
        Let $(\cl S, \cl A)$ be an operator C*-system. Then there exists a Hilbert space~$\hilb$ and 
        an operator C*-system monomorphism $(\Phi, \Pi): (\cl S, \cl A) \rightarrow (\BH, \BH)$ 
        such that the order unit of $\cl S$ is mapped to the identity operator.
    \end{theorem}

We denote by $\catoas$ the category whose objects are operator 
C*-systems and whose morphisms are operator C*-system homomorphisms.
            
Before considering inductive systems in $\catoas$, we make some observations which we shall refer to later in the section.

        \begin{lemma}\label{cauchy sequence bounded}
            Let 
            $\cl S_1\stackrel{\phi_1}{\longrightarrow} \cl S_2 \stackrel{\phi_2}{\longrightarrow} \cl S_3 \stackrel{\phi_3}{\longrightarrow} \cl S_4 \stackrel{\phi_4}{\longrightarrow} \cdots$ be an inductive system in~$\catos$. 
            If $s_{k_n}\in \cl S_{k_n}$ and 
            $(\phi_{k_n,\infty}(s_{k_n}))_{n \in \bb{N}}$ is a Cauchy sequence in $\limos \cl S_k$ 
            then there exists a sequence $(m_n)_{n \in \bb{N}} \subseteq \bb{N}$ such that $(\phi_{k_n, m_n}(s_{k_n}))_{n \in \bb{N}}$ is a bounded sequence.
                    \end{lemma}

        \begin{proof}
          Clearly, 
       there exists $M \in \bb{N}$ such that $\norm{\phi_{k_n,\infty}(s_{k_n})}_{\limos \cl S_k} < M$, $n \in \bb{N}$. 
       By Proposition~\ref{prop:norm for inductive limit operator system}, for each $n \in \bb{N}$, there exists $m_n \in \bb{N}$ such that $\norm{\phi_{k_n, m_n}(s_{k_n})}_{\cl S_{m_n}} < M$. 
        \end{proof}

        We fix throughout the section an inductive system
        \begin{equation}\label{eq:ind limit oas 1}
    (\cl S_1, \cl A_1)\stackrel{(\phi_1, \pi_1)}{\longrightarrow} (\cl S_2, \cl A_2) \stackrel{(\phi_2, \pi_2)}{\longrightarrow} (\cl S_3, \cl A_3) \stackrel{(\phi_3, \pi_3)}{\longrightarrow} (\cl S_4, \cl A_4) \stackrel{(\phi_4, \pi_4)}{\longrightarrow} \cdots
    \end{equation}
    in $\catoas$. Thus,
    $\cl S_1\stackrel{\phi_1}{\longrightarrow} \cl S_2 \stackrel{\phi_2}{\longrightarrow} \cl S_3 \stackrel{\phi_3}{\longrightarrow} \cl S_4 \stackrel{\phi_4}{\longrightarrow} \cdots$
    is an inductive system in $\catos$,
    $\cl A_1\stackrel{\pi_1}{\longrightarrow} \cl A_2 \stackrel{\pi_2}{\longrightarrow} \cl A_3 \stackrel{\pi_3}{\longrightarrow} \cl A_4 \stackrel{\pi_4}{\longrightarrow} \cdots$
    is an inductive system in $\catcalg$, 
    $\cl S_k$ is an operator $\cl A_k$-system and $(\phi_k, \pi_k)$ is an operator C*-system homomorphism, $k \in \bb{N}$. 
    We set 
    $$\cl S_{\infty} = \limos \cl S_k, \ \cl A_{\infty} = \limos \cl A_k, \ 
    \widehat{\cl A}_{\infty} = \limcalg \cl A_k, $$
    and $\widehat{\cl S}_{\infty}$ to be the completion of $\cl S_{\infty}$. 
    By Theorem \ref{p_opsysst}, $\widehat{\cl A}_{\infty}$ is the completion of $\cl A_{\infty}$. 
    
    We proceed with the construction of the inductive limit of $(\ref{eq:ind limit oas 1}).$
    Let $a \in \widehat{\cl A}_{\infty}$ and $s \in \widehat{\cl S}_{\infty}$. Then $a = \lim_{n \rightarrow \infty} \pi_{k_n, \infty}(a_{k_n})$ and $s = \lim_{n \rightarrow \infty} \phi_{l_n, \infty}(s_{l_n})$ for some $a_{k_n} \in \cl A_{k_n}$ and $s_{l_n} \in \cl S_{l_n}$. 
    Letting 
   $m_n = \max\{k_n, l_n\}$, $a_{m_n} = \pi_{k_n, m_n}(a_{k_n})$ and $s_{m_n} = \phi_{l_n, m_n}(s_{l_n})$, we have 
            \[
                a = \lim_{n \rightarrow \infty} \pi_{m_n, \infty}(a_{m_n}) ~\text{~and~}~ s = \lim_{n \rightarrow \infty} \phi_{m_n, \infty}(s_{m_n}).
            \]
We let
\begin{equation}\label{eq_asa}
                a \cdot s = \lim_{n \rightarrow \infty} \phi_{m_n, \infty}(a_{m_n} \cdot s_{m_n})
                \mbox{ and } 
                s \cdot a = \lim_{n \rightarrow \infty} \phi_{m_n, \infty}(s_{m_n} \cdot a_{m_n}).
\end{equation}

        \begin{proposition}\label{prop:module action catoas}
The operations (\ref{eq_asa}) are well-defined and turn
$\widehat{\cl S}_{\infty}$ into an $\widehat{\cl A}_{\infty}$-bimodule.
\end{proposition}

            \begin{proof}
         It is easy to see that $\big(\phi_{m_n, \infty}(a_{m_n} \cdot s_{m_n})\big)_{n \in \bb{N}}$ is a Cauchy sequence.
Moreover, if  
$a = \lim_{n \rightarrow \infty} \pi_{m_n, \infty}(b_{m_n}) \in \widehat{\cl A}_{\infty}$
            and
$s = \lim_{n \rightarrow \infty} \phi_{m_n, \infty}(t_{m_n}) \in \widehat{\cl S}_{\infty}$, 
a straightforward calculation shows that 
$$\lim_{l \rightarrow \infty} \norm{\phi_{m_n, l} (a_{m_n} \cdot s_{m_n} - b_{m_n} \cdot t_{m_n})}_{\cl S_l}  = 0.$$
It follows that the left action in (\ref{eq_asa}) is well-defined; similarly, the right action is well-defined. 
The fact that these operations are module actions is straightforward. 
We check, for example,  the property $(a\cdot s)\cdot b = a\cdot (s\cdot b)$: 
writing 
$b = \lim_{n \rightarrow \infty} \pi_{m_n, \infty}(b_{m_n}) \in \limcalg \cl A_k$, we have 
            \[
            \begin{split}
                (a \cdot s) \cdot b
                &= \big(\lim_{n \rightarrow \infty}\phi_{m_n, \infty}(a_{m_n} \cdot s_{m_n})\big) \cdot b\\
                &= \lim_{n \rightarrow \infty} \phi_{m_n, \infty} ((a_{m_n} \cdot s_{m_n}) \cdot b_{m_n})
                = \lim_{n \rightarrow \infty} \phi_{m_n, \infty} (a_{m_n} \cdot (s_{m_n} \cdot b_{m_n})) \\
                &= a \cdot \big(\lim_{n \rightarrow \infty} \phi_{m_n, \infty}(s_{m_n} \cdot b_{m_n})\big)
                = a \cdot (s \cdot b).\\
                \end{split}
            \]
\end{proof}

\begin{remark}\label{r_noncads}
{\rm Note that, if $k\in \bb{N}$, $a,b\in \cl A_k$ and $s\in \cl S_k$ then 
$$\pi_{k,\infty}(a)\cdot \phi_{k,\infty}(s)\cdot \pi_{k,\infty}(b) = \phi_{k,\infty}(a\cdot s\cdot b).$$}
\end{remark}

        \begin{lemma}\label{lem:inductive lim oas positive cones}
If $S \in M_n(\cl S_{\infty})^+$ and $A \in M_n(\cl A_{\infty})$ then $A^*\cdot S\cdot A \in M_n(\cl S_{\infty})^+$.
        \end{lemma}

        \begin{proof}
            Write $S = \phi_{k,\infty}^{(n)}(S_k) \in M_n(\cl S_{\infty})^+$
            and
            $A = \pi_{k,\infty}^{(n)}(A_k) \in M_n(\cl A_{\infty})$, where $S_k\in M_n(\cl S_k)$ and $A_k\in M_n(\cl A_k)$ for some $k$.    
            Then $A_k^*\cdot S_k \cdot A_k\in M_n(\cl S_k)^+$. Since 
            the map $\phi_{k,\infty}$ is completely positive, Remark \ref{r_noncads} shows that 
            \begin{eqnarray*}
                A^*\cdot S\cdot A 
                & = & 
                \pi_{k,\infty}^{(n)}(A_k^*) \cdot \phi_{k,\infty}^{(n)}(S_k) \cdot \pi_{k, \infty}^{(n)}(A_k)\\
                & = & 
                \phi_{k,\infty}^{(n)}(A_k^* \cdot S_k\cdot  A_k)
               \in M_n(\cl S_{\infty})^+.
\end{eqnarray*}
        \end{proof}

        \begin{proposition} \label{prop:limSkIsAnOperatorLimAkSystem}
            The space $\widehat{\cl S}_{\infty}$ is an operator $\widehat{\cl A}_{\infty}$-system 
            and $(\phi_{k,\infty}, \pi_{k, \infty})$ is an operator C*-system homomorphism
            from $(\cl S_k, \cl A_k)$ into $(\widehat{\cl S}_{\infty}, \widehat{\cl A}_{\infty})$ 
            such that $(\phi_{k+1,\infty}, \pi_{k+1,\infty}) \circ (\phi_k, \pi_k) = (\phi_{k,\infty}, \pi_{k, \infty})$,
            $k \in \bb{N}$.
                    \end{proposition}

        \begin{proof}
            By Proposition~\ref{prop:module action catoas}, $\widehat{\cl S}_{\infty}$ is a $\widehat{\cl A}_{\infty}$-bimodule;
            it is clear that $\widehat{\cl S}_{\infty}$ is a complete operator system.            
Suppose $S \in M_n(\widehat{\cl S}_{\infty})^+$ and $A \in M_n(\widehat{\cl A}_{\infty})$ so that $S = \lim_{p \rightarrow \infty} S_p$ 
and $A= \lim_{p \rightarrow \infty} A_p$
            where $S_p \in M_n(\cl S_{\infty})^+$ and $A_p \in M_n(\cl A_{\infty})$. 
            Then $A^*\cdot S \cdot A = \lim_{p \rightarrow \infty} A_p^* \cdot S_p \cdot  A_p$ 
            and, by Lemma~\ref{lem:inductive lim oas positive cones}, 
            $A_p^* \cdot S_p \cdot A_p \in M_n(\cl S_{\infty})^+$
            for all $p \in \mathbb{N}$. 
            Since the cone $M_n(\losc{\cl S})^+$ is closed, $A^*\cdot S \cdot A \in M_n(\losc{\cl S})^+$.
        \end{proof}

        \begin{theorem}\label{thm:universal property_operator_Asystem2}
The triple $(\widehat{\cl S}_{\infty}, \widehat{\cl A}_{\infty}, \{\phi_{k,\infty}, \pi_{k,\infty}\}_{k\in \bb{N}})$ 
is an inductive limit of the inductive system
            $$(\cl S_1, \cl A_1)\stackrel{(\phi_1, \pi_1)}{\longrightarrow} (\cl S_2, \cl A_2) \stackrel{(\phi_2, \pi_2)}{\longrightarrow} (\cl S_3, \cl A_3) \stackrel{(\phi_3, \pi_3)}{\longrightarrow} (\cl S_4, \cl A_4) \stackrel{(\phi_4, \pi_4)}{\longrightarrow} \cdots$$
            in $\catoas$.
        \end{theorem}
        
        \begin{proof}
            Suppose $\big((\cl T, \cl B), \{(\psi_k, \rho_k)\}_{k \in \mathbb{N}}\big)$ is a pair consisting of a complete operator C*-system and a family of operator C*-system homomorphisms
           $(\psi_k, \rho_k): (\cl S_k, \cl A_k) \rightarrow (\cl T, \cl B)$ such that $(\psi_{k+1}, \rho_{k+1}) \circ (\phi_{k}, \pi_k) = (\psi_k, \rho_k)$ for all $k \in \bb{N}$. By Theorem~\ref{thm:universal property_operator_system2}, 
           there exists a unique unital completely positive map $\psi : \cl S_{\infty} \rightarrow \cl T$ such that
            $\psi \circ \phi_{k, \infty} = \psi_k$. 
            Let $\widehat{\psi} : \widehat{\cl S}_{\infty} \to \cl T$ 
            be the continuous extension of~$\psi$. 
            Lemma \ref{lem:norm closure of op system} easily implies that 
            $\widehat{\psi}$ is completely positive. 
            By Section~\ref{subs_iil}, there exists a unique unital 
            *-homomorphism $\widehat{\rho}: \limcalg \cl A_k \rightarrow \cl B$ such that $\widehat{\rho} \circ \pi_{k,\infty} = \rho_k$.
            Suppose that
                $a = \lim_{n \rightarrow \infty} \pi_{m_n, \infty}(a_{m_n}), b = \lim_{n \rightarrow \infty} \pi_{m_n, \infty}(b_{m_n}) \in \widehat{\cl A}_{\infty}$ 
                and
                $s = \lim_{n \rightarrow \infty} \phi_{m_n, \infty}(s_{m_n}) \in \widehat{\cl S}_{\infty};$
            then
            \[
                \begin{split}
                    & \widehat{\psi}(a \cdot s \cdot b)\\
                    &= \widehat{\psi} \big(\lim_{n \rightarrow \infty} \phi_{m_n, \infty} (a_{m_n} \cdot s_{m_n} \cdot b_{m_n})\big)\\
                    &= \lim_{n \rightarrow \infty}  \psi \circ \phi_{m_n, \infty} (a_{m_n} \cdot s_{m_n} \cdot b_{m_n}) \\
                    &=\lim_{n \rightarrow \infty}  \psi_{m_n} (a_{m_n} \cdot s_{m_n} \cdot b_{m_n})\\
                    &= \lim_{n \rightarrow \infty}  \rho_{m_n}(a_{m_n}) \cdot \psi_{m_n}(s_{m_n}) \cdot \rho_{m_n}(b_{m_n})\\
                    &= \lim_{n \rightarrow \infty} \rho_{m_n}(a_{m_n}) \cdot \lim_{n \rightarrow \infty} \psi_{m_n}(s_{m_n}) \cdot \lim_{n \rightarrow \infty} \rho_{m_n}(a_{m_n})\\
                    &= \lim_{n \rightarrow \infty} \rho \circ \pi_{m_n, \infty}(a_{m_n}) \cdot \lim_{n \rightarrow \infty} \psi \circ \phi_{m_n, \infty}(s_{m_n})\cdot \lim_{n \rightarrow \infty} \rho \circ \pi_{m_n, \infty}(b_{m_n})\\
                    &= \widehat{\rho}(a) \cdot \widehat{\psi}(s)\cdot \widehat{\rho}(b).\\
                \end{split}
            \]
            This completes the proof.
        \end{proof}

We denote the inductive limit whose existence is established in 
Theorem \ref{thm:universal property_operator_Asystem2}
by $\limoas (\cl S_k, \cl A_k)$ or $\limoas \cl S_k$, when the context is clear.

    \begin{remark}
    {\rm
      Let $\big(\{\cl S_k, \cl A_k\}_{k \in \bb{N}}, \{\phi_k, \pi_k\}_{k \in \bb{N}}\big)$ and $\big(\{\cl T_k, \cl B_k\}_{k \in \bb{N}}, \{\psi_k, \rho_k\}_{k \in \bb{N}}\big)$ be inductive systems in $\catoas$ and let $\{(\theta_k, \varphi_k)\}_{k \in \mathbb{N}}$ be a sequences of operator C*-system homomorphisms such that the following diagrams commute:
      \[
    \begin{CD}
        \cl S_1 @>{\phi_1}>> \cl S_2 @>{\phi_2}>> \cl S_3 @>{\phi_3}>> \cl S_4 @>{\phi_4}>> \cdots \\
        @V{\theta_1}VV @V{\theta_2}VV @V{\theta_3}VV @V{\theta_4}VV \\
        \cl T_1 @>{\psi_1}>> \cl T_2 @>{\psi_2}>> \cl T_3 @>{\psi_3}>> \cl T_4 @>{\psi_4}>> \cdots
    \end{CD}
    \]
    and
    \[
    \begin{CD}
      \cl A_1 @>{\pi_1}>> \cl A_2 @>{\pi_2}>> \cl A_3 @>{\pi_3}>> \cl A_4 @>{\pi_4}>> \cdots \\
      @V{\varphi_1}VV @V{\varphi_2}VV @V{\varphi_3}VV @V{\varphi_4}VV \\
      \cl B_1 @>{\rho_1}>> \cl B_2 @>{\rho_2}>> \cl B_3 @>{\rho_3}>> \cl B_4 @>{\rho_4}>> \cdots.
    \end{CD}
    \]
    It follows from Theorem~\ref{thm:universal property infinitely many maps} and Theorem~\ref{thm:universal property_operator_Asystem2} that there exists a unique operator C*-system homomorphism 
    $$(\widehat{\theta}, \widehat{\varphi}): (\limoas \cl S_k, \limcalg \cl A_k) \rightarrow (\limoas \cl T_k, \limcalg \cl B_k)$$ such that
    $(\widehat{\theta}, \widehat{\varphi}) \circ (\phi_{k,\infty}, \pi_{k,\infty}) = (\psi_{k,\infty}, \rho_{k,\infty}) \circ (\theta_k, \varphi_k)$
    for all $k \in \bb{N}$.
    }
    \end{remark}

    \begin{remark}\label{rem:op A systems are isomorphic for commmuting diagram}
    {\rm
      Suppose that 
      $\big(\{\cl S_k, \cl A_k\}_{k \in \bb{N}}, \{\phi_k, \pi_k\}_{k \in \bb{N}}\big)$ and 
      $\big(\{\cl T_k, \cl B_k\}_{k \in \bb{N}},$ $\{\psi_k, \rho_k\}_{k \in \bb{N}}\big)$ are inductive systems in $\catoas$, and let $\{(\theta_{m_k}, \varphi_{m_k})\}_{k \in \mathbb{N}}$ and $\{(\mu_{n_k}, \nu_{n_k})\}_{k \in \mathbb{N}}$ be sequences of 
      operator C*-system monomorphisms such that the diagrams
      \vspace{-2mm}
      \[
        \begin{tikzcd}
            \cl S_1 \arrow{r}{\phi_{1, m_1}} \arrow{d}{\theta_{1}} & \cl S_{m_1} \arrow{d}{\theta_{m_1}}\arrow{r}{\phi_{m_1, m_2}}&  \cl S_{m_2} \arrow{d}{\theta_{m_2}} \arrow{r}{\phi_{m_2 m_3}}& \cdots \\
            \cl T_{n_1} \arrow[swap]{r}{\psi_{n_1, n_2}} \arrow[swap]{ur}{\mu_{n_1}}      & \cl T_{n_2} \arrow[swap]{r}{\psi_{n_2, n_3}} \arrow[swap]{ur}{\mu_{n_2}}          & \cl T_{n_3} \arrow[swap]{ur}{\mu_{n_3}}  \arrow[swap]{r}{\psi_{n_3, n_4}}          & \cdots \\
        \end{tikzcd}
          \vspace{-10mm}
      \]
      and
      \vspace{-2mm}
    \[
        \begin{tikzcd}
            \cl A_1 \arrow{r}{\pi_{1, m_1}} \arrow{d}{\varphi_{1}} & \cl A_{m_1} \arrow{d}{\varphi_{m_1}}\arrow{r}{\pi_{m_1, m_2}}&  \cl A_{m_2} \arrow{d}{\varphi_{m_2}} \arrow{r}{\pi_{m_2, m_3}}& \cdots \\
            \cl B_{n_1} \arrow[swap]{r}{\rho_{n_1, n_2}} \arrow[swap]{ur}{\nu_{n_1}}      & \cl B_{n_2} \arrow[swap]{r}{\rho_{n_2, n_3}} \arrow[swap]{ur}{\nu_{n_2}}         & \cl B_{n_3} \arrow[swap]{ur}{\nu_{n_3}} \arrow[swap]{r}{\rho_{n_3, n_4}}          & \cdots \\
        \end{tikzcd}
          \vspace{-10mm}
    \]
commute.
    By Theorem~\ref{thm:universal property_operator_Asystem2}, $\limoas(\cl S_k, \cl A_k)$ and $\limoas(\cl T_k, \cl B_k)$ are isomorphic. 
    In particular, $\limoas \cl S_k$ is unitally completely order isomorphic to $\limoas \cl T_k$ and $\limcalg \cl A_k$ is unitally *-isomorphic to $\limcalg \cl B_k$.
    }
    \end{remark}


\section{Inductive limits of graph operator systems} \label{chpt:graph operator systems}

In this section, we examine inductive limits of graph operator systems, viewing them as the operator systems 
of topological graphs {\it via} the theory of topological equivalence relations \cite{Power}. We identify the C*-envelope of 
such an operator system, and prove an isomorphism theorem; these can be viewed as a 
topological version of recent results from \cite{PaulsenOrtiz}. 
We also establish a version of the Glimm Theorem for this class of operator systems. 
As our results rely crucially on \cite{Power} (and thus on \cite{Power2}, \cite{Power3}, \cite{Power4} and \cite{DavidsonPower})), 
for the convenience of the reader, we often provide the background and details.



    A {\it UHF algebra} \cite{Glimm}
    (or, otherwise, \emph{uniformly hyper-finite} C*-algebra) is a C*-algebra 
    which is (*-isomorphic to) the inductive limit of an inductive system
 \begin{equation}\label{eq_UHFal}
        M_{n_1}\stackrel{\pi_1}{\longrightarrow} M_{n_2}\stackrel{\pi_2}{\longrightarrow} M_{n_3}\stackrel{\pi_3}{\longrightarrow} M_{n_4}\stackrel{\pi_4}{\longrightarrow} \cdots,
 \end{equation}
where $\pi_k$ is a unital *-homomorphism, $k \in \mathbb{N}$. 
UHF algebras and their classification appear extensively in the literature, see for example \cite{DavidsonBook}, \cite{MurphyBook} or \cite{RordamLarsenLausten}. For each $k \in \mathbb{N}$, let $e_{i,j}^k$ denote the matrix in~$M_{n_k}$ with $1$ at the 
$(i,j)^{\text{th}}$ entry and $0$ elsewhere and let $l_k = \frac{n_{k+1}}{n_k}$. 
We have that
    \[
        \pi_k(e_{i,j}^{k}) = \sum_{r=0}^{l_k-1} e^{k+1}_{rn_k+i, rn_k+j}.
    \]
We call $e_{i,j}^{k}$ the \emph{canonical matrix units}. 

    Let $\cl A$ be a C*-algebra. A C*-subalgebra of~$\cl A$ is called a {\it maximal abelian self-adjoint  algebra} 
    ({\it masa}, for short)
    if it is abelian and not properly contained in another abelian C*-subalgebra of~$\cl A$. 
    Let
    \[
        \cl D_1 \stackrel{\pi_1}{\longrightarrow} \cl D_2\stackrel{\pi_2}{\longrightarrow} \cl D_3\stackrel{\pi_3}{\longrightarrow} \cl D_4\stackrel{\pi_4}{\longrightarrow} \cdots
    \]
    be the inductive system in $\catcalg$ induced by (\ref{eq_UHFal}),
    where $\cl D_k$ is the subalgebra of diagonal matrices in $M_{n_k}$ for each $k \in \bb{N}$. A proof of the following result may be found in \cite[Proposition~4.1]{Power}.

    \begin{proposition} \label{prop:Ind limit masa}
        The C*-algebra $\limcalg \cl D_k$ is a masa in $\limcalg M_{n_k}$.\index{masa}
    \end{proposition}
    
    Denote by $\Delta(\cl C)$ the Gelfand spectrum of an abelian C*-algebra $\cl C$. 
    We call $\limcalg \cl D_k$ the {\it canonical masa} in the UHF algebra $\limcalg M_{n_k}$. 
    Since $\limcalg \cl D_k$ is an abelian C*-algebra, we have that $\limcalg \cl D_k$ is 
    *-isomorphic to $C(X_{\infty})$ where $X_{\infty} = \Delta(\limcalg \cl D_k)$. 
    For the following remark, which is a special case of 
    Remark \ref{ex:InductiveLimitC(X)}, 
    let $X_k = \Delta(\cl D_k)$.

\begin{remark} \label{rem:functionals on diagonal matrix algebra}
        The space $X_{\infty}$ is homeomorphic to $\limtop X_k$. 
\end{remark}

    The following theorem, whose proof may be found in \cite{Glimm} (see \cite{MurphyBook} for an alternative proof), 
    characterises UHF algebras.

    \begin{theorem}[Glimm]\label{thm:glimms}
The UHF algebras $\limcalg M_{n_k}$ and $\limcalg M_{m_k}$
are *-isomorphic if and only if for all $w \in \bb{N}$ there exists $x \in \bb{N}$ such that $n_w | m_x$, and for all $y \in \bb{N}$ there exists $z \in \bb{N}$ such that $m_y | n_z$.
    \end{theorem}

    Let $X$ be a topological space. We define a {\it graph} to be a pair $G=(X,E)$ of sets such that $E \subseteq X \times X$
    is a closed subset which is 
    symmetric (that is, $(x,y) \in E$ if and only if $(y,x) \in E$) 
    and anti-reflexive (that is, $(x,x)\notin E$ for all $x \in X$). 
    We call the elements of $X$ the {\it vertices} of $G$ and say that two vertices $x,y \in X$ are 
    {\it adjacent} if $(x,y)\in E$. Given $G$, we set 
    $\widetilde{G} = (X, \widetilde{E})$ where $\widetilde{E} = E \cup \{(x,x): x \in X\}$
    is the \emph{extended edge set} of $G$. 
    Two graphs $G = (X,E)$ and $G' = (X',E')$ are called {\it isomorphic} if there exists a homeomorphism 
    $\varphi: X \rightarrow X'$ such that $(x,y) \in E$ if and only if $(\varphi(x),\varphi(y))\in E'$.

    Let $G$ be a graph on $n$ vertices so that $X = \{1, \ldots, n\}$.
    Denote by $e_{i,j}$ the $n \times n$ matrix with $1$ in its $(i,j)$th-entry and $0$ elsewhere. 
    We define the {\it operator system} $\cl S_{G}$ of $G$ by letting 
    \[
        \cl S_{G} = \SPAN \lbrace e_{i,j} : (i,j) \in \widetilde{E}\rbrace.
    \]
    A {\it graph operator system} is an operator system of the form $\cl S_G$.
    
    Denote temporarily by $\cl D$ be the subalgebra of diagonal matrices in $M_n$. 
    Clearly, $(\cl S_G, \cl D)$ is an operator C*-system when we take the module operation 
    to be the usual matrix multiplication in $M_n$. 
    The following characterisation is well-known, see \cite{PaulsenPowerSmith}.

    \begin{proposition}
    Let $\cl S$ be an operator subsystem of $M_n$.
    Then $\cl S$ is a graph operator system if and only if $\cl D \cl S \cl D \subseteq \cl S$. 
    In this case the graph $G = (X,E)$ is defined by letting 
    $X = \{1, \ldots, n\}$ and $E = \{(i,j) : i\neq j \mbox{ and } e_{i,j} \in \cl S\}$.
    \end{proposition}

    The following two results about graph operator systems were proved in \cite[Theorem~3.2~and~Theorem~3.3]{PaulsenOrtiz}.

    \begin{theorem}[Paulsen--Ortiz] \label{thm:C*envelopefinitegraph}
        Let $G$ be a graph on $n$ vertices. Then the C*-subalgebra of $M_n$ generated by $\cl S_G$ is the C*-envelope of $\cl S_G$.
    \end{theorem}

    \begin{theorem}[Paulsen--Ortiz] \label{thm:PaulsenOrtiz}
        Let $G_1$ and $G_2$ be graphs on $n$ vertices. Then $G_1$ is isomorphic to $G_2$ if and only if $\cl S_{G_1}$ is unitally completely order isomorphic to $\cl S_{G_2}$.
    \end{theorem}


    \subsection{Operator C*-systems in UHF algebras}\label{sec:closed operator subsystems}

    We define a {\it concrete operator C*-system}
    to be a triple $(\cl D, \cl S, \cl A)$ where 
    $\cl D,\cl A \in \catcalg$, $\cl S \in \catos$, $(\cl S, \cl D) \in \catoas$, 
    $\cl D \subseteq \cl S \subseteq \cl A$ and $\cl D \cl S \cl D \subseteq \cl S$. 
    When the context is clear, we simplify the notation and call $\cl S$ a {\it concrete operator $\cl D$-system}, without mention of $\cl A$. 

Throughout this chapter, we fix an inductive system
$$        M_{n_1}\stackrel{\pi_1}{\longrightarrow} M_{n_2}\stackrel{\pi_2}{\longrightarrow} M_{n_3}\stackrel{\pi_3}{\longrightarrow} M_{n_4}\stackrel{\pi_4}{\longrightarrow} \cdots$$
in $\catcalg$.
Suppose that $G_k$ is a graph on $n_k$ vertices, such that $\pi_k(\cl S_{G_k})\subseteq \cl S_{G_{k+1}}$,  
and let $\phi_k = \pi_k|_{\cl S_{G_k}}$, $k\in \bb{N}$. We thus have inductive systems
    \[
        \cl S_{G_1} \stackrel{\phi_1}{\longrightarrow} \cl S_{G_2} \stackrel{\phi_2}{\longrightarrow} \cl S_{G_3} \stackrel{\phi_3}{\longrightarrow} \cl S_{G_4} \stackrel{\phi_4}{\longrightarrow} \cdots
    \]
    and
    \[
       \cl D_1 \stackrel{\pi_1}{\longrightarrow} \cl D_2 \stackrel{\pi_2}{\longrightarrow} \cl D_3 \stackrel{\pi_3}{\longrightarrow} \cl D_4 \stackrel{\pi_4}{\longrightarrow} \cdots;
    \]
    since $\cl S_{G_k}$ is an operator $\cl D_k$-system, 
    the latter inductive systems can 
be viewed as an inductive system in $\catoas$.
   Note that the inductive limit $\limoas \cl S_{G_k}$ is the completion of $\limos \cl S_{G_k}$ or, equivalently, 
   the closure of $\limos \cl S_{G_k}$ in $\limcalg \cl A_k$. 
   (Here, and in the sequel, write $\cl A_k = \pi_{k,\infty}(M_{n_k})$; note that $\cl A_k \cong M_{n_k}$.)
   We will see that every
 concrete operator $(\limcalg \cl D_k)$-system 
 (defined shortly) is the inductive limit of a sequence of graph operator systems,
and will associate to $\limoas \cl S_{G_k}$ a graph which is related to the sequence of graphs $(G_k)_{k \in \mathbb{N}}$.

We will use the following notation to denote the inductive limits:
    \[
        \begin{split}
            \los{\cl S} & = \limos \cl S_k, \\
            \losc{\cl S}& =\limoas \cl S_k,  \\
            \lcalg{\cl D} & = \limcalg \cl D_k \text{~and~}\\
             \lcalg {\cl A}& =\limcalg \cl A_k. \\
        \end{split}
    \]
    Observe that $(\lcalg{\cl D}, \losc{\cl S}, \lcalg {\cl A})$ is a concrete operator C*-system.
            Since each~$\pi_k$ is a unital injective *-homomorphism, by Remark~\ref{rem:universal prop Calg},
    $\pi_{k,\infty}$ is a unital injective *-homomorphism for all $k \in \bb{N}$; 
    we therefore sometimes simplify the notation and write $a_k$ in the place of $\pi_{k,\infty}(a_k)$. 

Recall \cite{Power} that a closed linear subspace $\cl S$ of $\lcalg{\cl A}$ is said to be {\it inductive} relative to $(\cl A_k)_{k \in \mathbb{N}}$ if
    \[
        \cl S = \overline{\bigcup_{k \in \mathbb{N}} \cl S \cap \cl A_k}.
    \]
We note the following fact which follows from \cite[Theorem~4.7]{Power}.

    \begin{proposition}\label{thm:closed D_bimodule is ind limit of graph operator systems}
        Let $\cl S \subseteq \lcalg{\cl A}$ be a concrete operator $\lcalg{\cl D}$-system and set 
        $\cl S_k = \cl S \cap \cl A_k$. Then $\cl S_k \subseteq \cl A_k$ is a concrete operator $\cl D_k$-system and $\cl S = \overline{\limos\cl S_k}$.
            \end{proposition}

The next result is an infinite dimensional analogue of Theorem~\ref{thm:C*envelopefinitegraph}.

     \begin{theorem} \label{thm:c*algebra generated by uhf algebra is c*envelope} 
        Let $\losc{\cl S} \subseteq \lcalg{\cl A}$ be a concrete operator $\lcalg{\cl D}$-system. 
        The C*-envelope of $\losc{\cl S}$ coincides with the 
        C*-subalgebra of~$\lcalg{\cl A}$ generated by $\losc{\cl S}$.
    \end{theorem}
    
    \begin{proof}
        Let $C^*(\losc{\cl S})$ denote the C*-subalgebra of $\lcalg{\cl A}$ generated by the operator system~$\losc{\cl S}$ and let $C^*(\cl S_k)$ denote the C*-subalgebra of $\cl A_k$ generated by~$\cl S_k$.
Since $\pi_k(\cl S_k)\subseteq \cl S_{k+1}$, we have that $\pi_k(C^*(\cl S_k))\subseteq C^*(\cl S_{k+1})$.
       
       Consider the following inductive system in $\catcalg$:
        \[
        C^*(\cl S_1) \stackrel{\pi_1}{\longrightarrow} C^*(\cl S_2) \stackrel{\pi_2}{\longrightarrow} C^*(\cl S_3) \stackrel{\pi_3}{\longrightarrow} C^*(\cl S_4) \stackrel{\pi_4}{\longrightarrow} \cdots.
        \]
Note that $\pi_{k,\infty}(C^*(\cl S_k)) \subseteq C^*(\los{\cl S})$. 
We denote again by $\pi_{k,\infty}$ its restriction to $C^*(\cl S_k)$; note that it is a 
*-homomorphism and $\pi_{k+1, \infty} \circ \pi_k = \pi_{k,\infty}$, $k \in \bb{N}$. 
We show that $C^*(\los{\cl S}) = C^*(\losc{\cl S})$, equipped with the family $\{\pi_{k,\infty}\}_{k\in \bb{N}}$,
satisfies the universal property of the inductive limit 
$\limcalg C^*(\cl S_k)$ and therefore they are *-isomorphic. 

Suppose $(\cl B, \{ \theta_k\}_{k \in \mathbb{N}})$ is a pair consisting of a 
C*-algebra and a family of unital *-homomorphisms $\theta_k : C^*(\cl S_k) \rightarrow \cl B$ such that $\theta_{k+1} \circ \pi_k = \theta_k$ for all $k \in \mathbb{N}$.
Note that, if $s_1, \ldots, s_n \in \los{\cl S}$ and  $a = s_1 \cdots s_n$ then, 
        writing $s_i = \pi_{k,\infty}({x_{k_i}})$ for some $x_{k_i} \in \cl S_{k}$, $i = 1,\dots,n$, we have that $a=\pi_{k,\infty}(x_1 \cdots x_n)$.
Suppose that 
$$\pi_{k,\infty}\left(\sum_{i=1}^p x_1^i \cdots x_{n_i}^i\right) = \pi_{l,\infty}\left(\sum_{j=1}^q y_1^j \cdots y_{m_j}^j\right),$$ 
for some $k,l\in \bb{N}$, $x_s^i \in \cl S_k$ and $y_t^j \in \cl S_l$. 
Then 
$$\lim_{d \rightarrow \infty}
\left\| \pi_{k,d}\left(\sum_{i=1}^p x_1^i \cdots x_{n_i}^i\right) - \pi_{l,d}\left(\sum_{j=1}^q y_1^j \cdots y_{m_j}^j\right) \right\|
= 0,$$ 
and letting $m= \max\{k,l\}$, we have that 
$$
\lim_{d \rightarrow \infty}
\left\| \pi_{m,d}\left(\pi_{k,m}\left(\sum_{i=1}^p x_1^i \cdots x_{n_i}^i\right) 
- \pi_{l,m}\left(\sum_{j=1}^q y_1^j \cdots y_{m_j}^j\right)\right)\right\| = 0.$$
Thus, 
$$\lim_{d \rightarrow \infty}
\left\| \theta_d\circ \pi_{m,d}\left(\pi_{k,m}\left(\sum_{i=1}^p x_1^i \cdots x_{n_i}^i\right) 
- \pi_{l,m}\left(\sum_{j=1}^q y_1^j \cdots y_{m_j}^j\right)\right)\right\| = 0.$$
It follows that
$$
 \theta_k \left(\sum_{i=1}^p x_1^i \cdots x_{n_i}^i\right)  = \theta_l \left(\sum_{j=1}^q y_1^j \cdots y_{m_j}^j\right).
$$

Let 
$$\cl U = {\rm span}\left\{\sum_{i=1}^p s_1^i \cdots s_{n_i}^i : p, n_i \in \bb{N}, s_m^i \in \cl S_{\infty}, k\in\bb{N}\right\}.$$ 
It follows from the previous paragraph that the map
       $\theta: \cl U \rightarrow \cl B$, given by 
\begin{equation}\label{eq_thetacsst}       
\theta \circ \pi_{k,\infty} = \theta_k,  \ \ \ k\in \bb{N}, 
\end{equation}
       is well-defined. It is clearly bounded, and we let
$\widetilde{\theta}: C^*(\cl S_{\infty}) \rightarrow \cl B$ 
be its continuous extension. Taking into account (\ref{eq_thetacsst}), we conclude that 
\begin{equation}\label{eq_envenew}
C^*(\losc{\cl S}) \cong \limcalg C^*(\cl S_k).
\end{equation}

By Theorem \ref{thm:C*envelopefinitegraph}, 
$C_e^*(\cl S_k) = C^*(\cl S_k)$, and hence (the restriction of) $\pi_k$ is a well-defined
*-monomorphism from $C_e^*(\cl S_k)$ into $C_e^*(\cl S_{k+1})$; we can thus form the inductive system 
$(\{C_e^*(\cl S_k)\}_{k\in \bb{N}}, \{\pi_k\}_{k\in \bb{N}})$.
Note that, by \cite[Theorem 3.2]{LuthraKumar}, 
\begin{equation}\label{eq_eeile}
C_e^*(\losc{\cl S}) = \limcalg C_e^*(\cl S_k); 
\end{equation}
we provide a direct argument for the equality (\ref{eq_eeile}) for the convenience of the reader. 
Namely, we show that $\limcalg C_e^*(\cl S_k)$ satisfies the universal property of the C*-envelope $C_e^*(\los{\cl S})$. 
Consider the following commuting diagram:
        \begin{equation*}\label{eq:c*alg generated by op system}
            \begin{CD}
                \cl S_1 @>{\phi_1}>> \cl S_2 @>{\phi_2}>> \cl S_3 @>{\phi_3}>> \cl S_4 @>{\phi_4}>> \cdots \\
                @V{\iota_1}VV @V{\iota_2}VV @V{\iota_3}VV @V{\iota_4}VV \\
                C_e^*(\cl S_1) @>{\pi_1}>> C_e^*(\cl S_2) @>{\pi_2}>> C_e^*(\cl S_3) @>{\pi_3}>> C_e^*(\cl S_4) @>{\pi_4}>> \cdots.
            \end{CD}
        \end{equation*}
        Note that we have denoted by $\iota_k$ the inclusion of $\cl S_k$ into $C^*_e(\cl S_k)$. 
        By Remark~\ref{rem:all maps coi}, 
        there exists a unital completely order isomorphic embedding 
        $\psi : \los{\cl S} \rightarrow \limos C_e^*(\cl S_k)$ such that $\psi \circ \phi_{k,\infty} = \pi_{k,\infty}\circ \iota_k$, $k\in \bb{N}$. 
       Observe that $\psi(\los{\cl S})$ generates $\limcalg C_e^*(\cl S_k)$; indeed, 
      each $a_k\in C_e^*(S_k)$
belongs to the span of elements of the form 
      $s_1  \cdots s_n$, where $s_i \in \cl S_k$, $1 \leq i \leq n$. 
      Thus, $\pi_{k,\infty}(a_k)$ belongs to the span of $\pi_{k,\infty}(s_1)  \cdots \pi_{k,\infty}(s_n)$. 
It follows that $(\limcalg C_e^*(\cl S_k), \psi)$ is a C*-cover of $\losc{\cl S}$.

        Suppose that 
        $(\cl B, \alpha)$ is a C*-cover of $\cl S_{\infty}$.
        It follows that $\alpha \circ \pi_{k,\infty} : \cl S_k \rightarrow \cl B$ is a unital complete isometry for all $k \in \bb{N}$. 
        Let $\cl B_k$ be the C*-subalgebra of $\cl B$ generated by $(\alpha \circ \pi_{k,\infty})(\cl S_k)$. 
        Since $\alpha(\cl S_{\infty})$ generates $\cl B$ and $\cup_{k \in \bb{N}}\pi_{k,\infty}(\cl S_k)$ generates $\cl S_{\infty}$, 
        we have that $\cl B = \overline{\cup_{k \in \bb{N}}\cl B_k}$. By the universal property of the C*-envelope, 
        for every $k \in \bb{N}$, 
        there exists a unique *-homomorphism $\rho_k : \cl B_k \rightarrow C^*(\cl S_k) $ such that 
        $\rho_k \circ \alpha \circ \phi_{k,\infty} =  \iota_k$. Therefore
        \[
            \pi_k \circ \rho_k \circ \alpha \circ \phi_{k,\infty}
            = \pi_k \circ \iota_k
            = \iota_{k+1} \circ \phi_k
            = \rho_{k+1} \circ \alpha \circ \phi_{k+1, \infty} \circ \phi_k
            = \rho_{k+1} \circ \alpha \circ \phi_{k, \infty},
        \]
        for all $k \in \bb{N}$. Thus, $\pi_k \circ \rho_{k} = \rho_{k+1}$, $k \in \bb{N}$. We may thus
        construct the following commuting diagram:
        \begin{equation*}\label{eq:c*alg generated by op system}
            \begin{CD}
              \cl B_1 @>{{\id}_{\cl B}}>> \cl B_2 @>{{\id}_{\cl B}}>> \cl B_3 @>{{\id}_{\cl B}}>> \cl B_4 @>{{\id}_{\cl B}}>> \cdots \\
                @V{\rho_1}VV @V{\rho_2}VV @V{\rho_3}VV @V{\rho_4}VV \\
              C_e^*(\cl S_1) @>{\pi_1}>> C_e^*(\cl S_2) @>{\pi_2}>> C_e^*(\cl S_3) @>{\pi_3}>> C_e^*(\cl S_4) @>{\pi_4}>> \cdots.
            \end{CD}
        \end{equation*}
        By Theorem~\ref{thm:universal property infinitely many maps}, 
        there exists a *-homomorphism 
        $\rho : \cl B \rightarrow \limcalg C^*_e(\cl S_k)$ such that $\rho = \pi_{k,\infty} \circ \rho_k$ for all $k \in \bb{N}$. Note that
        \[
        \rho \circ \alpha \circ \phi_{k,\infty} = \pi_{k,\infty} \circ \rho_k \circ \alpha \circ \phi_{k,\infty} = \pi_{k,\infty} \circ \iota_k = 
        \psi \circ \phi_{k,\infty},
        \]
        for all $k \in \bb{N}$. Therefore $\rho \circ \alpha = \psi$, and
        hence $\limcalg C_e^*(\cl S_k)$ is *-isomorphic to the C*-envelope of $\widehat{\cl S}_{\infty}$.
        It now follows from (\ref{eq_envenew}) and (\ref{eq_eeile}) that 
$C^*_e(\cl S_{\infty}) \cong C^*(\cl S_{\infty})$.
    \end{proof}

    \subsection{Graphs associated to operator subsystems of UHF algebras}\label{sec: infinite graphs}

The framework required to associate a graph with the UHF algebra $\lcalg{\cl A}$ is established in \cite{Power}. 
We give some of its details here, since it will be 
needed in order to define graphs associated with operator subsystems of $\lcalg{\cl A}$.
Recall that $X_{\infty} = \Delta(\lcalg {\cl D})$ and $X_k = \Delta(\cl D_k)$, $k\in \bb{N}$.
By Remark~\ref{rem:functionals on diagonal matrix algebra}, 
$X_{\infty}= \limtop X_k$. For each $k \in \mathbb{N}$ and each $1 \leq i \leq n_k$, we have that 
$e_{i,i}^k \in \cl D_k \subseteq \lcalg{\cl D}$. 
Let
    \[
    	X_i^k = \{x \in X_{\infty} : \angles{x}{e_{i,i}^k} = 1 \}.
    \]
    Clearly, $X_i^k$ is a closed and open subset of $X_{\infty}$ such that, for all $k \in \mathbb{N}$,
    \[
    	X_{\infty} = \bigcup_{i=1}^{n_k} X_i^k.
    \]
We note that, if $[l_k]$ denotes the set $\{0, 1, 2, \ldots, l_k-1\}$, the space 
$X_{\infty}$ is homeomorphic to the Cantor space $\Pi_{k=1}^{\infty} [l_k]$
(recall that $l_k = \frac{n_{k+1}}{n_k}$). 

    For each $k \in \mathbb{N}$ and each $1 \leq i,j \leq n_k$, 
    let $\phi_{i,j}^{k}: C(X_i^k) \rightarrow C(X_j^k)$ be the *-isomorphism given by 
    $\phi_{i,j}^{k}(d)= {e_{i,j}^k}^* d {e_{i,j}^k}$.
Let $\alpha_{i,j}^k : X_j^k \rightarrow X_i^k$ be 
the homeomorphism induced by $\phi_{i,j}^{k}$; thus, 
    \[
    	\angles{\alpha_{i,j}^k(x)}{d} = \angles{x}{\phi_{i,j}^k(d)}, \ \ \ x \in X_j^k, d \in C(X_i^k). 
    \]
For $k \in \mathbb{N}$ and $1 \leq i,j \leq n_k$, let 
    \[
    	E_{i,j}^k = \Big\lbrace(x,y) \in X_{\infty} \times X_{\infty} : x = \alpha_{i,j}^k(y) \text{~for some~} y \in X_j^k\Big \rbrace
    \]
    be the graph of the partial homeomorphism $\alpha_{i,j}^k$ of $X_{\infty}$.
    We have, equivalently,
    \[
        E_{i,j}^k = \Big\lbrace (x,y) \in X_{\infty} \times X_{\infty} : \angles{x}{d} = \angles{y}{e_{j,i}^k d e_{i,j}^k} \text{~for all~} d \in \cl D_k\Big\rbrace.
    \]
It will be convenient to write $R(e_{i,j}^k) = E_{i,j}^k$; for a subset $\cl E$ of canonical matrix units in $\lcalg{\cl A}$, we set 
$R(\cl E) = \cup_{e\in \cl E} R(e)$. 
In particular, 
    \begin{equation}\label{eq:R(A)}
    	R(\lcalg{\cl A}) = \bigcup \big\lbrace E_{i,j}^k : k\in \bb{N}, 1\leq i,j\leq n_k\big\rbrace.
    \end{equation}

In Remark~\ref{rem:facts about partial homeo1}, 
whose statement is drawn from \cite{Power},
we point out how the sets $E_{i,j}^k$ reflect the properties of the matrix units $e_{i,j}^k$. 
    We set ${E_{i,j}^k}^* = E_{j,i}^k$.
    For $E,F \subseteq X_{\infty} \times X_{\infty}$, let
    \begin{equation}\label{eq:circedges}
        E\circ F = \{(x,z) \in X_{\infty} \times X_{\infty} : \ \exists \ y\in X_{\infty}
        \mbox{ with } (x,y) \in E \mbox{ and } (y,z) \in F\}.
    \end{equation}

    \begin{remark}\label{rem:facts about partial homeo1}
    The following hold, for any $k,m \in \mathbb{N}$ and any $1 \leq i \leq n_k$, $1 \leq j \leq n_m$:

(i) \ \ $E_{i,i}^k = \{(x,x): x \in X_i^k\}$;

(ii) \ $ (x,y) \in E_{i,j}^k \text{~if and only if~} (y,x) \in E_{j,i}^k$;

(iii) $E_{i,j}^m \circ E_{j,k}^m = E_{i,k}^m \text{~and~} E_{i,j}^m \circ E_{k,l}^m = \emptyset \text{~when~} j \neq k$.
    \end{remark}

We have that $R(\lcalg{\cl A})$ is 
an equivalence relation on $X_{\infty} \times X_{\infty}$ and endows~$\lcalg{\cl A}$ with an associated graph.
    We define a topology on $R(\lcalg{\cl A})$ by specifying 
    $\{E_{i,j}^k : k\in \bb{N}, 1\leq i,j \leq n_k\}$ as a base of open sets. 
    Note that each $E_{m,n}^l$ is either disjoint from $E_{i,j}^k$ or is a subset of $E_{i,j}^k$ 
    (if the latter happens then $l > k$). 
    Thus, this base consists of closed and open sets. Since $X_{\infty}$ is compact, 
    the sets $E_{i,j}^k$ are compact, too. 

    If $\losc{\cl S}$ is an operator subsystem of $\lcalg{\cl A}$, set
    \begin{equation} \label{eq:R(S)}
    	R(\losc{\cl S}) = \bigcup \big\lbrace E_{i,j}^k : e_{i,j}^k \in \losc{\cl S}\big\rbrace.
    \end{equation}

We specialise to the case of operator systems the Spectral Theorem for Bimodules from \cite{Power}. 
    The following proposition follows from \cite[Proposition 7.3~and~Proposition~7.4]{Power}.

    \begin{proposition}\label{prop:local spectral theorem}
        Let $\losc{\cl S}$ and $\losc{\cl T}$ be concrete operator $\lcalg{\cl D}$-systems. 
        
        (i) \ We have that $E_{i,j}^k \subseteq R(\losc{\cl S})$ if and only if $e_{i,j}^k \in \losc{\cl S}$;

(ii)  If $R(\losc{\cl S}) = R(\losc{\cl T})$ then $\losc{\cl S} = \losc{\cl T}$.
    \end{proposition}




    \begin{proposition}\label{p_orss}
        Let $\losc{\cl S}$ be a concrete operator $\lcalg{\cl D}$-system. 
        Then $R(\losc{\cl S})$ is an open, reflexive and symmetric subset of $R(\lcalg{\cl A})$.
    \end{proposition}

    \begin{proof}
        We have that $R(\losc{\cl S})$ is 
      open since it is a union of open sets. 
      Since $\losc{\cl S}$ contains the identity operator, $R(\losc{\cl S})$ is reflexive. 
Suppose that $(x,y) \in R(\losc{\cl S})$. Then there exists $i,j,k$ such that $(x,y) \in E_{i,j}^k$ and $E_{i,j}^k\subseteq R(\losc{\cl S})$. 
By Proposition \ref{prop:local spectral theorem}, 
$e_{i,j}^k \in \losc{\cl S}$. Thus, $e_{j,i}^k = (e_{i,j}^k)^* \in \losc{\cl S}$ and, again by 
Proposition \ref{prop:local spectral theorem}, $E_{j,i}^k\subseteq R(\losc{\cl S})$. 
Thus, $(y,x)\in R(\losc{\cl S})$.
    \end{proof}

By Proposition \ref{p_orss}, we may view $R(\losc{\cl S})$ is a (closed and) open subgraph of $R(\lcalg{\cl A})$. 
Conversely, if 
$P\subseteq R(\lcalg{\cl A})$ is an open, symmetric and reflexive subset, let 
    \begin{equation} \label{eq:S(P)}
    	\cl S_{\infty}(P) = \overline{\SPAN} \{ d e_{i,j}^k f : d, f \in \lcalg{\cl D}, E_{i,j}^k \subseteq P\}.
    \end{equation}

    \begin{theorem}\label{thm:bijective correspondance between open symmetric reflexive subsets}
The map
$P\to \cl S_{\infty}(P)$ is a bijective correspondence between the open subgraphs of $R(\lcalg{\cl A})$ and 
the concrete operator $\lcalg{\cl D}$-systems.
    \end{theorem}

    \begin{proof}
The fact that, if $P$ is an open subgraph of $R(\lcalg{\cl A})$ then $\cl S_{\infty}(P)$ is a concrete operator $\lcalg{\cl D}$-system
follows easily from Remark~\ref{rem:facts about partial homeo1}.
It remains to show that for any open reflexive and symmetric subset $P$ of $R(\lcalg{\cl A})$, 
we have that $R(\cl S_{\infty}(P)) = P$.
It is clear that $P\subseteq R(\cl S_{\infty}(P))$. Conversely, suppose that $E_{i,j}^k \subseteq R(\cl S_{\infty}(P))$,
for some $i,j$ and $k$ with $i\neq j$. 
By Proposition \ref{prop:local spectral theorem}, $e_{i,j}^k\in \cl S_{\infty}(P)$. 
We claim that $E_{i,j}^k \subseteq P$; clearly, this claim will complete the proof. 

Let $\tilde{\cl A}_p$ be the $\lcalg{\cl D}$-bimodule, generated by $\cl A_p$, $p\in \bb{N}$. 
By \cite[Proposition 4.6]{Power}, there exists a $\lcalg{\cl D}$-bimodule 
surjective projection $\Phi_p : \lcalg{\cl A}\to \tilde{\cl A}_p$. 
Write 
$$\cl S_0 = \SPAN \{ d e_{s,t}^p  : d \in \lcalg{\cl D}, E_{s,t}^p \subseteq P, p\in \bb{N}\}.$$
We have that $e_{i,j}^k = \lim_{m\to\infty} x_m$, for some $x_m\in \cl S_0$, $m\in \bb{N}$; thus,
$$e_{i,j}^k = \lim_{m\to\infty} \Phi_k(x_m).$$
Let $E_k = \cup_{u,v} E_{u,v}^k$. 
Then $\Phi_k(x_m) = \sum_{E_{s,t}^p\subseteq P \cap E^k} d_{s,t}^{p,m} e_{s,t}^k$, 
for some $d_{s,t}^{p,m}\in \lcalg{\cl D}$ 
with $\supp(d_{s,t}^{p,m}) \subseteq X_s^p$.
It follows that 
\begin{equation}\label{eq_exdij}
e_{i,j}^k = \lim_{m\to\infty} \sum_{E_{s,t}^p\subseteq P \cap E^k_{i,j}} d_{s,t}^{p,m} e_{i,j}^k.
\end{equation}
Assume, by way of contradiction, that 
$$Y \stackrel{def}{=} \cup \{E_{s,t}^p : E_{s,t}^p\subseteq P\cap E_{i,j}^k\} \neq E_{i,j}^k.$$
Letting $a\in \lcalg{\cl D}$ be the projection corresponding to $Y$, we have that $a < e_{i,i}^k$
and $e_{i,j}^k = a e_{i,j}^k$, a contradiction. 
It follows that $Y = E_{i,j}^k$; since $P$ is open, $E_{i,j}^k\subseteq P$. 
    \end{proof}

Theorem \ref{thm:bijective correspondance between open symmetric reflexive subsets} allows us
to view the concrete operator $\lcalg{\cl D}$-systems as graph operator systems; 
we formalise this in the following definition.

    \begin{definition}\label{def:infinite graph operator system}
    Let $\lcalg{\cl A}$ be a UHF algebra with canonical masa 
    $\lcalg{\cl D}$.
    An open, reflexive and symmetric subset of $R(\lcalg{\cl A})$ will be called a \emph{Cantor graph}.
   
    If $P$ is a Cantor graph, the operator system 
    $\cl S_{\infty}(P)$ defined in (\ref{eq:S(P)}) will be called the 
    \emph{Cantor graph operator system} of $P$.
  \end{definition}

\subsection{A graph isomorphism theorem}

    In this section, we prove a version of Theorem~\ref{thm:PaulsenOrtiz} for 
    Cantor graph operator systems. 
Let $\lcalg{\cl A}$ and $\lcalg{\cl B}$ be UHF algebras with canonical masas $\lcalg{\cl D}$ and $\lcalg{\cl E}$, respectively,
and let $X_{\infty}=\Delta(\lcalg{\cl D})$ and $Y_{\infty}=\Delta(\lcalg{\cl E})$. 
We write $e_{i,j}^k$ and $E_{i,j}^k$ (resp. $f_{i,j}^k$ and $F_{i,j}^k$) for the canonical matrix units of $\lcalg{\cl A}$ (resp. $\lcalg{\cl B}$)
and their partial graphs.

    Using the notation introduced in (\ref{eq:circedges}), for a set $P\subseteq R(\lcalg{\cl A})$, let
$$            \epsilon(P) = \bigcup\{ E_1 \circ \cdots \circ E_n : 
            n \in \mathbb{N} \mbox{ and for each } j,
            E_j \subseteq P$$
$$ \mbox{ and } E_j = E_{s,t}^p \mbox { for some } s,t,p\}.
        $$

        \begin{lemma}\label{rem:graph produce of operator system is graph of generated c*alg}
Let $\losc{\cl S}$ be a concrete operator $\lcalg{\cl D}$-subsystem of $\lcalg{\cl A}$. 
Then $\epsilon(R(\losc{\cl S})) = R(C^*(\losc{\cl S}))$. 
        \end{lemma}

\begin{proof}
Write $P = R(\losc{\cl S})$ and $Q = \epsilon(P)$; it is clear that $Q$ is the smallest open equivalence relation 
containing $P$. 
Note that $C^*(\cl S_{\infty}) = \cl S_{\infty}(Q)$; indeed, 
every canonical matrix unit in $\cl S_{\infty}$ belongs to $\cl S_{\infty}(Q)$ and, since $\cl S_{\infty}(Q)$
is a C*-algebra, 
$C^*(\cl S_{\infty}) \subseteq \cl S_{\infty}(Q)$. 
Suppose that $e_{i,j}^k\in \cl S_{\infty}(Q)$. 
By Theorem 
\ref{thm:bijective correspondance between open symmetric reflexive subsets},
$E_{i,j}^k\subseteq Q$;
by compactness, $E_{i,j}^k$ is equal to a finite disjoint 
union of sets of the form $E_1 \circ \cdots \circ E_n$, where, for each $j$, the set $E_j $ is a graph of a 
canonical partial homeomorphism contained in $P$.  
            Thus, 
$e_{i,j}^k$ is equal to the sum of 
elements of the form $e_{i_1, j_1}^{k_1} \cdots e_{i_n, j_n}^{k_n}$, 
where $e_{i_r, j_r}^{k_r} \in \los{\cl S} \subseteq \losc{\cl S}$. 
It follows that $\cl S_{\infty}(Q) \subseteq C^*(\cl S_{\infty})$, and hence 
we have that $C^*(\cl S_{\infty}) = \cl S_{\infty}(Q)$. 
By Theorem \ref{thm:bijective correspondance between open symmetric reflexive subsets},
$Q = R(C^*(\losc{\cl S}))$.
\end{proof}

    \begin{theorem}\label{thm:graphs homeo then op system iso}
        Let $\lcalg{\cl A}$ and $\lcalg{\cl B}$ be UHF-algebras with canonical masas $\lcalg{\cl D}$ and $\lcalg{\cl E}$, respectively. 
        Set $X_{\infty} = \Delta(\lcalg{\cl D})$ and $Y_{\infty} = \Delta(\lcalg{\cl E})$. 
        Let $P\subseteq X_{\infty} \times X_{\infty}$ and $Q\subseteq Y_{\infty} \times Y_{\infty}$ 
        be Cantor graphs. 
        The following are equivalent:
        
        (i) \ 
        there exists a  homeomorphism 
        $\varphi: X_{\infty} \rightarrow Y_{\infty}$ such that $(\varphi \times \varphi) (P) = Q$;
        
        (ii) there exists a unital complete order isomorphism 
        $\phi : \cl S_{\infty}(P) \rightarrow \cl S_{\infty}(Q)$ such that $\phi(\lcalg{\cl D}) = \lcalg{\cl E}$.
     \end{theorem}
       
\begin{proof}
 Set $\losc{\cl S} = \cl S_{\infty}(P)$ (resp. $\losc{\cl T} = \cl S_{\infty}(Q)$); then $\losc{\cl S}$ 
 is a concrete operator $\lcalg{\cl D}$-system 
        (resp. a concrete operator $\lcalg{\cl E}$-system).

   (i)$\Rightarrow$(ii)   
   For ease of notation, set $\varphi^{(2)} = \varphi \times \varphi$. 
   Let $\widetilde{P} = \epsilon(P)$ and $\widetilde{Q} = \epsilon(Q)$.    
As in the proof of \cite[Proposition~7.5]{Power},
$\varphi^{(2)}$ is a homeomorphism from $\widetilde{P}$ onto $\widetilde{Q}$. 
By Lemma \ref{rem:graph produce of operator system is graph of generated c*alg},
$\widetilde{P} = R(C^*(\losc{\cl S}))$ and $\widetilde{Q} = R(C^*(\losc{\cl T}))$. 
Since $C^*(\losc{\cl S})$ (resp. $C^*(\losc{\cl T})$) is an AF C*-algebra with a canonical masa 
$\lcalg{\cl D}$ (resp. $\lcalg{\cl E}$), by \cite[Theorem 7.5]{Power},
there exists a *-isomorphism $\psi : C^*(\losc{\cl S})\to C^*(\losc{\cl T})$ such that 
$\psi(\lcalg{\cl D}) = \lcalg{\cl E}$. 
We have that the restriction $\phi$ of $\psi$ to $\losc{\cl S}$ has its range in $\losc{\cl T}$. 
By symmetry, $\phi$ is a bijection, and hence a unital complete order isomorphism.

(ii)$\Rightarrow$(i)   
  By Remark~\ref{cor:ucoi induces *iso in envelopes}, there exists a *-isomorphism 
  $\rho: C_e^*(\losc{\cl S}) \rightarrow C_e^*(\losc{\cl T})$ which extends $\phi.$ 
  By Theorem~\ref{thm:c*algebra generated by uhf algebra is c*envelope}, $\rho: C^*(\losc{\cl S}) \rightarrow C^*(\losc{\cl T})$ is a unital 
  *-isomorphism. Since $C^*(\losc{\cl S})$ and $C^*(\losc{\cl T})$ are subalgebras of $\lcalg{\cl A}$ and $\lcalg{\cl B}$, respectively, 
using \cite[Theorem 7.5]{Power}
  we obtain 
 a homeomorphism $\varphi: X_{\infty} \rightarrow Y_{\infty}$ such that,
  if $\varphi^{(2)} = \varphi \times \varphi$ then the map
        \[
        \varphi^{(2)}:R(C^*(\losc{\cl S})) \rightarrow R(C^*(\losc{\cl T}))
        \]
        is a homeomorphism
        and $R(\rho(e_{i,j}^k))= \varphi^{(2)}(R(e_{i,j}^k))$ for any $e_{i,j}^k \in C^*(\losc{\cl S})$.
        Suppose that 
$E_{i,j}^k \subseteq P$. By Proposition \ref{prop:local spectral theorem}, $e_{i,j}^k\in \losc{\cl S}$. 
Since $\phi$ is a (complete) isometry, \cite[Proposition 7.1]{Power}, along with the compactness of $Y_{\infty}$ shows that 
$\phi(e_{i,j}^k)$ is a sum of canonical matrix units. 
Moreover, by Theorem \ref{thm:bijective correspondance between open symmetric reflexive subsets}, 
$R(\phi(e_{i,j}^k))\subseteq R(\losc{\cl T}) = Q$. Thus, $\varphi^{(2)}(P) \subseteq Q$; 
by symmetry,
        $\varphi^{(2)}(P) = Q$.   
    \end{proof}

We point out that the condition $\phi(\lcalg{\cl D}) = \lcalg{\cl E}$ appearing in 
Theorem \ref{thm:graphs homeo then op system iso} (ii) is rather natural; indeed, since
the algebra $\lcalg{\cl D}$ uniquely determines $X_{\infty}$, this
condition can be thought of as 
the requirement that the map $\psi$ respect the \lq\lq vertex sets'' in the corresponding operator systems 
in order to give rise to a bona fide Cantor graph isomorphism.

    \subsection{A generalisation of Glimm's theorem}\label{sec:glimms theorem generalisation}

We conclude this section with a generalised version of Glimm's theorem (see~\cite{Glimm}).

    \begin{theorem} \label{thm:generalisation Glimms theorem}
        Let $\lcalg{\cl A}$ and $\lcalg{\cl B}$ be UHF algebras with canonical masas $\lcalg{\cl D}$ and $\lcalg{\cl E}$, respectively. 
        Let $\losc{\cl S}$ be a concrete operator $\lcalg{\cl D}$-system and $\losc{\cl T}$ be a concrete operator $\lcalg{\cl E}$-system. 
        The following are equivalent:
        \begin{enumerate}[\rm(i)]
        \item there exists a unital complete order isomorphism $\phi:\losc{\cl S} \rightarrow \losc{\cl T}$ such that 
        $\phi(\lcalg{\cl D}) = \lcalg{\cl E}$;
        \item there exist subsequences $(\cl S_{m_k})_{k \in \mathbb{N}}$ and $(\cl T_{n_k})_{k \in \mathbb{N}}$ 
        of the sequences in the inductive systems associated with $\losc{\cl S}$ and $\losc{\cl T}$, respectively, 
        and unital completely positive maps $\{\phi_k\}_{k \in \bb{N}}$ and $\{\psi_k\}_{k \in \bb{N}}$ such that
        \begin{enumerate}[\rm(a)]
          \item the diagram 
        \begin{equation*}
          \vspace{-2mm}
            \begin{tikzcd}
                \cl S_1 \arrow{r} \arrow{d}{\phi_1} & \cl S_{m_1} \arrow{d}{\phi_2}\arrow{r}&  \cl S_{m_2} \arrow{d}{\phi_3} \arrow{r}& \cdots \\
                \cl T_{n_1} \arrow{r} \arrow[swap]{ur}{\psi_1}      & \cl T_{n_2} \arrow{r} \arrow[swap]{ur}{\psi_2}         &  \cl T_{n_3} \arrow[swap]{ur}{\psi_3} \arrow{r}          & \cdots, \\
            \end{tikzcd}
            \vspace{-10mm}
        \end{equation*}
        commutes, and
        \item $\phi_{k+1}(\cl D_{m_k}) \subseteq \cl E_{n_{k+1}}$ and $\psi_k(\cl E_{n_k}) \subseteq \cl D_{m_k}$, for all $k \in \bb{N}$.
      \end{enumerate}
      \end{enumerate}
    \end{theorem}
    
    \begin{proof}
    (ii)$\Rightarrow$(i)    
    By Remark~\ref{rem:theta coi_os}, 
    there exists a unital complete order isomorphism 
    $\phi: \limos \cl S_k \rightarrow \limos \cl T_k$; let $\psi: \limos \cl T_k \rightarrow \limos \cl S_k$ be its inverse. 
    Let  
    $\widehat{\phi} : \losc{\cl S} \rightarrow \losc{\cl T}$ 
    (resp. $\widehat{\psi} : \losc{\cl T} \rightarrow \losc{\cl S}$) be the (unital completely positive)
    extension of $\phi$ (resp. $\psi$). Clearly, $\widehat{\phi}$ and $\widehat{\psi}$ are each other's inverses, and thus
    $\losc{\cl S}$ and $\losc{\cl T}$ are unitally completely order isomorphic. 
    Furthermore, condition (b) implies that 
    $\widehat{\phi}(\lcalg{\cl D}) = \lcalg{\cl E}$.

(i)$\Rightarrow$(ii)    
        Suppose that $\phi:\losc{\cl S} \rightarrow \losc{\cl T}$ is a unital complete order isomorphism such that $\phi(\lcalg{\cl D}) = \lcalg{\cl E}$.     
                By Remark~\ref{cor:ucoi induces *iso in envelopes}, 
                there exists a *-isomorphism $\phi: C_e^*(\losc{\cl S}) \rightarrow C_e^*(\losc{\cl T})$ extending $\phi.$
        By Theorem~\ref{thm:c*algebra generated by uhf algebra is c*envelope}, $\phi: C^*(\losc{\cl S}) \rightarrow C^*(\losc{\cl T})$ is a unital 
        *-isomorphism.
        By \cite[Theorem~7.5]{Power}, 
        there exists a homeomorphism $\alpha : X_{\infty} \rightarrow Y_{\infty}$ such that 
        $\alpha^{(2)} : R(C^*(\losc{\cl S})) \rightarrow R(C^*(\losc{\cl T}))$ is a homeomorphism and 
        $\alpha^{(2)}(E_{i,j}^k) = R(\phi(e_{i,j}^k))$.
By Theorem \ref{thm:graphs homeo then op system iso} and its proof, 
\begin{equation}\label{eq_strr}
\alpha^{(2)}(R(\losc{\cl S})) = R(\losc{\cl T}).
\end{equation}

Set $\cl L_k = C^*(\cl S_{k})$ and $\cl M_k = C^*(\cl T_{k})$, $k\in \bb{N}$. 
By (\ref{eq_eeile}) and \cite[Theorem~5.3]{Power} and its proof, 
there exist inductive systems of finite dimensional C*-algebras and corresponding unital *-homomorphisms 
such that the following diagram commutes:
   \begin{equation*}
          \vspace{-2mm}
            \begin{tikzcd}
                \cl L_1 \arrow{r} \arrow{d}{\phi_1} & \cl L_{m_1} \arrow{d}{\phi_2}\arrow{r}&  \cl L_{m_2} \arrow{d}{\phi_3} \arrow{r}& \cdots \\
                \cl M_{n_1} \arrow{r} \arrow[swap]{ur}{\psi_1}      & \cl M_{n_2} \arrow{r} \arrow[swap]{ur}{\psi_2}         &  \cl M_{n_3} \arrow[swap]{ur}{\psi_3} \arrow{r}          & \cdots. \\
            \end{tikzcd}
            \vspace{-10mm}
        \end{equation*}
        
         \bigskip
        
        \smallskip
        
The compactness of $Y_{\infty}$ and \cite[Proposition 7.1]{Power} show that the element $\phi_k(e_{i,j}^{m_k})$ is a sum of 
canonical matrix units. 
By passing to further 
subsequences if necessary, we may therefore assume that 
$\phi_k(\cl S_{m_k})\subseteq \cl T_{n_{k+1}}$, 
$\psi_k(\cl T_{m_k})\subseteq \cl S_{n_{k}}$, 
$\phi_k(\cl D_{m_k})\subseteq \cl E_{n_{k+1}}$ and 
$\psi_k(\cl E_{n_k})\subseteq \cl D_{m_{k}}$, 
for each $k$. 
Thus, conditions (a) and (b) are fulfilled. 
    \end{proof}

\end{document}